\documentclass[twoside,11pt]{article}

\usepackage{graphicx}
\usepackage{caption}
\usepackage{amsmath,amssymb,enumerate}
\usepackage{booktabs}
\usepackage[usenames,dvipsnames,svgnames,table]{xcolor}
\usepackage[caption=false,font=footnotesize]{subfig}
\usepackage{arydshln}
\usepackage{tabu}

\newtheorem{problem}{Problem}
\newtheorem{assumption}{Assumption}

\DeclareMathOperator{\diff}{Diff}
\DeclareMathOperator{\tr}{Tr}
\DeclareMathOperator{\Lie}{Lie}
\newcommand{\se}{\mathfrak{se}}
\newcommand{\so}{\mathfrak{so}}

\newcommand{\sumxz}{\sum_{\substack{x_i\in X\\ z_j\in Z}}}

\makeatletter
\renewcommand*\env@matrix[1][c]{\hskip -\arraycolsep
  \let\@ifnextchar\new@ifnextchar
  \array{*\c@MaxMatrixCols #1}}
\makeatother

\newcommand{\diag}{\mathop{\mathrm{diag}}}

\newcommand{\Hz}{\mathop{\mathrm{Hz}}}

\newcommand{\Fcal}{\mathcal{F}}
\newcommand{\Gcal}{\mathcal{G}}
\newcommand{\Hcal}{\mathcal{H}}
\newcommand{\Ical}{\mathcal{I}}
\newcommand{\Xcal}{\mathcal{X}}

\newcommand{\transpose}{\mathsf{T}}
\newcommand{\SO}{\mathrm{SO}}
\newcommand{\SE}{\mathrm{SE}}
\newcommand{\GL}{\mathrm{GL}}

\newcommand{\squeezeup}{\vspace{-2mm}}

\usepackage{lastpage}

\usepackage{jmlr2e}

\firstpageno{1}

\begin{document}

\jmlrheading{22}{2021}{1-\pageref{LastPage}}{12/20; Revised
8/21}{10/21}{20-1468}{William Clark, Maani Ghaffari, and Anthony Bloch}
\ShortHeadings{Nonparametric Continuous Sensor Registration}{Clark,  Ghaffari, and Bloch}

\title{Nonparametric Continuous Sensor Registration}

\author{\name William Clark \email wac76@cornell.edu  \\
       \addr Department of Mathematics\\
       Cornell University\\
       Ithaca, NY 14850, USA
       \AND
       \name Maani Ghaffari \email maanigj@umich.edu \\
       \addr Department of Naval Architecture and Marine Engineering\\
       University of Michigan\\
       Ann Arbor, MI 48109 USA
       \AND
       \name Anthony Bloch \email abloch@umich.edu \\
       \addr Department of Mathematics\\
       University of Michigan\\
       Ann Arbor, MI 48109 USA}

\editor{Sayan Mukherjee}

\maketitle

\begin{abstract}%
This paper develops a new mathematical framework that enables nonparametric joint semantic and geometric representation of continuous functions using data. The joint embedding is modeled by representing the processes in a reproducing kernel Hilbert space. The functions can be defined on arbitrary smooth manifolds where the action of a Lie group aligns them. The continuous functions allow the registration to be independent of a specific signal resolution. The framework is fully analytical with a closed-form derivation of the Riemannian gradient and Hessian. We study a more specialized but widely used case where the Lie group acts on functions isometrically. We solve the problem by maximizing the inner product between two functions defined over data, while the continuous action of the rigid body motion Lie group is captured through the integration of the flow in the corresponding Lie algebra. Low-dimensional cases are derived with numerical examples to show the generality of the proposed framework. The high-dimensional derivation for the special Euclidean group acting on the Euclidean space showcases the point cloud registration and bird's-eye view map registration abilities. An implementation of this framework for RGB-D cameras outperforms the state-of-the-art robust visual odometry and performs well in texture and structure-scarce environments.
\end{abstract}%



\section{Introduction}
We consider the problem of sensor registration, which is defined as finding the transformation between two sensors with overlapping measurements or tracking a moving sensor using its sequential measurements. This problem frequently arises in engineering domains such as point cloud registration~\citep{chen1991plane,censi2008icl,Segal-RSS-09,stoyanov2012fast,servos2017multi,sparkison-2018a,Zaganidis2018,parkison2019boosting}, visual odometry and SLAM~\citep{nister2006visual,scaramuzza2011visual,fraundorfer2012visual,strasdat2012local,kerl2013dense,engel2015lsdslam,MGhaffari-RSS-19,rosinol2019kimera}, satellite image registration~\citep{kim2003automatic,bentoutou2005automatic}, and photogrammetry~\citep{mikhail2001introduction}. In terms of applications, the above mentioned problems are fundamental to many computer vision~\citep{hartley2003multiple,szeliski2010computer} and robotics~\citep{thrun2005probabilistic,barfoot2017state} algorithms.

In this paper, we formulate a fundamentally novel nonparametric framework for sensor registration that constructs continuous functions from raw measurements; hence our approach can be considered as a direct method. The functions can be thought of as curves, surfaces, or hyper-surfaces that live in a Reproducing Kernel Hilbert Space (RKHS). For example, given an RGB-D image, a function can be constructed that maps a three-dimensional (3D) point to the RGB color space. Therefore, such functions inherently represent a joint model for semantic/appearance and geometry. Furthermore, the domain of the functions can be an arbitrary smooth manifold where the action of a Lie group is used to align them. This concept is illustrated in Figure~\ref{fig:ncsr}. The continuous functions allow the registration to be independent of a specific signal resolution and the framework is fully analytical with a closed-form derivation of the Riemannian gradient and Hessian.

\begin{figure}[t]
    \centering
    \includegraphics[width=0.6\columnwidth]{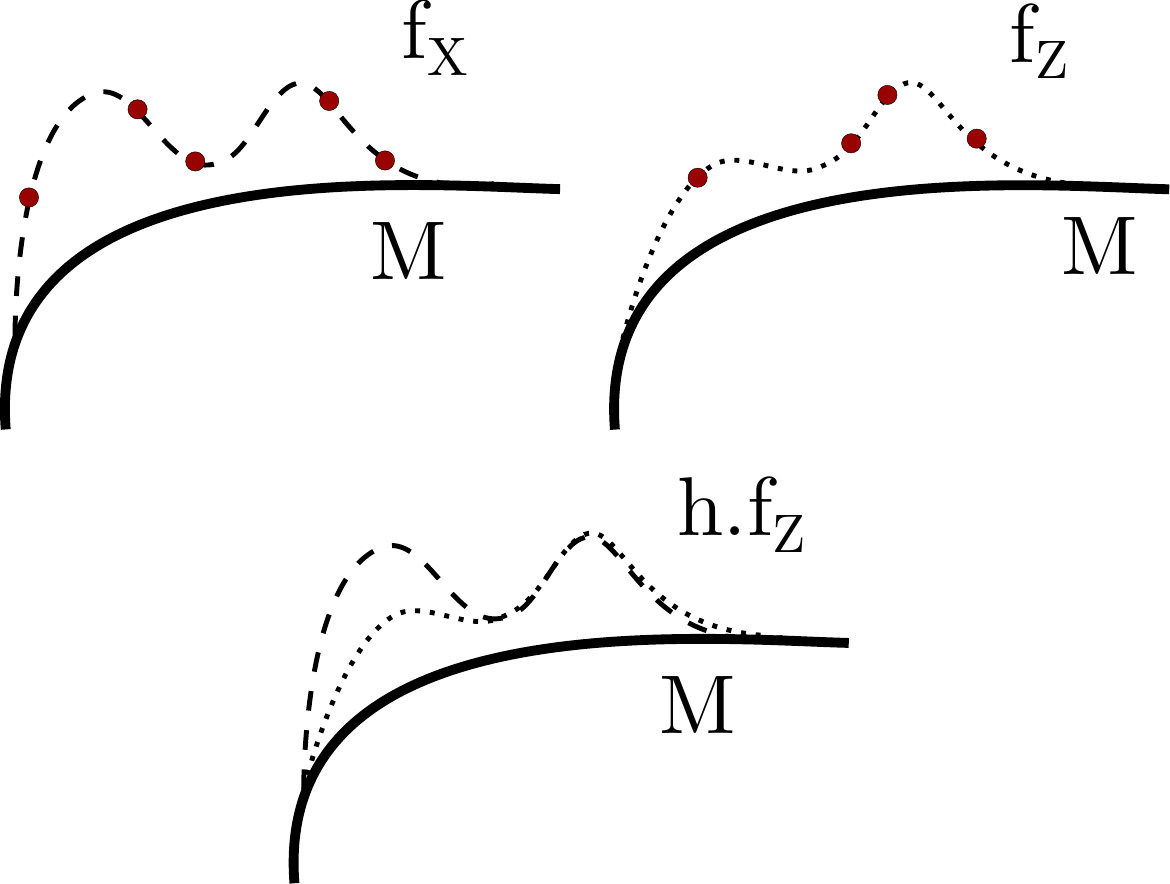}
    \caption{Conceptual illustration of nonparametric continuous sensor registration. Continuous functions are constructed from measurements shown in red circles. The domain of the functions can be an arbitrary smooth manifold, $M$, where the action of a Lie group, $\mathcal{G}$, is used to align them via $h.f_Z$ and $h\in\mathcal{G}$.}
    \label{fig:ncsr}
\end{figure}

In the present work, we study a restriction of the framework where the Lie group acts on functions isometrically. This restriction has a variety of applications and its high-dimensional derivation for the special Euclidean group acting on the Euclidean space showcases the point cloud registration and bird’s-eye view map registration abilities. In particular, this work has the following contributions:
\begin{enumerate}
    \item a new mathematical framework for sensor registration that enables nonparametric joint semantic/appearance and geometric representation of continuous functions using data;
    \item low-dimensional cases are derived with numerical examples to show the generality of the proposed framework;
    \item a high-dimensional derivation for the special Euclidean group acting on the Euclidean space showcases the point cloud registration and bird's-eye view map registration abilities. A specific derivation and implementation of this framework for RGB-D cameras performs well in texture and structure-scarce environments.
    \item the open-source implementation of the derived cases in this work is available for download at\\ \href{https://github.com/MaaniGhaffari/c-sensor-registration}{\url{https://github.com/MaaniGhaffari/c-sensor-registration}}\\
    \href{https://github.com/MaaniGhaffari/cvo-rgbd}{\url{https://github.com/MaaniGhaffari/cvo-rgbd}}
\end{enumerate}

An early idea of this work in the context of RGB-D visual odometry was presented at the 2019 Robotics: Science and Systems conference in Freiburg, Germany~\citep{MGhaffari-RSS-19}. This work generalizes the previous framework and places it into a broader context. In particular, an abstract sensor registration problem is defined in Section~\ref{sec:problem}, where low-dimensional and high-dimensional examples together with their second-order geometry on flat or curved spaces are studied. The generalization to curved spaces also connects this work to an ongoing research direction regarding positive-definite kernel design and hyperparameter learning on curved spaces~\citep{jayasumana2015kernel,feragen2015geodesic,borovitskiy2020mat}. Furthermore, the derivation of the Riemannian Hessian has important implications in quantifying the pose estimation covariance~\citep{censi2007accurate,bonnabel2016covariance,landry2019cello,brossard2020new} and symmetry of the problem by studying vector fields generated by the eigenvalues of Hessian~\citep{anonymous2021correspondencefree}, and therefore is included for completeness. However, its applications in those research areas and second-order solvers are left for future studies. Finally, the new experiments on RGB-D visual odometry in Section~\ref{sec:results} are done via an open-source and efficient C++ implementation and compared against DVO~\citep{kerl2013dense,kerl2013robust}.

The remainder of this paper is organized as follows. A review of the required mathematical machinery and notation is given in Appendix~\ref{sec:prelim}. Section~\ref{sec:problem} presents the novel continuous sensor registration framework. Section~\ref{sec:lowdim} provides low-dimensional examples of the framework including point cloud registration on circles and torus. We discuss some theoretical limitations of the kernel design on compact Riemannian manifolds in Section~\ref{sec:compact}. The high-dimensional example for the special Euclidean group is given in Section~\ref{sec:euclidean}. A brief theoretical analysis for the verification of the idea is provided in Section~\ref{sec:verification}. The integration of the flow for the special Euclidean group to obtain the solution is explained in Section~\ref{sec:senintegration}. Section~\ref{sec:RgeomSEn} deals with the Riemannian geometry of the special Euclidean groups. Experimental evaluations of the proposed method for registration and tracking using RGB-D images are presented in Section~\ref{sec:results}. Section~\ref{sec:discussion} provides discussions regarding several important topics related to the proposed approach.
Finally, Section~\ref{sec:conclusion} concludes the paper and provides future work ideas.

\section{Problem Setup}
\label{sec:problem}

Let $M$ be a smooth manifold and consider two (finite) collections of points, $X=\{x_i\}$, $Z=\{z_j\}\subset M$. Also, suppose we have a (Lie) group, $\Gcal$, acting on $M$. We want to determine which element $h\in \Gcal$ aligns ``best'' the two point clouds $X$ and $hZ = \{hz_j\}$. To assist with this, we will assume that each point contains information described by a point in an inner product space, $(\mathcal{I},\langle\cdot,\cdot\rangle_{\mathcal{I}})$. To this end, we will introduce two labeling functions, $\ell_X:X\to\Ical$ and $\ell_Z:Z\to\Ical$.

In order to measure their alignment, we will be turning the clouds, $X$ and $Z$, into functions $f_X,f_Z:M\to\Ical$ that live in some reproducing kernel Hilbert space, $(\Hcal,\langle\cdot,\cdot\rangle_{\mathcal{H}})$.
\begin{remark}
\label{rmk:hinv}
	The action, $\Gcal\curvearrowright M$ induces an action $\Gcal\curvearrowright C^\infty(M)$ by
	$$h.f(x) := f(h^{-1}x).$$
	Inspired by this observation, we will set $h.f_Z := f_{h^{-1}Z}$.
\end{remark}

In the following, we give a summary of the recipe for cloud alignment as well as formal sensor registration problems. 
\begin{definition}
    A sensor registration problem is given by a 5-tuple: $(\Gcal,M,\varphi,\langle\cdot,\cdot\rangle_{\mathfrak{g}},k)$ where
    \begin{enumerate}
        \item[(G1)] $\Gcal$ is a Lie group,
        \item[(G2)] $M$ is a smooth manifold,
        \item[(G3)] $\varphi:\Gcal\to\mathrm{Diff}(M)$ is a smooth group action,
        \item[(G4)] $\langle\cdot,\cdot\rangle_{\mathfrak{g}}$ is an inner product on $\mathfrak{g}=\mathrm{Lie}(\Gcal)$, and
        \item[(G5)] $k:M\times M\to\mathbb{R}$ is a symmetric positive definite function, called the kernel,
    \end{enumerate}
    while the information is given by a 3-tuple:
    $(\mathcal{I},(X,\ell_X),(Z,\ell_Z))$ where
    \begin{enumerate}
        \item[(I1)] $\mathcal{I}$ is an inner product space, called the information space,
        \item[(I2)] $X\subset M$ is a finite collection of points called the target and $\ell_X:X\to\mathcal{I}$ is its label, and
        \item[(I3)] $Z\subset M$ is a finite collection of points called the source and $\ell_Z:Z\to\mathcal{I}$ is its label.
    \end{enumerate}
\end{definition}
The information (G1)-(G5) is required to build the general form of the gradient while the information (I1)-(I3) encodes the actual clouds which is subsequently plugged into the gradient. The remainder of this work consists of understanding cases with varying (G1)-(G5) information. We first define the general form of the problem as follows.

\begin{problem}\label{prob:problem1}
    Suppose $f_X$ and $f_Z$ are the constructed functions over point clouds $X$ and $Z$, respectively. The problem of aligning two point clouds can be formulated as minimizing the distance between $f_X$ and $h.f_Z$ in the sense of the RKHS norm. That is, we want to solve
	\begin{equation}\label{eq:max1}
		\arg\min_{h\in \Gcal} \, J(h),\quad J(h):= \lVert f_X - h.f_Z \rVert^2_{\Hcal},
	\end{equation}
	where $\mathcal{H}$ is the RKHS induced by the kernel $k$ from (G5).
\end{problem}

In this work, we restrict the problem to a special case. If $\Gcal\curvearrowright C^\infty(M)$ is an isometry, we have $\lVert f_X - h.f_Z \rVert^2_{\Hcal} = \lVert f_X \rVert^2_{\Hcal} + \lVert f_Z \rVert^2_{\Hcal} - 2 \langle f_X, h.f_Z\rangle_{\mathcal{H}}$. Then we can reformulate Problem~\ref{prob:problem1} into a reduced form as follows.
\begin{problem}\label{prob:problem2}
	The problem of aligning the point clouds can now be rephrased as maximizing the scalar products of $f_X$ and $h.f_Z$, i.e., we want to solve
	\begin{equation}\label{eq:max2}
		\arg\max_{h\in \Gcal} \, F(h),\quad F(h):= \langle f_X, h.f_Z\rangle_{\mathcal{H}}.
	\end{equation}
\end{problem}
\subsection{Constructing the Functions}
We first choose a symmetric function $k:M\times M\to \mathbb{R}$ to be the kernel of our RKHS, $\Hcal$. This allows us to turn the point clouds to functions via
\begin{equation}
\begin{split}
	f_X(\cdot) &:= \sum_{x_i\in X} \, \ell_X(x_i) k(\cdot,x_i), \\
	f_Z(\cdot) &:= \sum_{z_j\in Z} \, \ell_Z(z_j) k(\cdot,z_j).
\end{split}
\end{equation}
We can now define the inner product of $f_X$ and $f_Z$ by
\begin{equation}\label{eq:scalar}
\langle f_X,f_Z\rangle_{\Hcal} := \sum_{\substack{x_i\in X, z_j\in Z}} \, \langle \ell_X(x_i),\ell_Z(z_j)\rangle_{\mathcal{I}} \cdot k(x_i,z_j).
\end{equation}

We note that it is possible to use the well-known kernel trick in machine learning~\citep{bishop2006pattern,rasmussen2006gaussian,murphy2012machine} to substitute the inner products in~\eqref{eq:scalar} with the appearance/semantic kernel, i.e., $k_c:\mathcal{I} \times \mathcal{I} \to \mathbb{R}$. The kernel trick can be applied to carry out computations implicitly in the high dimensional space, which leads to computational savings when the dimensionality of the feature space is large compared to the number of data points~\citep{rasmussen2006gaussian}. After applying the kernel trick to~\eqref{eq:scalar}, we get
\begin{align}
\label{eq:newscalar}
    \langle f_X,f_Z\rangle_{\Hcal} =& \sum_{\substack{x_i\in X, z_j\in Z}} \,  k_c(\ell_X(x_i),\ell_Z(z_j)) \cdot k(x_i,z_j)
    := \sum_{\substack{x_i\in X, z_j\in Z}} \,  c_{ij} \cdot k(x_i,z_j).
\end{align}
\begin{remark}
	We note two advantages of measuring the alignment of $X$ and $Z$ by \eqref{eq:scalar}. The first is that we do not need identification of which point of $X$ should be paired with what point of $Z$. The second is that the number of points in $X$ does not even need to be equal to the number of points in~$Z$!
\end{remark}
\begin{remark}
    Although the double sum in~\eqref{eq:scalar} seems to make the problem computationally intractable when dealing with a large number of data points, the inner product in~\eqref{eq:scalar} is sparse. The sparse structure is due to the kernelized formulation and the fact that the kernel bandwidth is spatially limited. As such, most kernel terms are zero. See Fig. 1 in~\citet{MGhaffari-RSS-19}. This notion of sparsity extends to labels, i.e., semantic information, once we kernelize the information inner product as done in~\eqref{eq:newscalar}. Furthermore, to exploit this sparse structure is the motivation for the development of a custom solver. However, the specific solver used in this work is not tied to the proposed framework.
\end{remark}
\subsection{Building the Gradient Flow}
In order to (at least locally) solve~\eqref{eq:max2}, we will construct a gradient flow: $\dot{h}=\nabla F(h)$ (this is similar to the treatment in \cite{bloch1992} where a gradient flow is used to maximize the inner product between elements of a Lie algebra corresponding to a compact Lie group). Before we proceed, we will first determine the differential, $dF$. In order to do this, we will need the notion of an infinitesimal generator for a group action (see chapter 4 of \citealt{berndt2001introduction}). 
\begin{definition}\label{def:infgen}
	Suppose that a Lie Group $\Gcal$ acts diffeomorphically on a smooth manifold $M$ via $\varphi$; that is
	\begin{equation*}
	\begin{split}
		\varphi:\Gcal&\to \diff(M) \\
		g &\mapsto \varphi_g.
	\end{split}
	\end{equation*}
	For a given $\xi\in\mathfrak{g}=\Lie(\Gcal)$, we denote the vector field $\xi_M$ (called the infinitesimal generator) on $M$ given by the rule:
	\begin{equation}
		df_x(\xi_M) := \left.\frac{d}{dt}\right|_{t=0} \, f(\varphi_{\exp(t\xi)}(x)),\quad f\in C^\infty(M).
	\end{equation}
\end{definition}
This lets us compute the differential, $dF$.

\begin{theorem}
	Suppose that $F(h) = \langle f_X,h.f_Z\rangle_{\mathcal{H}}$ as described above. Then
	\begin{equation}\label{eq:differential}
	dF_e(\xi) = \sum_{\substack{x_i\in X\\ z_j\in Z}} \, c_{ij}\cdot d\left(\tilde{k}_{x_i}\right)_{z_j}\left(-\xi_M(z_j)\right),
	\end{equation}
	where $\tilde{k}_{x_i} = k(x_i,\cdot)$.
\end{theorem}
\begin{remark}
	The notation for the differential of a function used throughout this paper is $df_x(v)$, where $x\in M$ and $v\in T_xM$:
	\begin{equation*}
	df_x(v) = \left.\frac{d}{dt}\right|_{t=0} \, f(c(t)),\quad c(0)=x,\quad c'(0)=v.
	\end{equation*}
\end{remark}
\begin{proof}
	This follows from a straightforward application of the chain rule. 
	\begin{equation*}
	\begin{split}
	dF_e(\xi) &= \left.\frac{d}{dt}\right|_{t=0} \, \langle f_X,\exp(t\xi).f_Z\rangle_{\Hcal} \\
	&= \sum \, c_{ij}\cdot \left.\frac{d}{dt}\right|_{t=0}\, k\left(x_i,\exp(-t\xi)z_j\right) \\
	&= \sum \, c_{ij}\cdot d\left(\tilde{k}_{x_i}\right)_{z_j} \cdot \left.\frac{d}{dt}\right|_{t=0} \, \exp(-t\xi)z_j \\
	&= \sum_{\substack{x_i\in X\\ z_j\in Z}} \, c_{ij} \cdot d\left(\tilde{k}_{x_i}\right)_{z_j}\left(-\xi_M(z_j)\right).
	\end{split}
	\end{equation*}
	This matches equation \eqref{eq:differential}.
\end{proof}
Of course, to construct the gradient flow we are interested in computing $dF_h$ instead of just $dF_e$. We will accomplish this via left-translation. Left-translation is given by the smooth map $\ell_h:\Gcal\to \Gcal$ where $x\mapsto hx$. Its differential gives rise to an isomorphism of tangent spaces, $(\ell_h)_*:\mathfrak{g}\xrightarrow{\sim} T_h\Gcal$.
\begin{corollary}
	Under the identification $T_h\Gcal \cong (\ell_h)_*\mathfrak{g}$, we have
	\begin{equation}\label{eq:diff_calc}
	dF_h\left( (\ell_h)_*\xi\right) = 
	\sum_{\substack{x_i\in X\\ z_j\in Z}} \, c_{ij} \cdot
	d\left(\tilde{k}_{x_i}\right)_{h^{-1}z_j}\left(-\xi_M(h^{-1}z_j)\right),
	\end{equation}
\end{corollary}

In order to turn the co-vector $dF_h\in T_h^*\Gcal$ into a vector $\nabla F_h\in T_h\Gcal$, we will use a left-invariant metric. This can be accomplished by defining an inner-product, $\langle\cdot,\cdot\rangle_\mathfrak{g}$ on $\mathfrak{g}$ and lifting to a (Riemannian) metric on $\Gcal$ via left-translation, i.e.
$$\langle (\ell_h)_*\eta,(\ell_h)_*\xi\rangle_{T_h\Gcal} := \langle \eta,\xi\rangle_{\mathfrak{g}}.$$
This allows us to define the gradient of $F$ as
\begin{equation}
\langle \nabla F_h,(\ell_h)_*\xi\rangle_{T_h\Gcal} = dF_h\left((\ell_h)_*\xi\right).
\end{equation}
This allows for a way to obtain a (local) solution to \eqref{eq:max2} by following
\begin{equation}\label{eq:gradient_flow}
\dot{h} = \nabla F(h).
\end{equation}
To demonstrate the generality of this registration algorithm, the cases in Table \ref{tab:MG} will be examined in the remainder of this paper.
\begin{table}[h]
    \centering
    \begin{tabular}{r|cccc}
         $M$ & $S^1$ & $T^2$ & $S^2$  & $\mathbb{R}^n$\\ \hline  
         $\mathcal{G}$& $S^1$ & $T^2$ & $\mathrm{SO}(3)$ & $\SE(n)$ \\
    \bottomrule
    \end{tabular}
    \caption{A list of the manifolds / Lie groups to be studied in this paper. Particular interest will be placed on the $(\mathbb{R}^2,\SE(2))$ and $(\mathbb{R}^3,\SE(3))$ examples.}
	\label{tab:MG}
\end{table}

\section{Low-dimensional Examples}
\label{sec:lowdim}
While the primary focus of this paper is to align RGB-D images in $\mathbb{R}^3$, we take the opportunity to show how the above algorithm works on other spaces: specifically the spaces $S^1$, $T^2$, and $S^2$: the circle, the 2-torus, and the sphere.
\subsection{The Circle}\label{sec:circle_example}
To demonstrate our algorithm, we will first work out the case where the manifold containing the clouds is the circle, $M=S^1$. $S^1$ acts naturally on itself via left-translations so we shall take $G=S^1$ as well and $\varphi$ to be left-translation, i.e.
\begin{equation*}
    \varphi(\alpha)(\theta) = \alpha+\theta \mod 2\pi.
\end{equation*}
The Lie algebra for the circle is the real line and we will therefore take the inner product to be scalar multiplication,
\begin{equation*}
    \langle a,b\rangle_{\mathbb{R}} = ab.
\end{equation*}
The last piece of geometry we need is the kernel (the information space and the clouds will be determined later). To build $k$, we will use the Euclidean distance of points on the circle: let $x,y\in [0,2\pi)\sim S^1$.
\begin{equation}
\begin{split}
    d_2(x,y)^2 &:= \left( \cos x-\cos y\right)^2 + \left(\sin x-\sin y\right)^2 \\
    &= 2\left( 1 - \cos(x-y)\right).
\end{split}
\end{equation}
We will use the Gaussian kernel based on this distance function:
\begin{equation}
    k(x,y) = \sigma^2\exp\left(\frac{-d_2(x,y)^2}{2\ell^2}\right),
\end{equation}
where $\ell>0$ is a parameter called the length-scale.

If we again call $c_{ij} = \langle \ell_X(x_i),\ell_Z(z_j)\rangle_\mathcal{I}$, then using \eqref{eq:diff_calc}, the differential is (for $s\in\mathbb{R}$)
\begin{equation*}
    dF_h((\ell_h)_*s) = \frac{1}{\ell^2} \sum_{\substack{x_i\in X\\ z_j\in Z}}
    c_{ij}k(x_i,z_j-h)\sin(z_j-h-x_i)\cdot s.
\end{equation*}
Therefore, the gradient is
\begin{equation}
    \nabla F(h) = \frac{1}{\ell^2}
    \sum_{\substack{x_i\in X, z_j\in Z}}
    c_{ij}k(x_i,z_j-h)\sin(z_j-h-x_i).
\end{equation}
Finally, the Hessian can be found by simply differentiating the gradient as is
\begin{equation}
    H_{S^1} = \frac{1}{\ell^2} \, \sum_{x_i\in X, z_j\in Z} \, c_{ij}k(x_i,z_j-h) [\sin^2\left(z_j - h - x_i \right) -\cos\left(z_j - h - x_i \right) ].
\end{equation}

\begin{example}
We perform the registration problem on the circle. $X=\texttt{linspace}(0,\pi/2,10)$ and $Z=\texttt{linspace}(\pi/2,\pi,12)$. Figure~\ref{fig:clouds_circle_registration} shows two point clouds before and after registration. The labels , $\ell_X$ and $\ell_Z$, for all points are set to $1$.
\begin{figure}%
    \centering
    \subfloat
    {\includegraphics[width=0.4\columnwidth]{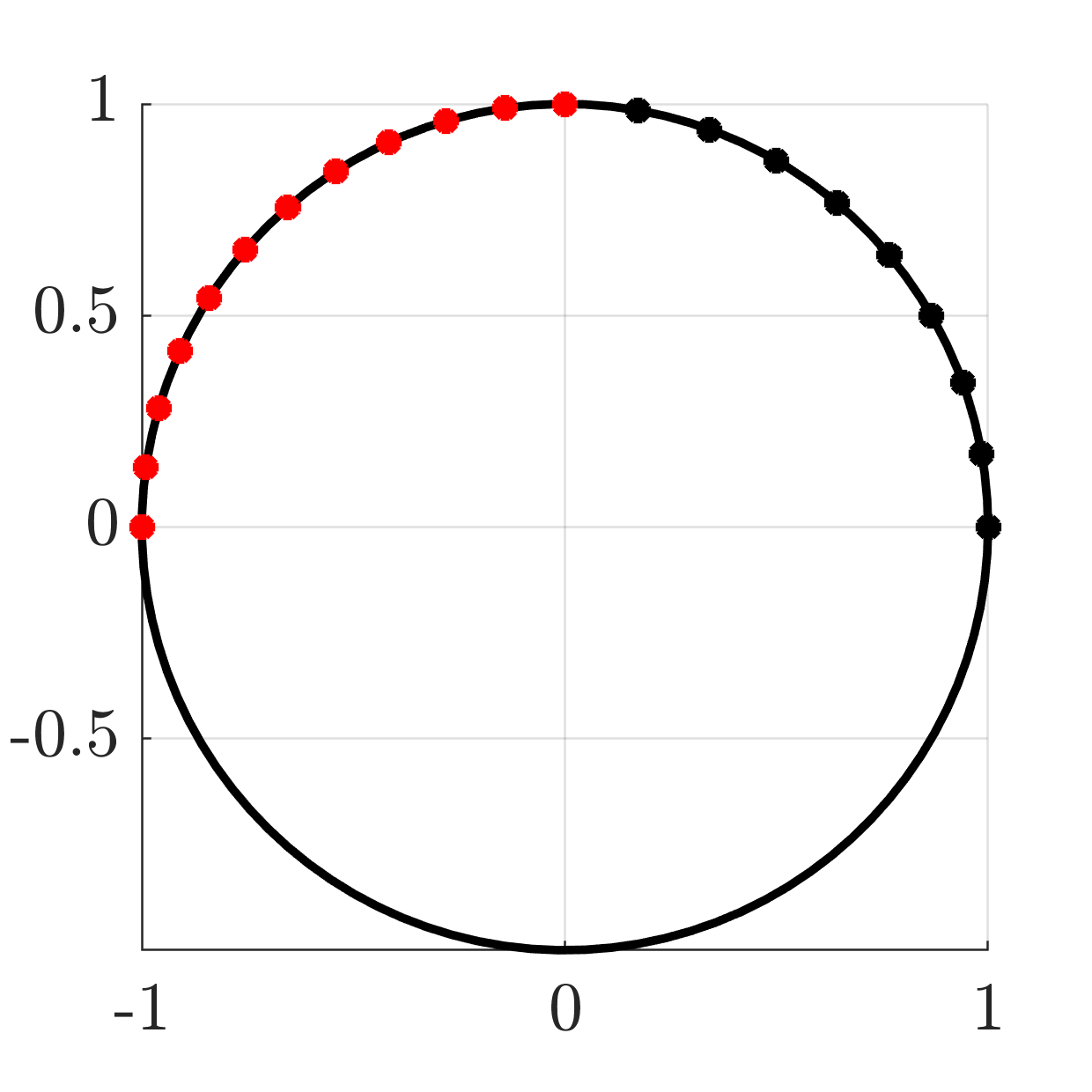}}
    \subfloat
    {\includegraphics[width=0.4\columnwidth]{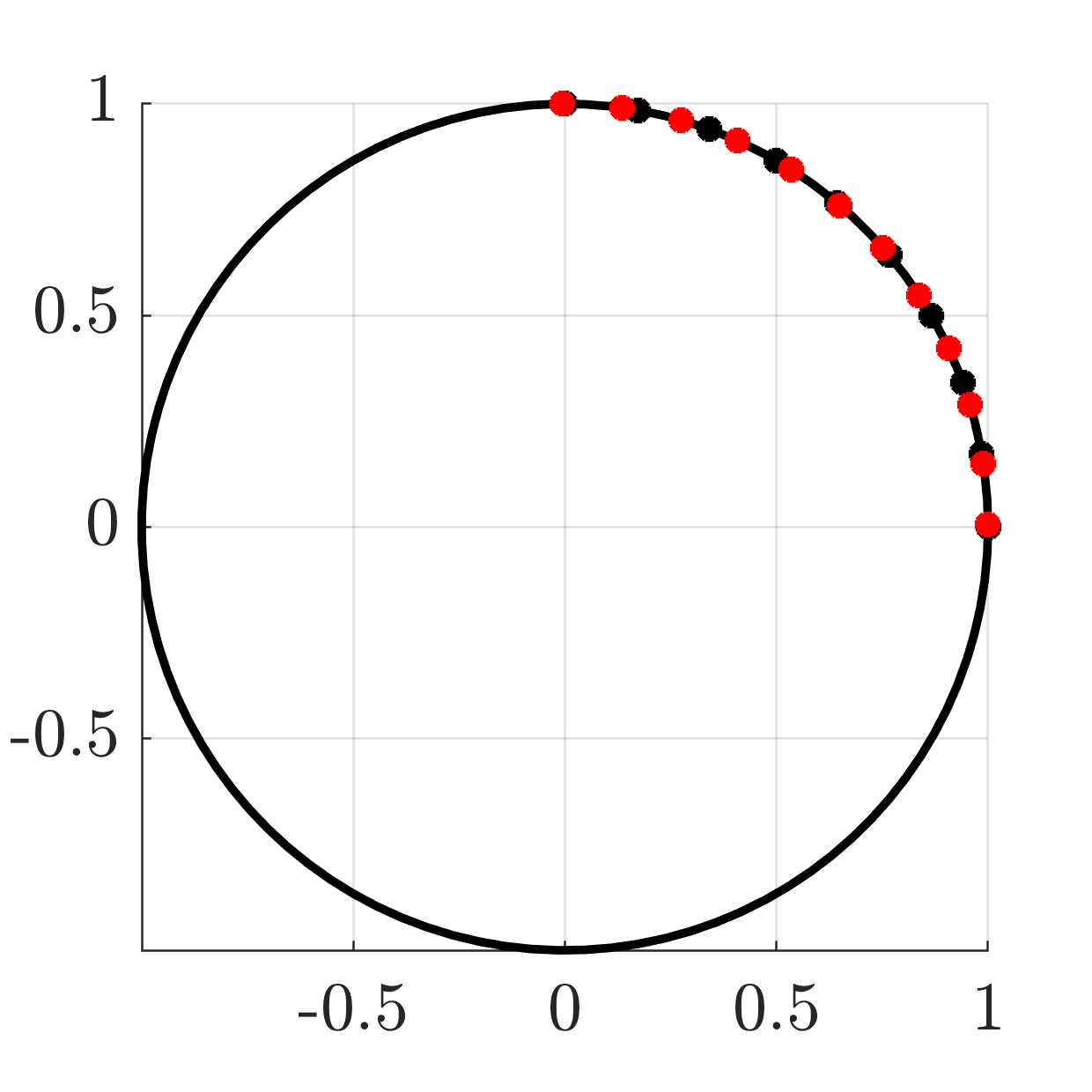}}
    \caption{Left: The clouds $X$ and $A\cdot Z$. Right: The clouds $X$ and $T\cdot A\cdot Z$.}
    \label{fig:clouds_circle_registration}
\end{figure}
\end{example}
\subsection{The Torus}\label{sec:torus}
The torus is the product of two circles: $T^2 = S^1\times S^1$. Therefore, both the gradient and Hessian will be two copies of the corresponding object from the circle case. 

Similarly to the circle case, the group will act via left-translations and we will take $(x^1,x^2)\in[0,2\pi)\times[0,2\pi)\sim T^2$ for coordinates. The kernel will be
\begin{equation}
\begin{split}
    k(x,y) &= \sigma^2\exp\left( \frac{-d_{T^2}(x,y)^2}{2\ell^2} \right), \\
    d_{T^2}(x,y)^2 &= d_2(x^1,y^1)^2 + d_2(x^2,y^2)^2.
    \end{split}
\end{equation}
The gradient is then
\begin{equation}
    \nabla F(h) = \frac{1}{\ell^2} \, \sumxz \, c_{ij} k(x_i,z_j-h) \begin{bmatrix}
    \sin\left(z_j^1 - h^1 - x_i^1\right) \\
    \sin\left(z_j^2 - h^2 - x_i^2\right)
    \end{bmatrix}.
\end{equation}
The Hessian is straight-forward to calculate due to the fact that $T^2$ is abelian. The Hessian is
\begin{equation}
    H_{T^2} = \frac{1}{\ell^2} \, \sumxz \, c_{ij} k(x_i,z_j-h) \cdot h_{ij},
\end{equation}
where
\begin{equation*}
\begin{split}
    h_{ij} = \diag \bigg( &\sin^2\left( z_j^1-h^1-x_i^1\right) - \cos\left(z_j^1 - h^1 - x_i^1\right), \\
    & \sin^2\left( z_j^2-h^2-x_i^2\right) - \cos\left(z_j^2 - h^2 - x_i^2\right) \bigg)
\end{split}
\end{equation*}
This procedure naturally generalizes to $T^n$.

\begin{example}
We perform the registration problem on the torus with a 5 pointed star as shown in Figure~\ref{fig:clouds_torus_registration}. $X$ contains 50 points while $Z$ contains 75. The displacement is $[1,1]$ and the transformation found was $[1.0077,0.9962]$ which is an error of $[-0.0077,0.0038]$ which has norm of 0.0085.
\begin{figure}%
    \centering
    \subfloat
    {\includegraphics[width=0.45\columnwidth]{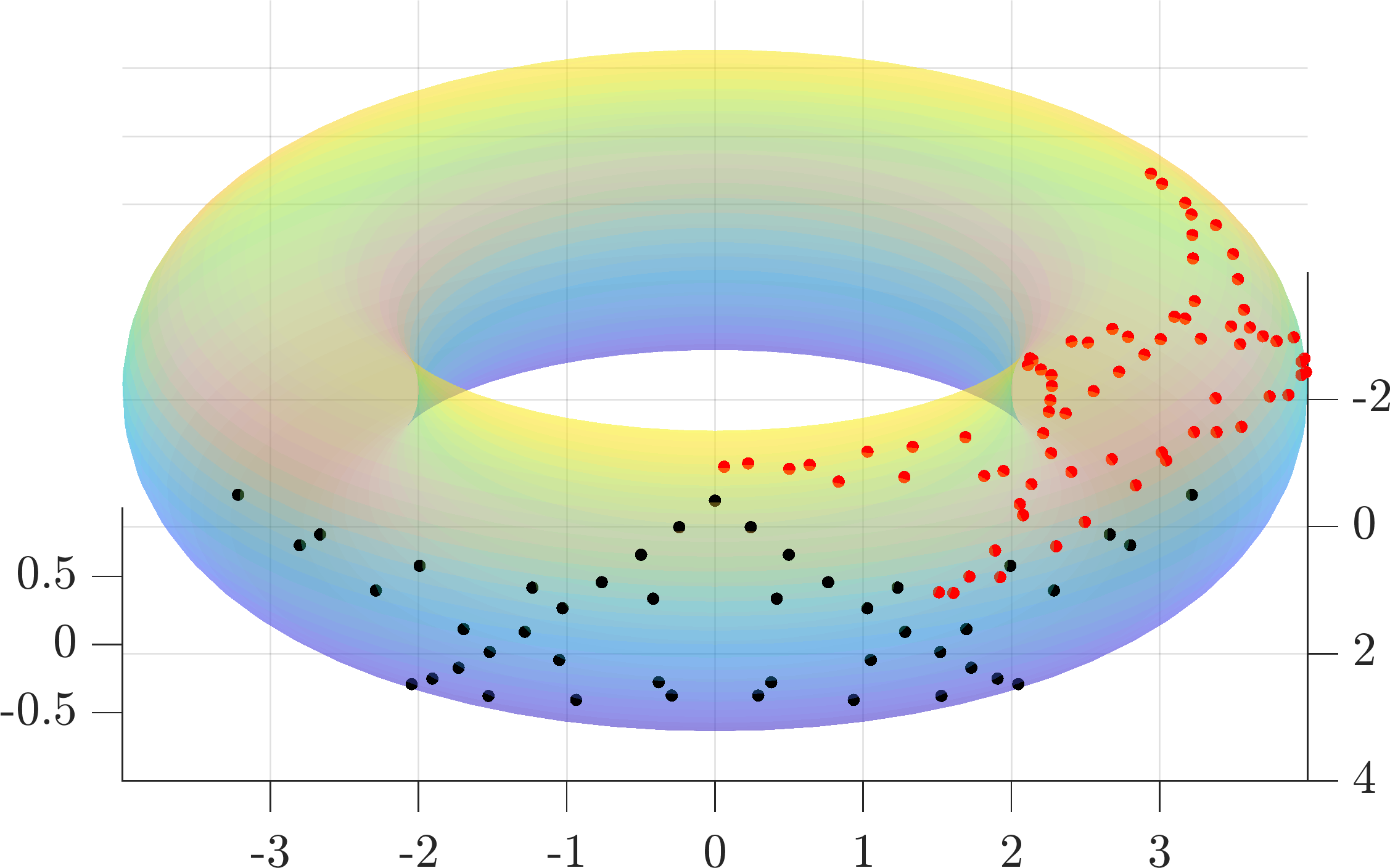}}
    \subfloat
    {\includegraphics[width=0.45\columnwidth]{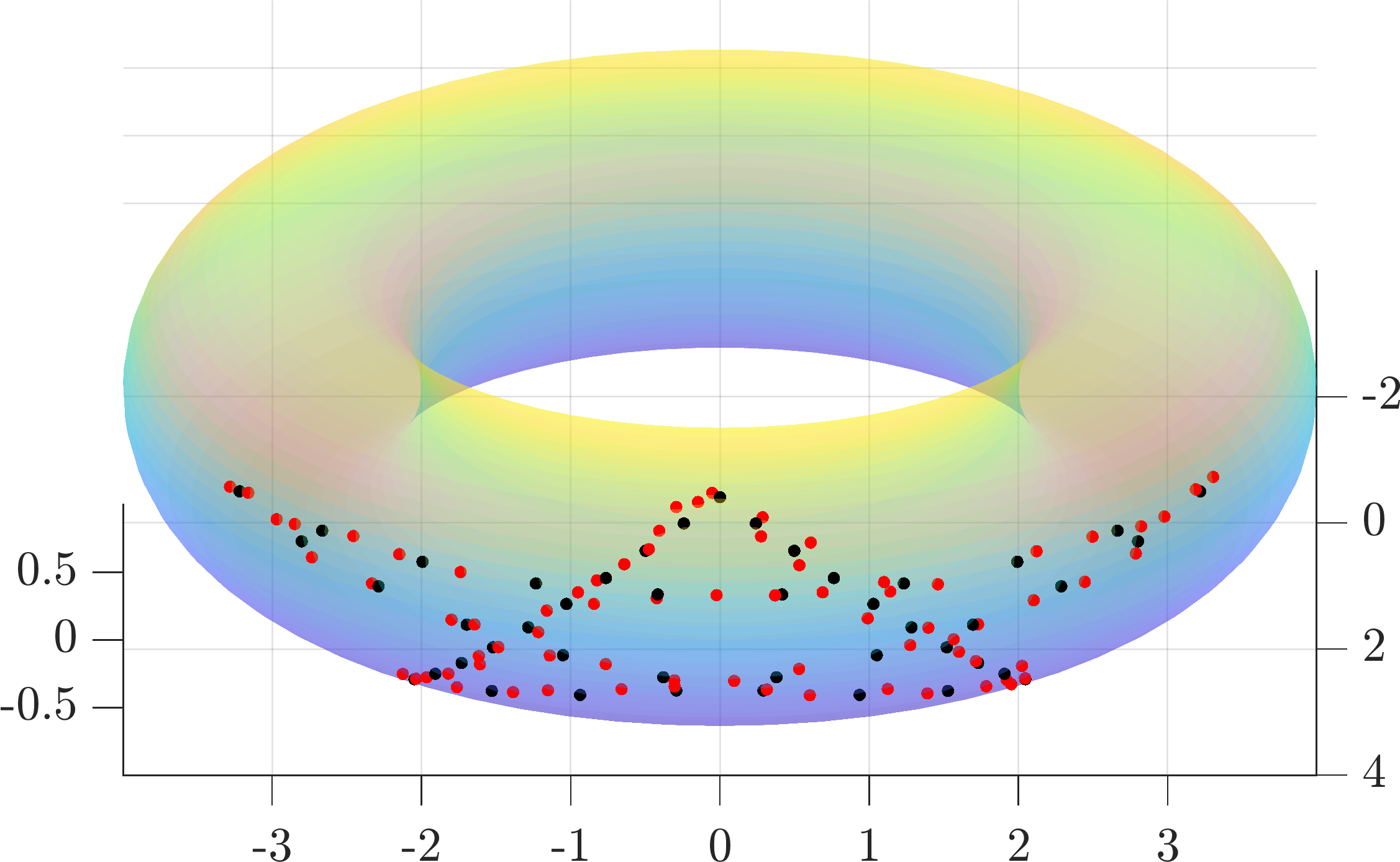}}
    \caption{Top: The clouds $X$ and $A\cdot Z$. Bottom: The clouds $X$ and $T\cdot A\cdot Z$.}
    \label{fig:clouds_torus_registration}
\end{figure}
\end{example}

\subsection{The Sphere}
The other initial example we wish to develop is the case when $M=S^2$, the sphere. The Lie group we choose to act will be $\mathrm{SO}_3$, the special orthogonal group. This will act via usual matrix multiplication: let $x\in S^2\subset \mathbb{R}^3$, then
\begin{equation}
    \varphi(g)(x) = gx.
\end{equation}
In terms of the point cloud data, assume that we have clouds with points in $\mathbb{R}^3$. These points will be projected to $S^2$ (via normalizing) which is compatible with the group action described above. Given two points $x,y\in\mathbb{R}^3$, their arc-length distance, once projected to $S^2$, is
\begin{equation}
    d_{S^2}(x,y) = \arccos\left(\frac{\langle x,y\rangle_3}{\lVert x\rVert\lVert y\rVert}\right).
\end{equation}
With this distance (which is different from the distance used in the circle example), we define the kernel to be
\begin{equation}
k(x,y) = \exp\left(\frac{-d_{S^2}(x,y)^2}{2\ell^2}\right).
\end{equation}
In order to determine the gradient, we still need a metric for $\mathfrak{so}_3$ for which we will use the (negative of the) Killing form. This gives a gradient of $\nabla F(R) = R\hat{\omega}$ where $R\in\mathrm{SO}_3$, $\hat{\cdot}:\mathbb{R}^3\to \mathrm{skew}_3$ (three-by-three skew-symmetric matrices), and
\begin{equation}\label{eq:sphere_grad}
    \omega = \frac{1}{\ell^2} \sum_{\substack{x_i\in X\\ z_j\in Z}} c_{ij}
    \frac{k(x_i,R^{-1}z_j)d_{S^2}(x_i,R^{-1}z_j)}
    {\sqrt{\lVert x_i\rVert^2\lVert z_j\rVert^2 - \langle x_i,R^{-1}z_j\rangle^2}}
    \left(x_i\times R^{-1}z_j\right).
\end{equation}

\subsubsection{The Hessian}
Unlike the circle and torus examples, calculating the Hessian for the sphere is substantially more involved. This follows from the fact that $\SO(3)$ is no longer abelian. In order to determine the Hessian, we need to know the connection on $\SO(3)$ induced by the Killing form on $\mathfrak{so}(3)$. Let us choose the standard basis $\{e_x,e_y,e_z\}\in\mathfrak{so}(3)$ where
\begin{equation*}
    e_x = \begin{bmatrix}[r]
    0 & 0 & 0 \\
    0 & 0 & -1 \\
    0 & 1 & 0
    \end{bmatrix},~ e_y = \begin{bmatrix}[r]
    0 & 0 & 1 \\
    0 & 0 & 0 \\
    -1 & 0 & 0
    \end{bmatrix},~ e_z = \begin{bmatrix}[r]
    0 & -1 & 0 \\
    1 & 0 & 0 \\
    0 & 0 & 0
    \end{bmatrix}.
\end{equation*}
Additionally, let $\{E_x,E_y,E_z\}$ be their corresponding left-invariant vector fields. The connection, calculated from \eqref{eq:LC_connection}, is given by
\begin{equation*}
\begin{split}
    \nabla_{E_x}E_y = \frac{1}{2} E_z, & \quad \nabla_{E_x}E_z = -\frac{1}{2}E_y, \\
    \nabla_{E_y}E_z = \frac{1}{2} E_x, & \quad \nabla_{E_y}E_x = -\frac{1}{2}E_z, \\
    \nabla_{E_z}E_x = \frac{1}{2} E_y, & \quad \nabla_{E_z}E_y = -\frac{1}{2}E_x.
\end{split}
\end{equation*}
Writing \eqref{eq:sphere_grad} as $\hat{\omega} = \omega_xe_x + \omega_ye_y+\omega_ze_z$, the Hessian is given by
\begin{equation*}
    H(e_\alpha,e_\beta) = \sum_{\gamma\in\{x,y,z\}} \, \langle \left(\mathcal{L}_{E_\alpha}\omega_\gamma\right)E_\gamma 
    + \omega_\gamma \nabla_{E_\alpha}E_\gamma,e_\beta\rangle 
    = \mathcal{L}_{E_\alpha}\omega_\beta + \frac{1}{2}\varepsilon_{\alpha\gamma\beta}\omega_\gamma,
\end{equation*}
where $\varepsilon_{\alpha\gamma\beta} = \langle [e_\alpha,e_\gamma],e_\beta\rangle$ are the structure constants. Combining, the Hessian is
\begin{align}\label{eq:Hsphere}
    -H_{\SO(3)} 
    = \begin{bmatrix}
    \mathcal{L}_{E_x}\omega_x & \mathcal{L}_{E_x}\omega_y -\frac{1}{2}\omega_z & \mathcal{L}_{E_x}\omega_z + \frac{1}{2}\omega_y \\
    \mathcal{L}_{E_y}\omega_x + \frac{1}{2}\omega_z & \mathcal{L}_{E_y}\omega_y & \mathcal{L}_{E_y}\omega_z - \frac{1}{2}\omega_x \\
    \mathcal{L}_{E_z}\omega_x - \frac{1}{2}\omega_y & \mathcal{L}_{E_z}\omega_y + \frac{1}{2}\omega_x & \mathcal{L}_{E_z}\omega_z
    \end{bmatrix},
\end{align}
which, although it is tedious, can be computed exactly.
\begin{remark}
    The reason that \eqref{eq:Hsphere} returns the \textit{negative} of the Hessian is because we are actually differentiating the inverse of the group element.
\end{remark}

\begin{table}[b!]
\begin{center}
\footnotesize
\begin{tabular}{l r}
\toprule
    Parameters & Value \\
        \midrule
        transformation convergence threshold $\epsilon$ & $1\mathrm{e}{-4}$ \\
        gradient norm convergence threshold $\epsilon$ & $5\mathrm{e}{-4}$ \\
        kernel characteristic length-scale $\ell$ & $0.25$ \\
        kernel characteristic length-scale $\ell$ (iteration $> 3$)  & $0.15$ \\
        kernel characteristic length-scale $\ell$ (iteration $> 10$)  & $0.10$ \\
        kernel characteristic length-scale $\ell$ (iteration $> 20$)  & $0.05$ \\
        kernel signal variance $\sigma$ & $0.1$  \\
        step length & $0.1$ \\
        color space inner product scale & $1\mathrm{e}{-5}$ \\
        kernel sparsification threshold & $1\mathrm{e}{-3}$ \\
\bottomrule
\end{tabular}
\caption{ Parameters used for evaluation of the globe registration. The kernel characteristic length-scale is chosen to be adaptive as the algorithm converges; intuitively, we prefer a large neighborhood of correlation for each point, but as the algorithm reaches the convergence reducing the local correlation neighborhood allows for faster convergence and better refinement.}
\label{tab:parameters_so3}
\end{center}
\end{table}

\begin{figure}[t]
    \centering
    \subfloat{\includegraphics[width=0.3\columnwidth]{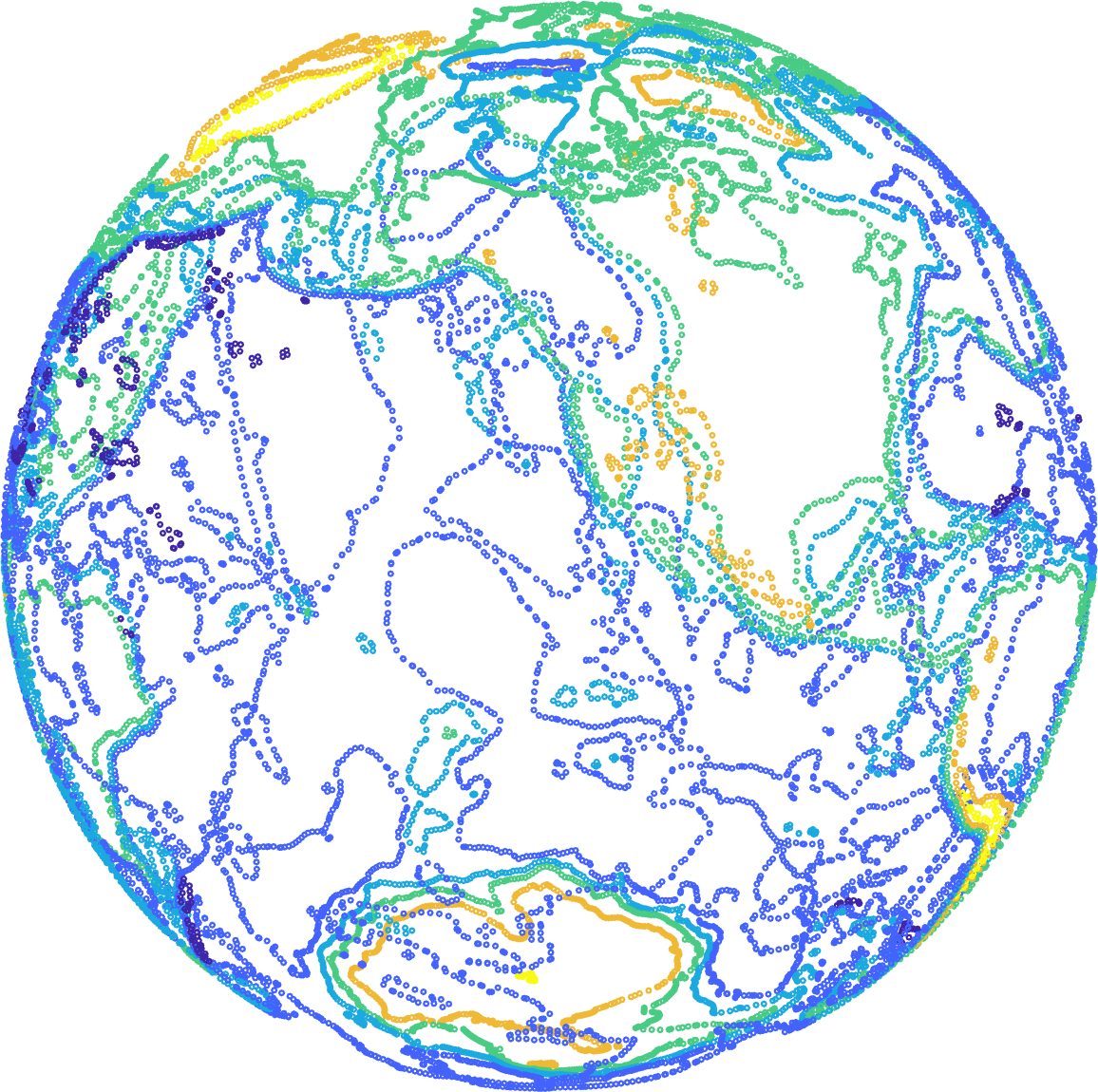}}\hfill
    \subfloat{\includegraphics[width=0.3\columnwidth]{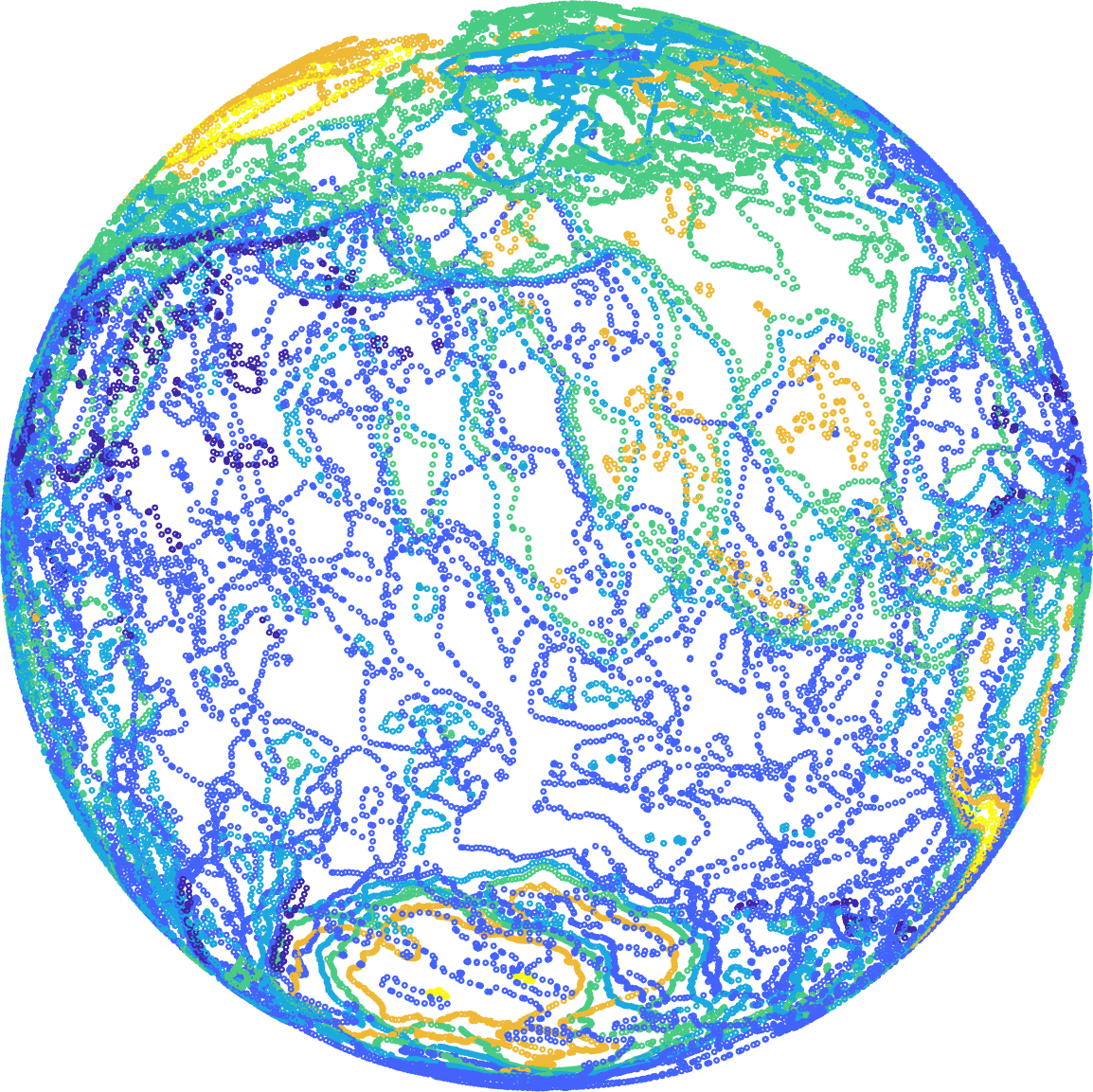}}\hfill
    \subfloat{\includegraphics[width=0.3\columnwidth]{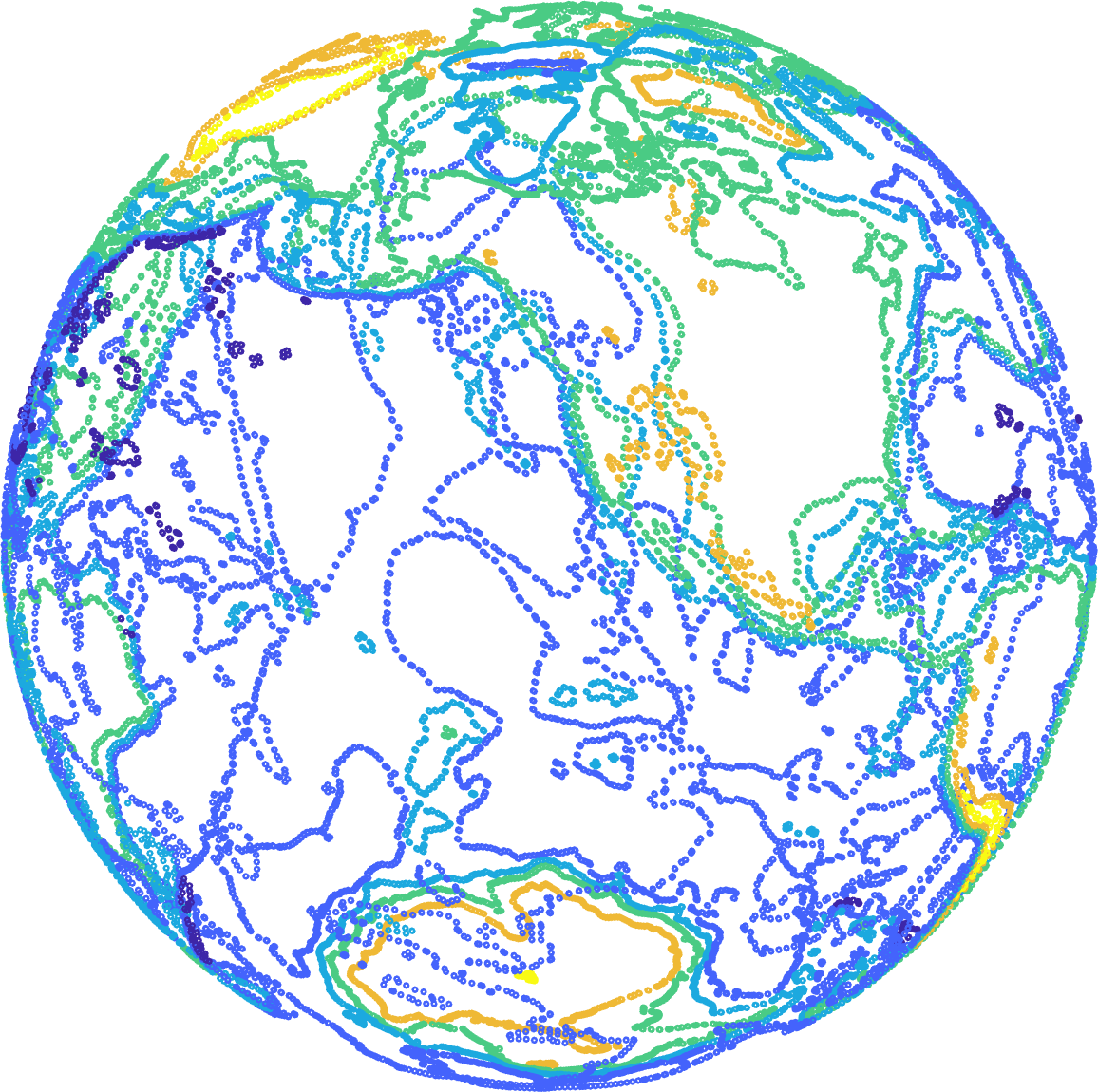}}
    \caption{Left (Reference): The image of the contour map of the Earth. Center (Perturbed): The plot of $X$ along with $A\cdot Z$. Right (Aligned): The plot of $X$ along with $T\cdot A\cdot Z$.}
    \label{fig:globe_registration}
\end{figure}

\begin{example}
To illustrate registration in this case, we will align two pictures of the globe. This will be done with a topographic map of the Earth using Matlab's \texttt{topo}. Both $X$ and $Z$ are initialized to be identical, however $Z$ is perturbed by a transformation $A\in \SO(3)$. The goal is to recover $A^{-1}$ to realign $Z$ with $X$. The algorithm reports $T\in\SO(3)$ which should be the inverse of $A$. The values of $A$, $T$, $T\cdot A$, and $\lVert T\cdot A\rVert_F$ are shown in \eqref{eq:results_so3_registration}.
\begin{equation}\label{eq:results_so3_registration}
    \begin{split}
        A &= \begin{bmatrix}[r]
        0.9386 &  -0.3170  &  0.1361 \\
        0.3230  &  0.9461  & -0.0244 \\
        -0.1210  &  0.0668  &  0.9904
        \end{bmatrix}, \quad 
        T = \begin{bmatrix}[r]
        0.9391  &  0.3217  & -0.1205 \\
        -0.3160 &   0.9466 &   0.0643 \\
        0.1347 &  -0.0223  &  0.9906
        \end{bmatrix}, \\
        T\cdot A &= \begin{bmatrix}[r]
        1.0000  & -0.0014  &  0.0007 \\
        0.0014  &  1.0000 &  -0.0024 \\
        -0.0007 &   0.0024 &   1.0000
        \end{bmatrix}, \quad 
        \lVert T\cdot A\rVert_F = 0.0040.
    \end{split}
\end{equation}
The parameters used in this registration are included in Table~\ref{tab:parameters_so3}. The initial and final images are shown in Figure~\ref{fig:globe_registration}.

\end{example}

\section{Limitations on Compact Riemannian Manifolds}
\label{sec:compact}

In this section, we discuss an important problem regarding the positive-definiteness of the kernel on compact Riemannian manifolds~\citep{jayasumana2015kernel,feragen2015geodesic,borovitskiy2020mat}. The key result is the following theorem.
\begin{theorem}[Theorem 5 of \citealt{feragen2015geodesic}, \citealt{bekka2008kazhdan}]
	Let $X$ be a topological space and let $\Psi:X\times X\to\mathbb{R}$ be a continuous kernel such that $\Psi(x,x)=0$ and $\Psi(x,y)=\Psi(y,x)$ for all $x,y\in X$. Then the following are equivalent
	\begin{enumerate}
		\item $\Psi$ is conditionally of negative type, and
		\item the kernel $\exp(t\Psi)$ is of positive type for every $t\geq 0$.
	\end{enumerate}
\end{theorem}

However, the issue is that for $\Psi=-d(x,y)^2$ to be conditionally negative, we need to be on Euclidean space. Fortunately, this is more than what we require in this work as we are not concerned if the positive-definiteness holds for all $t\geq 0$. Even though we are not on a Euclidean space, we do not know that there does not exist a $t\geq 0$ to make it positive definite. The following theorem ensures the existence of such a $t$.
\begin{theorem}
	Let $M\subset \mathbb{R}^m$ be some embedded manifold. Let $d_c$ be the chordal distance on $M$ (straight lines in $\mathbb{R}^m$) and $d_g$ be the geodesic distance. For a finite collection of points $\{x_i\}\subset M$, there exists a $t\geq 0$ such that the matrix
	\begin{equation*}
		K_{ij} = \left[ \exp\left( -t\cdot d_g(x_i,x_j)^2\right) \right]
	\end{equation*}
	is positive definite.
\end{theorem}
\begin{proof}
	It is clear that for all $i$, $K_{ii}=1$. As long as the diagonal elements dominate, the matrix will be positive definite. Let
	\begin{equation*}
		R_i = \sum_{j\ne i} \, \exp\left( -t\cdot d_g(x_i,x_j)\right).
	\end{equation*}
	Choose $t\geq 0$ large enough such that $R_i < 1$ for all $i$. Then by the Gershgorin circle theorem, all the eigenvalues of $K$ are positive. A bound on $t$ can be explicitly found via the following. Let 
	\begin{equation*}
		\Lambda := \min_{i\ne j} d_g(x_i,x_j).
	\end{equation*}
	Then 
	\begin{equation*}
		R_i \leq (N-1)\cdot \exp(-t\Lambda),
	\end{equation*}
	where $N$ is the number of points. This produces
	\begin{equation*}
		t \geq \frac{\ln(N-1)}{\Lambda}.
	\end{equation*}
\end{proof}
This is not a particularly useful result as the parameter $t$ depends on the point cloud. Therefore it remains an open problem whether we can find a $t$ (or an interval) that works on all possible point clouds systematically. This is a particularly attractive problem in machine learning as one ultimately wants to learn all hyperparameters. In particular, the derivation of useful bounds on $t$ as well as algorithmic implementation of learning it are interesting future research directions.

\begin{example}[The Sphere]
Let us see how the case of the sphere holds up. Here, we take the kernel
\begin{equation}\label{eq:sphere_kernel}
	k(x,y) = \exp\left( \frac{-d_{S^2}(x,y)^q}{2\ell^2}\right), \quad d_{S^2}(x,y) = \arccos\left( \frac{\langle x,y\rangle_3}{\lVert x\rVert \lVert y \rVert} \right).
\end{equation}

In particular, $t = 1/(2\ell^2)$. For the usual case the exponent is $q=2$. However, this may be changed as well. The idea is that $k$ will not be positive definite for all $\ell$, but it should be when $\ell$ is small enough. This will be numerically explored by taking a random sample of $N$ points on $S^2$ and computing the matrix $K$. If all of its eigenvalues are positive, then it is positive definite. As such, we are interested in finding the smallest eigenvalue.

\begin{figure}[t]
	\centering
	\subfloat{\includegraphics[width=0.5\textwidth]{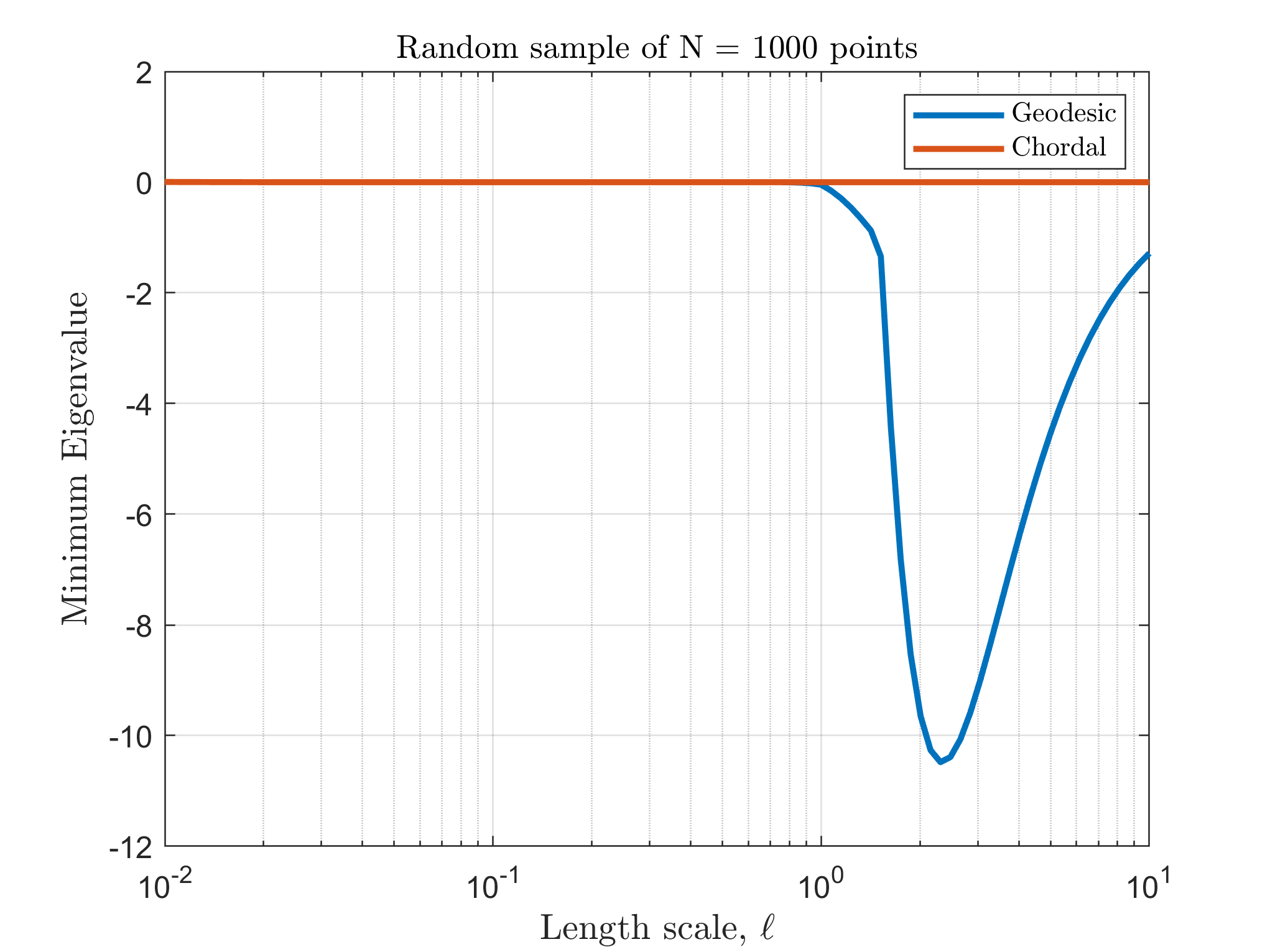}} 
	\subfloat{\includegraphics[width=0.5\textwidth]{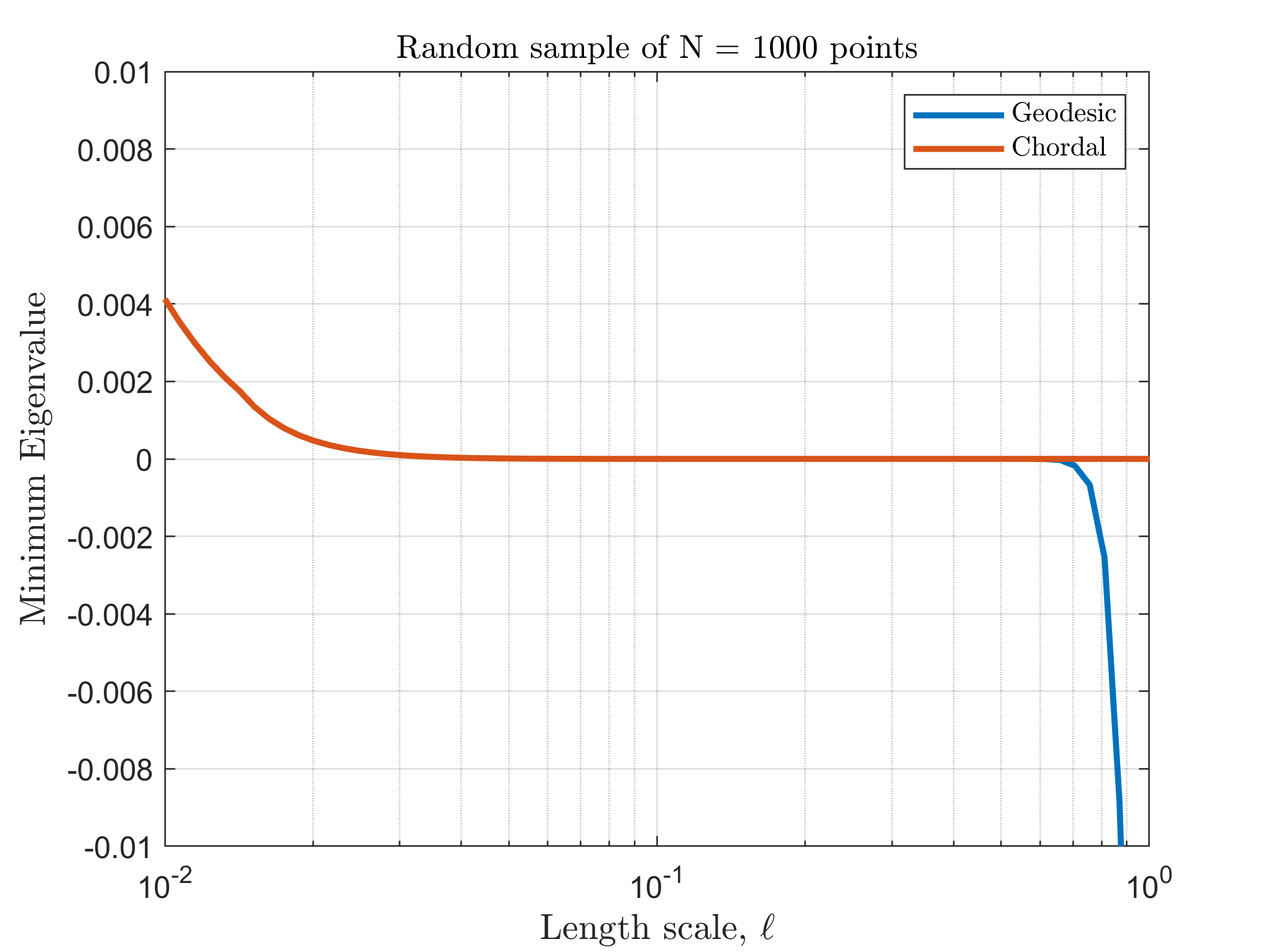}}\\
	\subfloat{\includegraphics[width=0.5\textwidth]{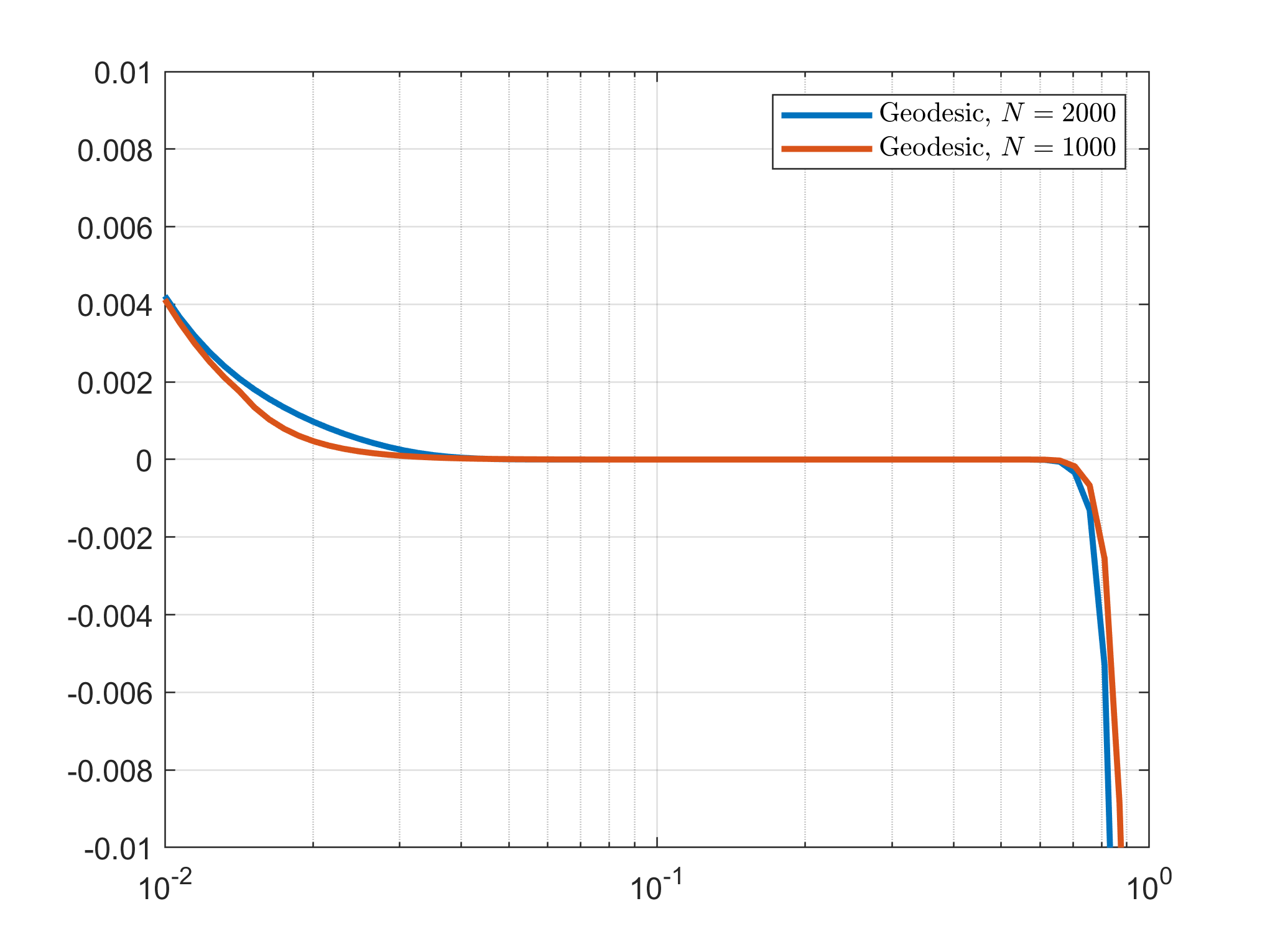}}
	\subfloat{\includegraphics[width=0.5\textwidth]{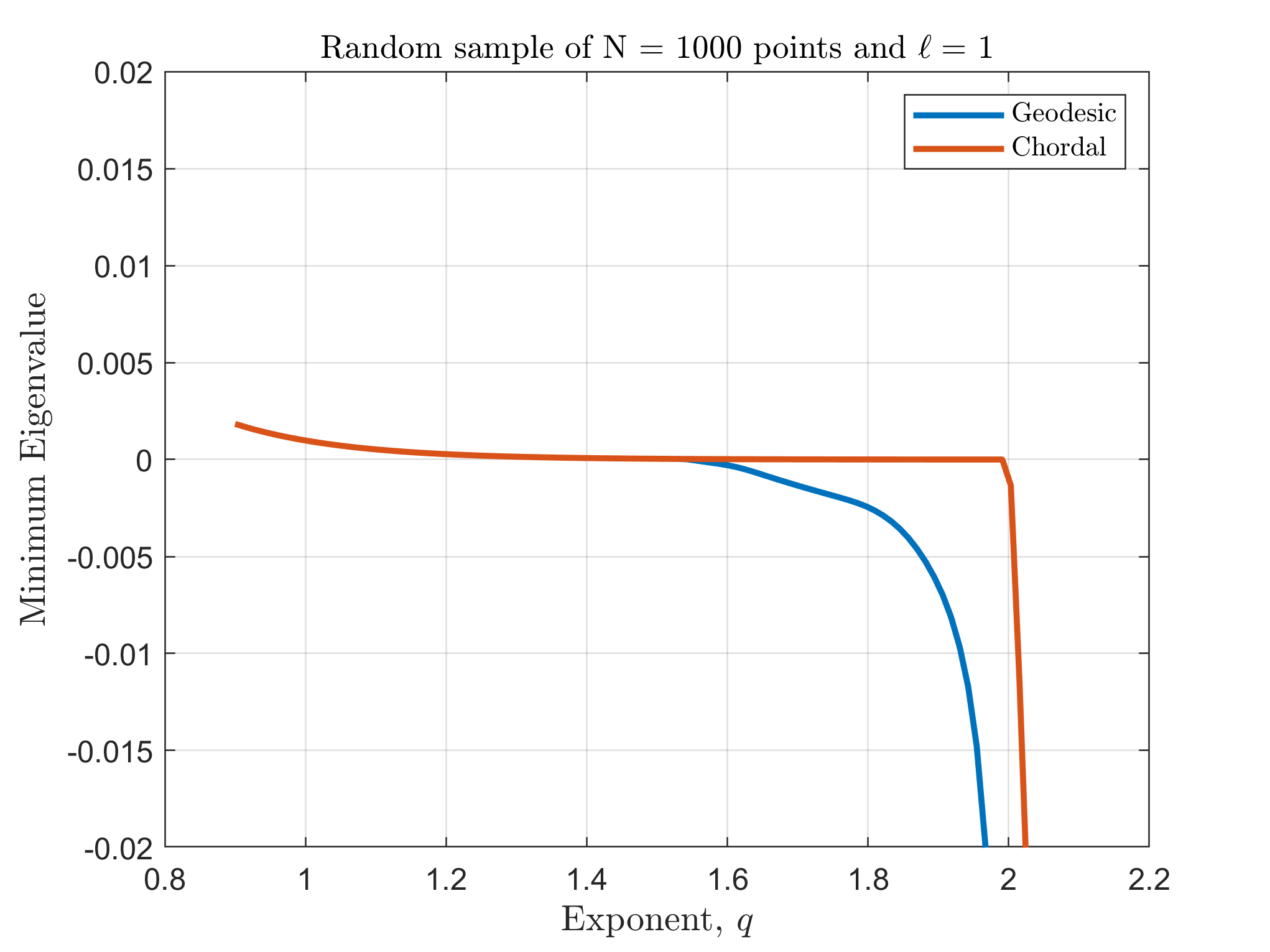}}
	\caption{Top left: The geodesic kernel \eqref{eq:sphere_kernel} fails to be positive definite as $\ell$ becomes large. As long as $\ell$ is small (which corresponds to $t$ being large), the kernel appears to be positive definite. Top right: A zoomed version. Bottom left: The positive definiteness appears to not depend too heavily on the number of points. Bottom right: All of the previous figures were with the exponent fixed at $q=2$. This figure varies the exponent. The chordal kernel appears to be positive-definite when $q\leq 2$ which is in agreement with the work of \citet{feragen2015geodesic}.}
\end{figure}
\end{example}

\section{Special Euclidean Groups}
\label{sec:euclidean}
We now move onto the more involved examples where $M=\mathbb{R}^n$ and $\Gcal = \SE(n)$, the special Euclidean group in $n$ dimensions.  $\Gcal$ acts on $M$ in in the standard fashion: let $(R,T)\in \SE(n)$ where $R\in \SO(n)$ and $T\in\mathbb{R}^n$,
$$(R,T).x = Rx+T,\quad x\in\mathbb{R}^n.$$
We will also choose the squared exponential kernel for $k:\mathbb{R}^n\times\mathbb{R}^n\to\mathbb{R}$:
\begin{equation}
k(x,y) = \sigma^2\exp\left(\frac{-\lVert x-y\rVert_n^2}{2\ell^2}\right),
\end{equation}
for some fixed real parameters $\sigma$ and $\ell$, and $\lVert\cdot\rVert_n$ is the standard Euclidean norm on $\mathbb{R}^n$. In order to determine the gradient flow \eqref{eq:gradient_flow}, we need to compute the infinitesimal generators of the action $\SE(n)\curvearrowright\mathbb{R}^n$ as well as decide on a left-invariant metric for $\se(n)=\Lie(\SE(n))$.
\subsection{Infinitesimal Generator}
For a fixed $\xi\in\se(n)$, it has the form $\xi=(\omega,v)$ where $\omega\in\so(n)$ and $v\in\mathbb{R}^n$. Because the infinitesimal generator map $\mathfrak{g}\to \mathfrak{X}(M)$ is a Lie algebra homomorphism (where $\mathfrak{X}(M)$ is the space of all vector fields over $M$, see \textsection 27 of \citealt{diffGeometry}), we see that $\xi_M = \omega_M+v_M$. A straight forward computation leads to
\begin{equation}
\xi_{\mathbb{R}^n}x = \hat{\omega}x+v,\quad \hat{\omega}\in \mathrm{skew}(n),\quad v\in\mathbb{R}^n.
\end{equation}
\subsection{Metric}
We need to choose a metric on $\se(n)$ to turn $dF_h$ into $\nabla F_h$. For this example, we will take a multiple of the Killing form on $\so(n)$ and the Euclidean norm on $\mathbb{R}^n$. That is,
\begin{equation}\label{eq:metric}
\langle (\omega,v),(\eta,u)\rangle_{\se(n)} =  b^2 \cdot\langle v,u\rangle_n - a^2\left(\frac{n-2}{2}\right)\cdot \tr(\omega\eta),
\end{equation}
where $\langle\cdot,\cdot\rangle_n$ is the standard Euclidean inner product on $\mathbb{R}^n$ (see \citet{park1994kinematic} for a discussion in three dimensions), and $a$ and $b$ are tuning parameters. The reason for the $(2-n)/2$ term is because with this normalization (with $a=1$) the skew matrices $E_{ij}-E_{ji}$ are orthonormal. Here $E_{ij}$ denotes the matrix with all zeros except for a $1$ in the $(i,j)$-coordinate.
\subsection{Calculating the Gradient}
Before we find the gradient, let us first determine its differential (at the identity for simplicity). 
\begin{equation*}
\begin{split}
dF_e(\xi) &= \sum_{\substack{x_i\in X\\ z_j\in Z}} \, c_{ij} \cdot d\left(\tilde{k}_{x_i}\right)_{z_j}\left(-\xi_M(z_j)\right) \\
&= \sum_{\substack{x_i\in X\\ z_j\in Z}} \, c_{ij} \cdot 
\frac{1}{\ell^2}k(x_i,z_j)\cdot \langle (x_i-z_j), \left( -\hat{\omega}z_j-v\right)\rangle_n.
\end{split}
\end{equation*}
To turn $dF_e$ into $\nabla F_e$, we will compute $\nabla_\omega F_e$ and $\nabla_v F_e$ separately:
\begin{equation*}
\begin{split}
-a^2\left(\frac{n-2}{2}\right)\cdot\tr\left[(\nabla_\omega F_e)\hat{\omega}\right] &= \sum_{\substack{x_i\in X\\ z_j\in Z}} \, c_{ij} \cdot 
\frac{1}{\ell^2}k(x_i,z_j)\cdot \langle (x_i-z_j) , \left( -\hat{\omega}z_j\right)\rangle_n,\\
b^2\langle (\nabla_vF_e),v\rangle_n &= \sum_{\substack{x_i\in X\\ z_j\in Z}} \, c_{ij} \cdot 
\frac{1}{\ell^2}k(x_i,z_j)\cdot \langle (x_i-z_j) , \left( -v\right)\rangle_n.
\end{split}
\end{equation*}
To solve for this in coordinates, we will let $\{e^m\}_{m=1}^n$ be the standard orthonormal basis for $(\mathbb{R}^n,\langle\cdot,\cdot\rangle_n)$ and $\{ J^{pq} \}_{p<q} := \{E_{pq}-E_{qp}\}_{p<q}$ be as above. Then, the gradient becomes:
\begin{equation*}
\begin{split}
\left(\nabla_\omega F_e\right)^{pq} &= \frac{1}{a^2\ell^2} \sum_{\substack{x_i\in X\\ z_j\in Z}} \, c_{ij} \cdot 
k(x_i,z_j)\cdot \langle (x_i-z_j) , \left( -J^{pq}z_j\right)\rangle_n, \\
\left( \nabla_v F_e\right)^m &= \frac{1}{b^2\ell^2} \sum_{\substack{x_i\in X\\ z_j\in Z}} \, c_{ij} \cdot 
k(x_i,z_j)\cdot \langle (x_i-z_j) , \left( -e^m\right)\rangle_n.
\end{split}
\end{equation*}
The above can be simplified by computing the inner product on the right hand side:
\begin{equation*}
\begin{split}
\left(\nabla_\omega F_e\right)^{pq} &= \frac{1}{a^2\ell^2} \sum_{\substack{x_i\in X\\ z_j\in Z}} \, c_{ij} \cdot 
k(x_i,z_j)\cdot \left(x_i^pz_j^q-x_i^qz_j^p\right), \\
\left( \nabla_v F_e\right)^m &= \frac{1}{b^2\ell^2} \sum_{\substack{x_i\in X\\ z_j\in Z}} \, c_{ij} \cdot 
k(x_i,z_j)\cdot \left(z_j^m-x_i^m\right).
\end{split}
\end{equation*}
Likewise, to translate away from the origin, we note that if $h=(R,T)\in\SE(n)$, $(\ell_h)_*(\hat{\omega},v) = (R\hat{\omega},Rv)$. Then, if we express the gradient as $\nabla F_h = (\ell_h)_*(\hat{\omega},v) = (R\hat{\omega},Rv)$, we get the following expression for $(\hat{\omega},v)\in\se(n)$:
\begin{equation}\label{eq:flow}
\begin{split}
\hat{\omega}^{pq} &= \frac{1}{a^2\ell^2} \sum_{\substack{x_i\in X\\ z_j\in Z}} \, c_{ij} \cdot 
k(x_i,\tilde{z}_j)\cdot \left(x_i^p\tilde{z}_j^q-x_i^q\tilde{z}_j^p\right), \\
 v^m &= \frac{1}{b^2\ell^2} \sum_{\substack{x_i\in X\\ z_j\in Z}} \, c_{ij} \cdot 
k(x_i,\tilde{z}_j)\cdot \left(\tilde{z}_j^m-x_i^m\right),
\end{split}
\end{equation}
where $\tilde{z}_j = h^{-1}z_j = R^{\transpose} z_j-R^{\transpose} T$.

The Hessian calculation will be postponed until Section \ref{sec:RgeomSEn}.

\section{Analysis and Verification of Idea}
\label{sec:verification}
It is important to take a moment and examine when solving Problem \ref{prob:problem2} actually causes the clouds to be best aligned. It is of course impossible to perfectly align two non-identical clouds. Presented below is a discussion of the question: when the two clouds \textit{are} identical, when does the identity element in the group maximize \eqref{eq:max2}?

Suppose that $Z=X$ and $\ell_Z=\ell_X$. Then the identity ideally should be a fixed-point of \eqref{eq:gradient_flow}. This leads to the following, which merely checks the consistency of the framework.
\begin{proposition}
	Assume that for all $h\in \Gcal$ and $x\in M$, $k(hx,hx)= k(x,x)$. Then the identity is a global maximum of $F$.
\end{proposition}
\begin{proof}
	We have that
	\begin{equation*}
	F(h) = \langle f_X,h.f_X\rangle_{\Hcal},\quad F(e) = \lVert f_X\rVert_{\Hcal}^2 \geq 0.
	\end{equation*}
	Then using the Cauchy-Schwarz inequality we obtain:
	\begin{equation*}
	\begin{split}
		\left|\langle f_X,h.f_X\rangle_{\Hcal}\right| &\leq \lVert f_X\rVert_{\Hcal} \cdot \lVert h.f_X\rVert_{\Hcal},
	\end{split}
	\end{equation*}
	which is less than $F(e)$ provided that $\lVert h.f_X\rVert_{\Hcal}=\lVert f_X\rVert_{\Hcal}$. Computing this, we see that
	\begin{equation*}
	\begin{split}
		\lVert h.f_X\rVert_{\Hcal}^2 &= \sum_{x_i,z_j\in X} \, c_{ij}\cdot k\left(h^{-1}x_i,h^{-1}z_j\right)\\
		&= \sum_{x_i,z_j\in X} \, c_{ij}\cdot k\left(x_i,z_j\right) = \lVert f_X\rVert_{\Hcal}^2.
	\end{split}
	\end{equation*}
	Combining everything, we get that
	\begin{equation*}
	\left|F(h)\right| \leq \lVert f_X\rVert_{\Hcal}\cdot \lVert h.f_X\rVert_{\Hcal} \leq \lVert f_X\rVert_{\Hcal}^2 = F(e).
	\end{equation*}
\end{proof}
\begin{corollary}
	Suppose $k:M\times M\to\mathbb{R}$ is a stationary kernel~\citep[Page 82]{rasmussen2006gaussian}, that is $k(x,y) = k(d(x,y))$ for some distance function $d$. If $\Gcal$ acts isometrically on $M$, then the identity is a global maximum of $F$.
\end{corollary}
\begin{corollary}
	The identity is a global maximum for the $\SE(n)$ case.
\end{corollary}
\begin{proof}
	This follows from the fact that $\SE(n)$ acts on $\mathbb{R}^n$ isometrically, i.e. $\lVert hx-hy\rVert_n = \lVert x-y\rVert_n$.
\end{proof}

\begin{theorem}
        The maximizer of Problem~\ref{prob:problem2}, minimizes the angle between $f_X$ and $f_Z$.
\end{theorem}
\begin{proof}
        Suppose $h^* \in \Gcal$ is the maximizer of Problem~\ref{prob:problem2}. Then $\langle f_X, f_Z^* \rangle \geq \langle f_X, f_Z \rangle$ and $\lVert f_Z^* \rVert_{\Hcal} \leq \lVert f_Z \rVert_{\Hcal}$. Using Cauchy-Schwarz inequality:
    \begin{equation}
        \nonumber 0 \leq \lvert \langle f_X, f_Z \rangle \rvert \leq \lvert \langle f_X, f_Z^* \rangle \rvert \leq \lVert f_X \rVert_{\Hcal} \lVert f_Z^* \rVert_{\Hcal} \leq \lVert f_X \rVert_{\Hcal} \lVert f_Z \rVert_{\Hcal}
    \end{equation}
    dividing by $\lVert f_X \rVert_{\Hcal} \lVert f_Z \rVert_{\Hcal}$ and replacing $\lVert f_Z \rVert_{\Hcal}$ in the denominator by $\lVert f_Z^* \rVert_{\Hcal}$:
    \begin{equation}
        \nonumber 0 \leq \cos(\theta) \leq \frac{\lvert \langle f_X, f_Z^* \rangle \rvert}{\lVert f_X \rVert_{\Hcal} \lVert f_Z^* \rVert_{\Hcal}} \leq \frac{\lVert f_Z^*\rVert_{\Hcal}}{\lVert f_Z^* \rVert_{\Hcal}} \leq 1
    \end{equation} 
    \begin{equation}
        \nonumber 0 \leq \cos(\theta) \leq \cos(\theta^*) \leq 1
    \end{equation}
    \begin{equation}
        \nonumber 0 \leq \theta^* \leq \theta \leq \frac{\pi}{2}
    \end{equation}
where $\cos(\theta) = \lvert \langle f_X, f_Z \rangle \rvert / (\lVert f_X \rVert_{\Hcal} \lVert f_Z \rVert_{\Hcal})$.
\end{proof}

\section{Integrating the Flow for the Special Euclidean Group}
\label{sec:senintegration}

Now that we know the direction for the flow, what remains is to determine a way to integrate the flow and to determine a reasonable step size. We integrate using the Lie exponential map to preserve the group structure and the step size is calculated using a $4^{th}$-order Taylor approximation in a line search algorithm.
\subsection{Integrating}
We will take care that in integrating \eqref{eq:gradient_flow}, our trajectories will remain on $SE(n)$. This is slightly problematic because integrating is an additive process and $SE(n)$ is not closed under addition. To address this, we note that if in \eqref{eq:gradient_flow}, $(\hat{\omega},v)$ is constant in $\se(n)$ (i.e. $\nabla F_h$ is a left-invariant vector field) the solution is merely
\begin{equation*}
(R(t),T(t)) = (R_0,T_0)\exp(t(\hat{\omega},v)),
\end{equation*}
where $\exp:\se(n)\to\SE(n)$ is the Lie exponential map (which is merely the matrix exponential). We will exploit this by assuming that $\hat{\omega}$ and $v$ are constant over each time step.
\begin{align*}
\nonumber \begin{bmatrix}
R_{k+1} & T_{k+1} \\ 0 & 1
\end{bmatrix} &= \begin{bmatrix}
R_{k} & T_{k} \\ 0 & 1
\end{bmatrix} \cdot \exp \begin{bmatrix}
t\hat{\omega} & t v \\ 0 & 0
\end{bmatrix} \\
\nonumber &= \begin{bmatrix}
R_k & T_k \\ 0 & 1
\end{bmatrix}\begin{bmatrix}
\Delta R & \Delta T \\ 0 & 1
\end{bmatrix} \\
&= \begin{bmatrix}
R_k\Delta R & R_k\Delta T + T_k \\
0 & 1
\end{bmatrix}.
\end{align*}
Explicit formulas for $\Delta R$ and $\Delta T$ will be discussed in \textsection\ref{subsec:se3} for the special case where $n=3$.
Combining all of this, we get our integration step to be
\begin{align}\label{eq:newstep}
\nonumber R_{k+1} &= R_k \Delta R \\
T_{k+1} &= R_k\Delta T + T_k.
\end{align}
\subsection{Step Size}\label{sec:step} 
We can use \eqref{eq:flow} to point in  the direction of maximal growth and \eqref{eq:newstep} to find the updateded element in $\SE(n)$. However, we currently have no intelligent way of choosing $t$. We will proceed by a Taylor approximation of the solution curve. If we let $G(t) := F(h\exp(t\xi))$, then we want to find the value of $t$ that maximizes $G$. We compute a $4^{th}$-order Taylor expansion of $G(t)$ about $t=0$ and determine the value of $t$ that maximizes this polynomial.
\begin{equation}\label{eq:fourth}\begin{split}
G(t) \approx \sumxz \, c_{ij} \cdot e^{\alpha_{ij}} \left\{ 
g^1_{ij}t + g^2_{ij}t^2 + g^3_{ij}t^3 + g^4_{ij}t^4 \right\},
\end{split}\end{equation}
where 
\begin{equation*}
\begin{split}
g^1_{ij} &= \beta_{ij} \\
g^2_{ij} &= \gamma_{ij}+\frac{1}{2}\beta_{ij}^2 \\
g^3_{ij} &= \delta_{ij}+\beta_{ij}\gamma_{ij}+\frac{1}{6}\beta_{ij}^3 \\
g^4_{ij} &= \varepsilon_{ij} + \beta_{ij}\delta_{ij}+\frac{1}{2}\beta_{ij}^2\gamma_{ij} + \frac{1}{2}\gamma_{ij}^2 + \frac{1}{24}\beta_{ij}^4 \\
\alpha_{ij} &= \frac{-1}{2\ell^2}\lVert x_i-z_j\rVert_n \\
\beta_{ij} &= \frac{-1}{\ell^2} \langle \hat{\omega}z_j+v,x_i-z_j\rangle_n \\
\gamma_{ij} &= \frac{-1}{2\ell^2}\left( \lVert \hat{\omega}z_j+v\rVert_n^2 + 2\langle \hat{\omega}^2z_j+v,x_i-z_j\rangle_n\right) \\
\delta_{ij} &= \frac{1}{\ell^2}\left( \langle -\hat{\omega}z_j-v,\hat{\omega}^2z_j +\hat{\omega}v\rangle_n + \langle -\hat{\omega}^3z_j-\hat{\omega}^2v,x_i-z_j\rangle_n\right) \\
\varepsilon_{ij} &= \frac{-1}{2\ell^2}\left( \lVert\hat{\omega}^2z_j+\hat{\omega}v\rVert_n^2 +2\langle \hat{\omega}z_j+v,\hat{\omega}^3z_j +\hat{\omega}^2v\rangle_n + 2\langle \hat{\omega}^4z_j+\hat{\omega}^3v,x_i-z_j\rangle_n\right)
\end{split}
\end{equation*}
The ``optimal'' step size is then taken to be the value of $t>0$ that maximizes the quartic \eqref{eq:fourth}.

\begin{figure}[t]
    \centering
    \subfloat
    {\includegraphics[width=0.35\columnwidth]{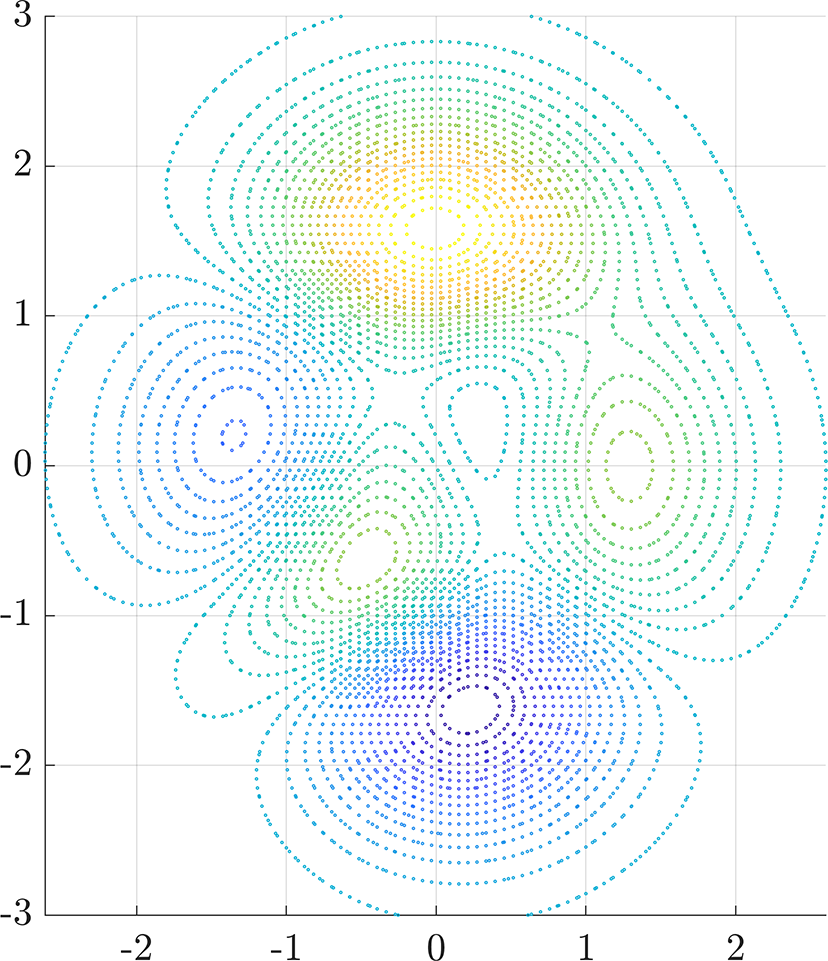}}~
    \subfloat
    {\includegraphics[width=0.375\columnwidth]{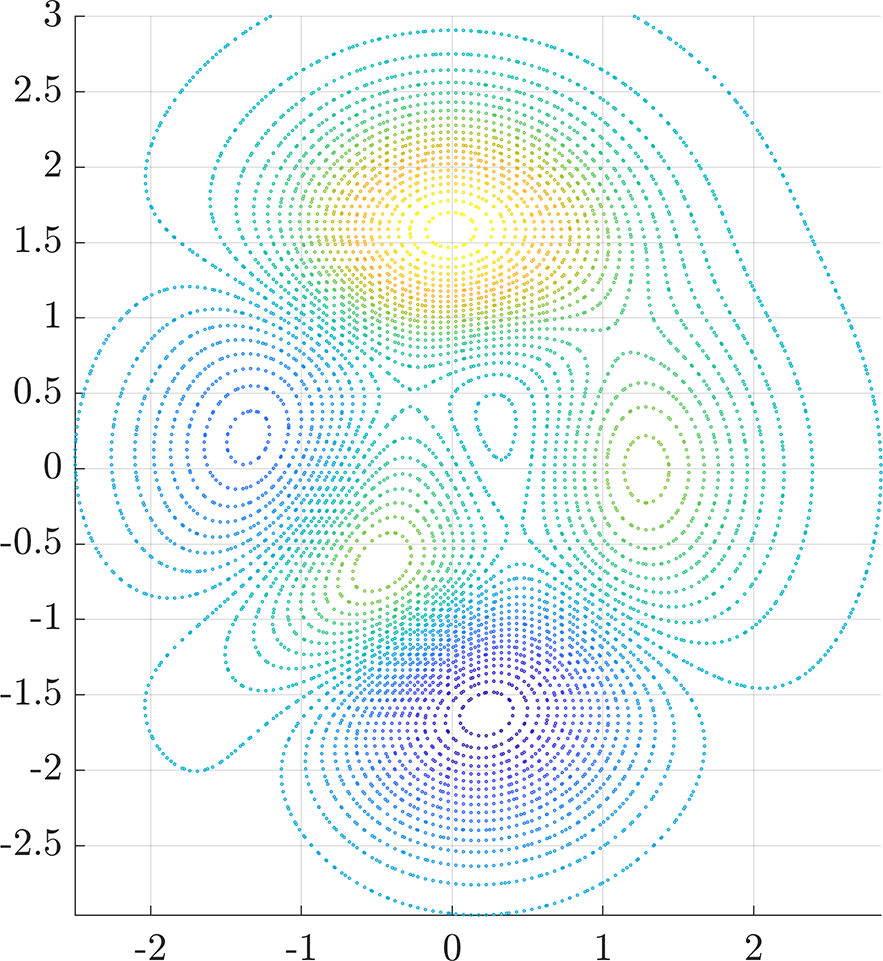}}
    \caption{The images above show $X$ and $Z$ as used in Example \ref{ex:se2_registration}. The cloud $X$ contains 5282 points while $Z$ contains 6865 points.}
    \label{fig:clouds_se2_images}
\end{figure}
\begin{figure}[t]
    \centering
    \subfloat
    {\includegraphics[width=0.35\columnwidth]{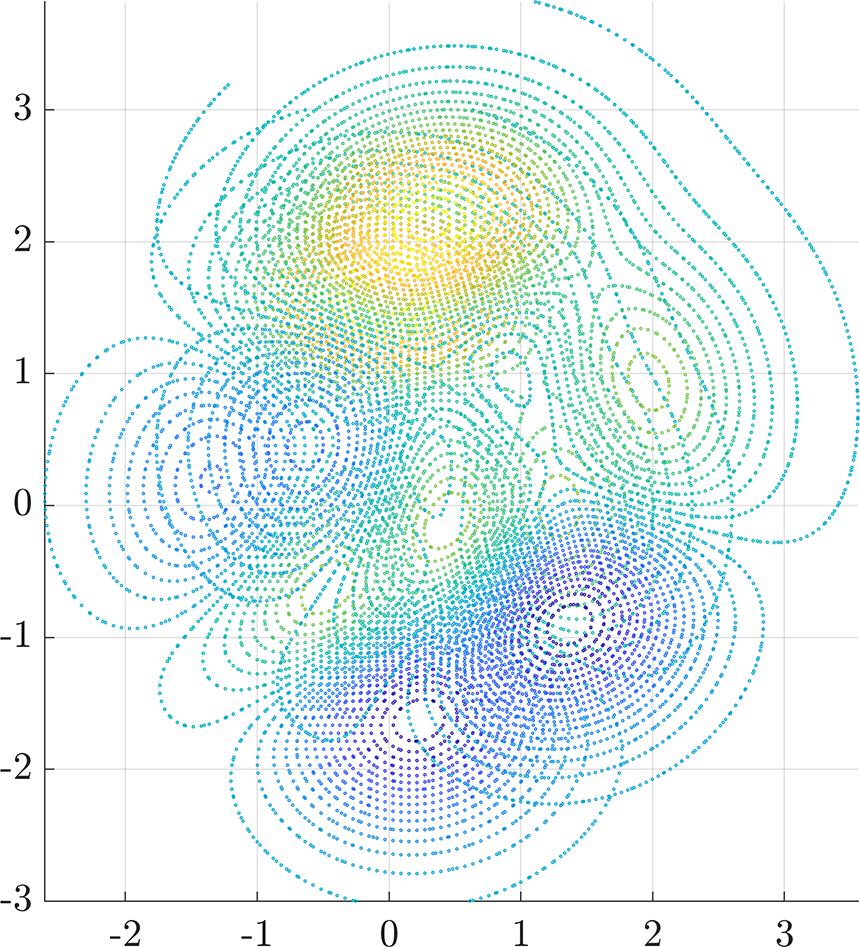}}~
    \subfloat
    {\includegraphics[width=0.35\columnwidth]{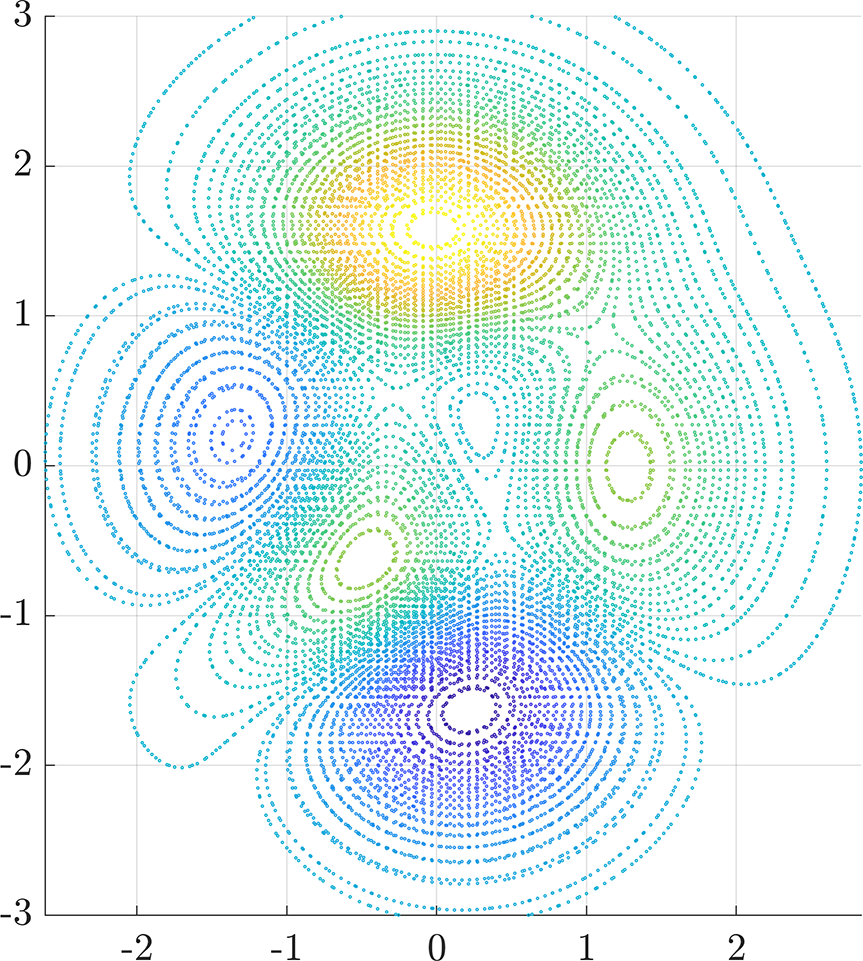}}
    \caption{Left: The clouds $X$ and $A\cdot Z$. Right: The clouds $X$ and $T\cdot A\cdot Z$.}
    \label{fig:clouds_se2_registration}
\end{figure}

\subsection{Special Case: $n=2$}
When we take the special case of $\SE(2)$, the gradient \eqref{eq:flow} takes the special form:
\begin{equation}\begin{split}
\omega &= \frac{1}{a^2\ell^2} 
\sum_{\substack{x_i\in X\\ z_j\in Z}} c_{ij}
k(x_i,h^{-1}z_j)\left(x_i\tilde{\times}(h^{-1}z_j)\right) \\
v &= \frac{1}{b^2\ell^2}\, \sumxz c_{ij} k(x_i,h^{-1}z_j)\left( h^{-1}z_j-x_i\right),
\end{split}\end{equation}
where $\tilde\times$ is the ``two-dimensional cross product,''
\begin{equation*}
\begin{bmatrix} x_1 \\ x_2 \end{bmatrix}
\tilde{\times} \begin{bmatrix}
y_1 \\ y_2 \end{bmatrix} = 
x_1y_2 - x_2y_1.
\end{equation*}
Finally, the exponential map for $\SE(2)$ can be solved exactly. 
\begin{equation}
    \begin{split}
        \Delta R &= \begin{bmatrix}
        \cos\omega & -\sin\omega \\
        \sin\omega & \cos\omega
        \end{bmatrix}, \\
        \Delta T &= \frac{1}{\omega} \begin{bmatrix}
        \sin\omega & \cos\omega-1 \\
        1-\cos\omega & \sin\omega
        \end{bmatrix} v.
    \end{split}
\end{equation}

\begin{example}\label{ex:se2_registration}
We perform a bird's-eye view map of registration. For the purposes of this example, we choose $X$ to be a contour plot of Matlab's \texttt{peaks}(100) while $Z$ is of \texttt{peaks}(120). Specifically, $X$ is made up of $40$ contour lines while $Z$ is made up of $43$ lines. For both images, the coordinates are constrained in the following way: $x,y\in [-3,3]$. The two images are shown in Figure \ref{fig:clouds_se2_images}.

As these images are initially aligned, we first perturb them and attempt to realign them. The initial perturbation is give by $A\in\SE(2)$ while the algorithm determines $T\in\SE(2)$. In this sense, we want $T\cdot A = \mathrm{Id}$. The values of $A$, $T$, $T\cdot A$, and $\lVert T\cdot A\rVert_F$ are shown in \eqref{eq:results_se2_registration}.

\begin{equation}\label{eq:results_se2_registration}
    \begin{split}
        A &= \begin{bmatrix}[r]
        0.9654 &  -0.2607  &  0.7250 \\
        0.2607  &  0.9654  &  0.6074 \\
         0      &   0  &  1.0000
        \end{bmatrix}, \\
        T &= \begin{bmatrix}[r]
        0.9659  &  0.2591 &  -0.8698 \\
        -0.2591  &  0.9659  &  -0.4049 \\
         0      &   0  &  1.0000
        \end{bmatrix}, \\
        T\cdot A &= \begin{bmatrix}[r]
        1.0000  & -0.0017 &  -0.0122 \\
        0.0017  &  1.0000  & -0.0060 \\
         0  &       0  &  1.0000
        \end{bmatrix}, \\
        \lVert T\cdot A\rVert_F &= 0.0138.
    \end{split}
\end{equation}

The parameters used in this registration are included in Table \ref{tab:parameters_se2}. The initial and final images are in Figure \ref{fig:clouds_se2_registration}. For each cloud, the color information is generated based on the level set numbers.
\end{example}

\begin{table}[t]
\begin{center}
\footnotesize
\begin{tabular}{l r}
\toprule
    Parameters & Value \\
        \midrule
        transformation convergence threshold $\epsilon$ & $1\mathrm{e}{-4}$ \\
        gradient norm convergence threshold $\epsilon$ & $5\mathrm{e}{-4}$ \\
        kernel characteristic length-scale $\ell$ & $0.25$ \\
        kernel characteristic length-scale $\ell$ (iteration $> 3$)  & $0.15$ \\
        kernel characteristic length-scale $\ell$ (iteration $> 10$)  & $0.10$ \\
        kernel characteristic length-scale $\ell$ (iteration $> 20$)  & $0.05$ \\
        kernel signal variance $\sigma$ & $1$  \\
        minimum step length & $0.2$ \\
        color space inner product scale & $1\mathrm{e}{-5}$ \\
        kernel sparsification threshold & $1\mathrm{e}{-3}$ \\
\bottomrule
\end{tabular}
\caption{ Parameters used for evaluation of the bird's eye map registration. The kernel characteristic length-scale is chosen to be adaptive as the algorithm converges; intuitively, we prefer a large neighborhood of correlation for each point, but as the algorithm reaches the convergence reducing the local correlation neighborhood allows for faster convergence and better refinement.}
\label{tab:parameters_se2}
\end{center}
\squeezeup\squeezeup
\end{table}

\subsection{Special Case: $n=3$}
\label{subsec:se3}

When we restrict attention to $\SE(3)$, the gradient \eqref{eq:flow} takes a special form:
\begin{equation}\label{eq:3_gradient}
\begin{split}
\omega &= \frac{1}{a^2\ell^2} \, \sumxz c_{ij} k(x_i,h^{-1}z_j)\left(x_i\times(h^{-1}z_j)\right) \\
v &= \frac{1}{b^2\ell^2}\, \sumxz c_{ij} k(x_i,h^{-1}z_j)\left( h^{-1}z_j - x_i\right).
\end{split}
\end{equation}
Additionally, an explicit formula for the exponential map $\exp:\se(3)\to\SE(3)$ exists (see \cite{1104.1106} and \cite{rohan2013} for example). This gives an exact way to solve \eqref{eq:newstep}.
\begin{equation}\label{eq:step}
\begin{split}
\Delta R &= I + \left( \frac{\sin t\theta}{\theta}\right)\hat{\omega} + 
\left(\frac{1-\cos t\theta}{\theta^2}\right)\hat{\omega}^2,\\
\Delta T &= \left[ t I + \left( \frac{1-\cos t\theta}{\theta^2}\right)\hat{\omega} + \left(
\frac{t\theta - \sin t\theta}{\theta^3}\right)\hat{\omega}^2
\right] v.
\end{split}
\end{equation}
where $\theta = \lVert \omega \rVert_3$ with $\omega\in\mathbb{R}^3$ and $t$ is taken to maximize $G(t)$ in equation \eqref{eq:fourth}.

\section{Riemannian Geometry of the Special Euclidean Groups}
\label{sec:RgeomSEn}
This section deals with the geometry of $\SE(2)$ and $\SE(3)$. Specifically, the Riemannian exponential corresponding to the metric \eqref{eq:metric} (and the analogous version for $\SE(2)$) as well as the Hessians will be computed. Additionally, there will be a discussion on using Newton's method as an alternate update rule.
\subsection{$\SE(2)$}
We begin with the simpler case of $\SE(2)$. We will compute the exponential first and the Hessian second.
\subsubsection{The Riemann Exponential}
To compute the Riemann exponential, we will use the Euler-Poincar\'{e} equations as described by Theorem \ref{th:EP}. The Lagrangian will be $\mathcal{L}:\mathfrak{se}(2)\to\mathbb{R}$ where
\begin{equation*}
    \mathcal{L}(\omega;v_1,v_2) = \frac{1}{2}a^2\omega^2 + \frac{1}{2}b^2\left(v_1^2+v_2^2\right).
\end{equation*}
The Euler-Poincar\'{e} equations are then
\begin{equation*}
    \begin{split}
        \dot{\omega} &= 0, \\
        \dot{v}_1 &= \omega\cdot v_2, \\
        \dot{v}_2 &= -\omega\cdot v_1.
    \end{split}
\end{equation*}
These can be integrated to get a path $(\omega(t);v_1(t),v_2(t))\in\mathfrak{se}(2)$. However, this only defines a path in the Lie algebra and it needs to be lifted to the group. Doing so results in the following system of differential equations:
\begin{equation*}\begin{split}
    \dot{R}(t) &= R(t)\cdot\omega(t), \\
    \dot{T}(t) &= R(t)\cdot\begin{bmatrix}
    v_1(t) \\ v_2(t)
    \end{bmatrix}.
\end{split}\end{equation*}
The rotation part gives the same answer as the Lie exponential while cancellation occurs in the translation part:
\begin{equation*}
    \dot{T} = R(t) \cdot\begin{bmatrix}
    v_1(t) \\ v_2(t)
    \end{bmatrix} = R(t)\cdot R(t)^{-1} \cdot\begin{bmatrix}
    v_1(0) \\ v_2(0)
    \end{bmatrix}.
\end{equation*}
This provides us with the Riemann exponential (at the identity), $\mathrm{Exp}_0:\mathfrak{se}(2)\to\SE(2)$
\begin{equation*}
    \mathrm{Exp}_0\left( \begin{bmatrix}[r]
    0 & -\omega & v_1 \\
    \omega & 0 & v_1 \\
    0 & 0 & 0 \end{bmatrix} \right) = \begin{bmatrix}[r]
    \cos\omega & -\sin\omega & v_1 \\
    \sin\omega & \cos\omega & v_2 \\
    0 & 0 & 1
    \end{bmatrix}.
\end{equation*}
\begin{remark}\label{rmk:exp}
    It is interesting to compare the Riemann and Lie exponential for $\SE(2)$. When we restrict to $\SO(2)$, both agree. This follows from the fact that $\SO(2)$ is a compact Lie group and the metric restricted to this subgroup is bi-invariant. The second interesting phenomena is that the translation part is the identity for the Riemann case while the Lie case is more involved. The reason for this is that in the metric there are no cross terms intertwining rotation with translation, i.e. the Riemann exponential \textit{is} the Lie exponential for the group $\SO(2)\times\mathbb{R}^2$ as opposed to $\SE(2) = \SO(2)\ltimes\mathbb{R}^2$.
\end{remark}
\begin{figure*}[t]
    \centering
    \subfloat{\includegraphics[width=0.3\columnwidth]{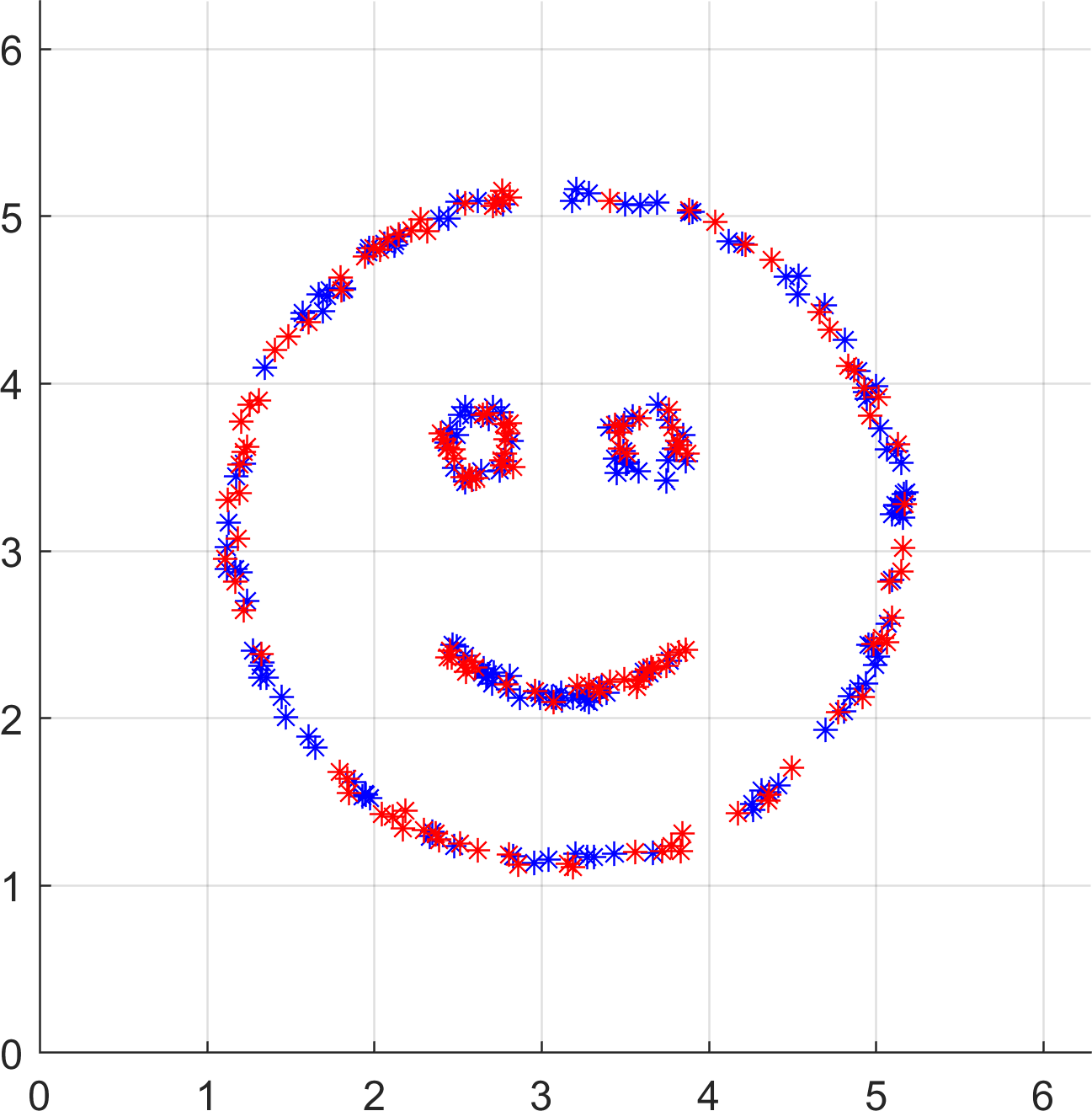}}\hfill
    \subfloat
    {\includegraphics[width=0.3\columnwidth]{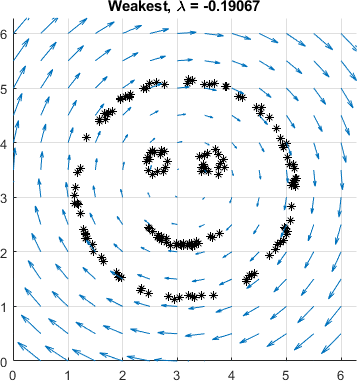}}\hfill
    \subfloat
    {\includegraphics[width=0.3\columnwidth]{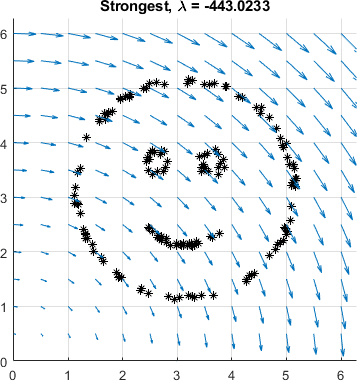}}
    \caption{Left: The images of the two point clouds. The blue stars represent $X$ while the red are $Z$. Center: The vector field corresponding to the eigenvector of the Hessian with the weakest eigenvalue. Right: The vector field corresponding to the eigenvector of the Hessian with the strongest eigenvalue.}
    \label{fig:2faces}
\end{figure*}

\subsubsection{The Hessian}
To calculate the Hessian for the $\SE(2)$ case, we first choose an orthogonal basis $\{e_z,e_1,e_2\}\in\mathfrak{se}(2)$ where
\begin{equation*}
    e_z = \begin{bmatrix}[r]
    0 & -1 & 0 \\ 1 & 0 & 0 \\ 0 & 0 & 0
    \end{bmatrix},\quad e_1 = \begin{bmatrix}[r]
    0 & 0 & 1 \\ 0 & 0 & 0 \\ 0 & 0 & 0
    \end{bmatrix}, \quad e_2 = \begin{bmatrix}[r]
    0 & 0 & 0 \\ 0 & 0 & 1 \\ 0 & 0 & 0
    \end{bmatrix}.
\end{equation*}
Additionally, we will take $\{E_z,E_1,E_2\}$ to be the corresponding left-invariant vector fields. The connection is given by 
\begin{equation*}
    \begin{split}
        \nabla_{E_z}E_1 &= E_2, \\
        \nabla_{E_z}E_2 &= -E_1,
    \end{split}
\end{equation*}
while all other combinations are zero. The gradient in these coordinates is $\omega e_z + v_1e_1 + v_2e_2$. The Hessian in these coordinates is
\begin{equation}\label{eq:formSE2Hess}
-H_{\SE(2)} = \begin{bmatrix}
a^2\mathcal{L}_{E_z}\omega & b^2\mathcal{L}_{E_z}v^1 + b^2v^2 & b^2\mathcal{L}_{E_z}v^2 - b^2v^1 \\
a^2\mathcal{L}_{E_1}\omega & b^2\mathcal{L}_{E_1}v^1 & b^2\mathcal{L}_{E_1}v^2 \\
a^2\mathcal{L}_{E_2}\omega & b^2\mathcal{L}_{E_2}v^1 & b^2\mathcal{L}_{E_2}v^2
\end{bmatrix}.
\end{equation}
For the sake of simplicity, let the Hessian be
\begin{equation*}
    H_{\SE(2)} = \frac{1}{\ell^2} \, \sumxz \, c_{ij} k(x_i,\tilde{z}_j) \mathcal{S}_{ij},
\end{equation*}
where
\begin{equation*}
    \mathcal{S}_{ij} = \begin{bmatrix}
    \frac{1}{\ell^2}\left(x_i\tilde{\times}\tilde{z}_j\right)^2 - \langle x_i,\tilde{z}_j\rangle & 
    x_i^2 + \frac{1}{\ell^2}(\tilde{z}_j^1-x_i^1)(x_i\tilde{\times}\tilde{z}_j) & 
    -x_i^1 + \frac{1}{\ell^2}(\tilde{z}_j^2-x_i^2)(x_i\tilde{\times}\tilde{z}_j) \\
    x_i^2 + \frac{1}{\ell^2}(\tilde{z}_j^1-x_i^1)(x_i\tilde{\times}\tilde{z}_j)  &  
    \frac{1}{\ell^2}(\tilde{z}_j^1 - x_i^1)^2-1 & 
    \frac{1}{\ell^2} (x_i^1-\tilde{z}_j^1)(x_i^2-\tilde{z}_j^2) \\
    -x_i^1 + \frac{1}{\ell^2}(\tilde{z}_j^2-x_i^2)(x_i\tilde{\times}\tilde{z}_j) & 
    \frac{1}{\ell^2} (x_i^1-\tilde{z}_j^1)(x_i^2-\tilde{z}_j^2) &
    \frac{1}{\ell^2}(\tilde{z}_j^2 - x_i^2)^2-1
    \end{bmatrix}.
\end{equation*}
\begin{example}[Hessian Eigenvalues] \label{ex:se2Hess}
As an initial thought experiment, what happens if each cloud contains only a single point? As we are interested in the registration problem, let us assume that each of these points is at the same location $x\in\mathbb{R}^2$. Then, the Hessian is
\begin{equation*}
    H = \frac{1}{\ell^2} c_{11} \begin{bmatrix}[r]
    -\lVert x\rVert^2 & x_2 & -x_1 \\
    x_2 & -1 & 0 \\
    -x_1 & 0 & -1
    \end{bmatrix}.
\end{equation*}
It is interesting to note that this matrix is singular and its kernel is the span of $[-1,-x_2,x_1]^\transpose$. This Lie algebra element corresponds to a vector field on $\mathbb{R}^2$ (the infinitesimal generator, see Definition \ref{def:infgen}). The corresponding vector field is 
\begin{equation*}
    \xi(v) = \begin{bmatrix}
    v_2 - x_2 \\ -v_1 + x_1
    \end{bmatrix},
\end{equation*}
which is merely rotations about the point $x$. This rotation does not move either cloud and the degeneracy of $H$ predicts this!

In the more general context, the eigenvalues/vectors of $H$ carry much information about the structure and symmetry of the problem. At the solution, the Hessian is negative-definite and so the eigenvector corresponding to the largest eigenvalue (smallest in absolute value) reports on the direction of ``greatest symmetry.'' That is, suppose that $H\xi = \lambda\xi$ for $\xi\in\mathfrak{g}$ and $|\lambda|$ is ``small.'' Then moving $Z$ along the vector field $\xi_M$ is not noticeable under $F$. To see this, we consider the case where $X$ and $Z$ are images of a smiley face as in Figure \ref{fig:2faces}. The Hessian for this example is
\begin{equation*}
H = \begin{bmatrix}[r]
    -422.1054 &  55.3179 & -75.8545 \\
   55.3179 & -16.1852  &  0.8360 \\
  -75.8545  &  0.8360 & -23.3209
\end{bmatrix}.
\end{equation*}
The weakest eigenvalue is $-0.1907$ with a corresponding eigenvector of $$\xi = [-0.2140,-0.7048,0.6764]^\transpose.$$
The vector field for this eigenvector is shown in Figure \ref{fig:2faces}.
The impressive aspect of this is that it is possible to learn symmetries in images!
\end{example}

\subsection{$\SE(3)$}
This section will closely mimic the previous; however it will be more involved. 
\subsubsection{The Riemann Exponential}
The Euler-Poincar\'{e} equations for $\SE(3)$ are
\begin{equation*}
\begin{split}
    \dot{\omega} &= 0, \\
    \dot{v} &= -\omega\times v.
\end{split}
\end{equation*}
Integrating these in conjunction with the reconstruction equations produce the Lie exponential for the rotation part and the identity for the translations; the same as the $\SE(2)$ case above. That is,
\begin{equation*}
    \mathrm{Exp}_0(R,T) = \left( \exp(R),T\right).
\end{equation*}
Remark \ref{rmk:exp} holds in this case as well. Moreover, this holds for the higher-dimensional $\SE(n)$ as well.
\subsubsection{The Hessian}
To compute the Hessian, we choose the following orthogonal coordinates for $\mathfrak{se}(3)$:
\begin{align*}
    \begin{array}{llllll}
        e_x &= \begin{bmatrix}[r]
        0 & 0 & 0 & 0 \\
        0 & 0 & -1 & 0 \\
        0 & 1 & 0 & 0 \\
        0 & 0 & 0 & 0
        \end{bmatrix}, & e_y = \begin{bmatrix}[r]
        0 & 0 & 1 & 0 \\
        0 & 0 & 0 & 0 \\
        -1 & 0 & 0 & 0 \\
        0 & 0 & 0 & 0
        \end{bmatrix}, & e_z = \begin{bmatrix}[r]
        0 & -1 & 0 & 0 \\
        1 & 0 & 0 & 0 \\
        0 & 0 & 0 & 0 \\
        0 & 0 & 0 & 0
        \end{bmatrix}, \\
        e_1 &= \begin{bmatrix}[r]
        0 & 0 & 0 & 1 \\
        0 & 0 & 0 & 0 \\
        0 & 0 & 0 & 0 \\
        0 & 0 & 0 & 0
        \end{bmatrix}, & e_2 = \begin{bmatrix}[r]
        0 & 0 & 0 & 0 \\
        0 & 0 & 0 & 1 \\
        0 & 0 & 0 & 0 \\
        0 & 0 & 0 & 0
        \end{bmatrix}, & e_3 = \begin{bmatrix}[r]
        0 & 0 & 0 & 0 \\
        0 & 0 & 0 & 0 \\
        0 & 0 & 0 & 1 \\
        0 & 0 & 0 & 0
        \end{bmatrix}.
    \end{array}
\end{align*}
Likewise, let $\{E_x,E_y,E_z,E_1,E_2,E_3\}$ be their corresponding left-invariant vector fields. With this, the connection is given by
\begin{align}
\label{eq:connection_coef_se3}
\nonumber \nabla_{E_x}E_y &= \frac{1}{2}E_z, \nabla_{E_x}E_z = -\frac{1}{2}E_y, \nabla_{E_y}E_z = \frac{1}{2}E_x, \\
\nonumber \nabla_{E_y}E_x &= -\frac{1}{2}E_z, \nabla_{E_z}E_x = \frac{1}{2}E_y, \nabla_{E_z}E_y = -\frac{1}{2}E_x,\\
\nonumber \nabla_{E_x}E_2 &= E_3, \nabla_{E_x}E_3 = -E_2, \nabla_{E_y}E_3 = E_1, \\ 
\nabla_{E_y}E_1 &= -E_3, \nabla_{E_z}E_1 = E_2, \nabla_{E_z}E_2 = -E_1,
\end{align}
while all other combinations are zero. 
\begin{remark}
Notice that the coefficients for the connection in \eqref{eq:connection_coef_se3} do not depend on either $a$ or $b$! This is to be expected because the Levi-Civita connection needs to be torsion free. This means that $\nabla_XY-\nabla_YX=[X,Y]$. Since the coordinates we are taking are the left-invariant vector fields, their bracket is just the bracket of the Lie algebra which does not depend on $a$ nor $b$.
\end{remark}
We are now ready to write down the Hessian matrix (we will assume that $a=b=1$).

\begin{align}\label{eq:mSE3_Hessian}
    \nonumber &-H_{\SE(3)} = \\ 
    \nonumber &\left[\begin{array}{ccc;{2pt/2pt}ccc}
    \mathcal{L}_{E_x}\omega_x & \mathcal{L}_{E_x}\omega_y + \frac{1}{2}\omega_z & \mathcal{L}_{E_x}\omega_z - \frac{1}{2}\omega_y & \mathcal{L}_{E_x}v_1 & \mathcal{L}_{E_x}v_2 + v_3 & \mathcal{L}_{E_x}v_3 - v_2 \\
    \mathcal{L}_{E_y}\omega_x - \frac{1}{2}\omega_z & \mathcal{L}_{E_y}\omega_y & \mathcal{L}_{E_y}\omega_z + \frac{1}{2}\omega_x & \mathcal{L}_{E_y}v_1 - v_3 & \mathcal{L}_{E_y}v_2 & \mathcal{L}_{E_y}v_3 + v_1 \\
    \mathcal{L}_{E_z}\omega_x + \frac{1}{2}\omega_y & \mathcal{L}_{E_z}\omega_y - \frac{1}{2}\omega_x & \mathcal{L}_{E_z}\omega_z & \mathcal{L}_{E_z}v_1 + v_2 & \mathcal{L}_{E_z}v_2-v_1 & \mathcal{L}_{E_z}v_3 \\ \hdashline[2pt/2pt]
    \mathcal{L}_{E_1}\omega_x & \mathcal{L}_{E_1}\omega_y & \mathcal{L}_{E_1}\omega_z & \mathcal{L}_{E_1}v_1 & \mathcal{L}_{E_1}v_2 & \mathcal{L}_{E_1}v_3 \\
    \mathcal{L}_{E_2}\omega_x & \mathcal{L}_{E_2}\omega_y & \mathcal{L}_{E_2}\omega_z & \mathcal{L}_{E_2}v_1 & \mathcal{L}_{E_2}v_2 & \mathcal{L}_{E_2}v_3 \\
    \mathcal{L}_{E_3}\omega_x & \mathcal{L}_{E_3}\omega_y & \mathcal{L}_{E_3}\omega_z & \mathcal{L}_{E_3}v_1 & \mathcal{L}_{E_3}v_2 & \mathcal{L}_{E_3}v_3 \\
    \end{array}\right] \\
    &=: -\left[ \begin{array}{c;{2pt/2pt}c}
    \mathcal{A} & \mathcal{B} \\ \hdashline[2pt/2pt]
    \mathcal{C} & \mathcal{D} \end{array} \right].
\end{align}
Notice that the first 3 by 3 block, $\mathcal{A}$, is identical to \eqref{eq:Hsphere}. This reflects the copy of $\SO(3)$ lying inside $\SE(3)$.

For the sake of brevity, we will use the following abbreviation:
\begin{equation*}
    H_{\SE(3)} = \frac{1}{\ell^2} \, \sumxz \, c_{ij}\cdot k(x_i,\tilde{z}_j)\cdot \left[ \begin{array}{c;{2pt/2pt}c}
    \mathcal{A}_{ij} & \mathcal{B}_{ij} \\ \hdashline[2pt/2pt]
    \mathcal{C}_{ij} & \mathcal{D}_{ij} \end{array} \right].
\end{equation*}
Notice that by symmetry of $H$, $\mathcal{B}_{ij}^\transpose = \mathcal{C}_{ij}$. Therefore, we will only compute $\mathcal{A}_{ij}$, $\mathcal{C}_{ij}$, and $\mathcal{D}_{ij}$. To simplify expressions, we will call $v = [v_1,v_2,v_3]^\transpose := x_i\times\tilde{z}_j$ and $u = [u_1,u_2,u_3]^\transpose := \tilde{z}_j-x_i$.
\begin{equation*}
    \mathcal{A}_{ij} = \begin{bmatrix} 
    \mathcal{A}_{ij}^1 & \mathcal{A}_{ij}^2 & \mathcal{A}_{ij}^3
    \end{bmatrix} %
\end{equation*}
\begin{equation*}
    \mathcal{A}_{ij}^1 = 
    \begin{bmatrix}
    \frac{1}{\ell^2} v_1^2 - (x_i^2\tilde{z}_j^2+x_i^3\tilde{z}_j^3) \\
    \frac{1}{\ell^2}v_1v_2 + \frac{1}{2}(x_i^1\tilde{z}_j^2+x_i^2\tilde{z}_j^1) \\
    \frac{1}{\ell^2}v_1v_3 + \frac{1}{2}(x_i^1\tilde{z}_j^3+x_i^3\tilde{z}_j^1) 
    \end{bmatrix} 
\end{equation*}
\begin{equation*}
    \mathcal{A}_{ij}^2 = 
    \begin{bmatrix}
    \frac{1}{\ell^2}v_1v_2 + \frac{1}{2}(x_i^1\tilde{z}_j^2+x_i^2\tilde{z}_j^1) \\
    \frac{1}{\ell^2} v_2^2 - (x_i^1\tilde{z}_j^1+x_i^3\tilde{z}_j^3) \\
    \frac{1}{\ell^2}v_2v_3 + \frac{1}{2}(x_i^2\tilde{z}_j^3+x_i^3\tilde{z}_j^2) 
    \end{bmatrix} 
\end{equation*}
\begin{equation*}
    \mathcal{A}_{ij}^3 = 
    \begin{bmatrix}
    \frac{1}{\ell^2}v_1v_3 + \frac{1}{2}(x_i^1\tilde{z}_j^3+x_i^3\tilde{z}_j^1) \\
    \frac{1}{\ell^2}v_2v_3 + \frac{1}{2}(x_i^2\tilde{z}_j^3+x_i^3\tilde{z}_j^2) \\
    \frac{1}{\ell^2} v_3^2 - (x_i^1\tilde{z}_j^1+x_i^2\tilde{z}_j^2)
    \end{bmatrix}
\end{equation*}
\begin{equation*}
    \mathcal{C}_{ij} = \begin{bmatrix}
    \frac{1}{\ell^2}u_1v_1 &
    -x_i^3 + \frac{1}{\ell^2}u_1v_2 &
    x_i^2 + \frac{1}{\ell^2} u_1v_3 \\
    x_i^3 + \frac{1}{\ell^2} u_2v_1 &
    \frac{1}{\ell^2}u_2v_2 &
    -x_i^1 + \frac{1}{\ell^2} u_2v_3 \\
    -x_i^2 + \frac{1}{\ell^2} u_3v_1 &
    x_i^1 + \frac{1}{\ell^2} u_3v_2 &
    \frac{1}{\ell^2}u_3v_3
    \end{bmatrix} \qquad
\end{equation*}
\begin{equation*}
    \mathcal{D}_{ij} = \begin{bmatrix}
    \frac{1}{\ell^2}u_1^2-1 & \frac{1}{\ell^2}u_1 u_2 &
    \frac{1}{\ell^2}u_1 u_3 \\
    \frac{1}{\ell^2}u_1 u_2 &
    \frac{1}{\ell^2}u_2^2-1 &
    \frac{1}{\ell^2}u_2 u_3 \\
    \frac{1}{\ell^2}u_1 u_3 &
    \frac{1}{\ell^2}u_2 u_3 &
    \frac{1}{\ell^2}u_3^2-1
    \end{bmatrix} .
\end{equation*}

\begin{table}[b]
\begin{center}
\resizebox{\columnwidth}{!}{
\begin{tabu} { c | c c c c c c}
 Eigenvalue & -613.9730 & -598.4636 & -547.6175 & -8.1304 & -0.2898 & 3.5950 \\ \midrule 
 Eigenvector & $\begin{bmatrix}[r]
    0.0000  \\  0.0022  \\  0.0021  \\ -0.5827 \\  0.6858  \\  0.4360
    \end{bmatrix}$ & $\begin{bmatrix}[r]
   -0.0032  \\  0.0016 \\  -0.0015  \\  0.6897  \\  0.7011 \\  -0.1810
   \end{bmatrix}$ & $\begin{bmatrix}[r]
    0.0033  \\  0.0027 \\   0.0010 \\  -0.4298  \\  0.1952 \\  -0.8815
    \end{bmatrix}$ & $\begin{bmatrix}[r]
    0.4029  \\  0.9089 \\   0.1073  \\  0.0030  \\ -0.0023 \\   0.0025
    \end{bmatrix}$ & $\begin{bmatrix}[r]
    0.2685  \\ -0.2295  \\  0.9355 \\   0.0032  \\  0.0006 \\  -0.0000
    \end{bmatrix}$ & $\begin{bmatrix}[r]
   -0.8750  \\  0.3481  \\  0.3365 \\  -0.0019  \\ -0.0027 \\  -0.0015
   \end{bmatrix}$ 
\end{tabu}}
\end{center}
\caption{The eigenvalue/vectors of $H$ in \eqref{eq:H_se3_sphere}.}
\label{tab:sphere_eig}
\end{table}

\begin{example}[Hessian Eigenvalues, Cont.]
The same analysis as presented in Example \ref{ex:se2Hess} can be carried out for $\SE(3)$ as well. Again, we will use the Hessian to identify symmetries within the point clouds. 
Specifically, we will consider the case where both $X$ and $Z$ are a sphere living in $\mathbb{R}^3$. The spheres can be aligned via translations but rotations do nothing. This should manifest in the Hessian; the block $\mathcal{D}$ in \eqref{eq:mSE3_Hessian} should dominate the Hessian while $\mathcal{A}$ should be the smallest.

In what follows, $X$ contains 1200 points on $S^2\in\mathbb{R}^3$ and $Z$ contains 1500 points (both of these clouds are already on $S^2$ so the registration problem is already solved) and all $c_{ij}=3$. The remaining parameters are $\ell=0.25$ and $\sigma = 10^{-4}$. The Hessian is
\begin{equation}\label{eq:H_se3_sphere}
    H = \left[\begin{array}{rrr;{2pt/2pt}rrr}
    1.3995 &  -4.0561 &  -1.4875  & 2.1207 &   0.9970  &  1.2481 \\
   -4.0561 &  -6.3054 &  -0.3125  & 0.7368 &  -1.8921  &  0.8737 \\
   -1.4875 &  -0.3125 &   0.0554  & 1.5803 &  -0.3814  & -0.2286 \\ \hdashline[2pt/2pt]
    2.1207 &   0.7368 &   1.5803  & -594.3227 & 1.9310    &  23.2125 \\
    0.9970 &  -1.8921 &  -0.3814  &   1.9310  & -603.8121 & -13.3898 \\
    1.2481 &   0.8737 &  -0.2286  &  23.2125  &  -13.3898 & -561.8939
    \end{array}\right].
\end{equation}
It can be seen that the translation component is on the order of 100 times larger than the rotation component. More formally, we can examine the eigenvalues/vectors of $H$. 

Table~\ref{tab:sphere_eig} shows that the eigenvalues corresponding to rotations are around 100 times smaller than the eigenvalues corresponding to translations. As a result, the Hessian can report on the regularity of the solution / identify infinitesimal symmetries in the pictures.

\end{example}

\section{Experimental Results: RGB-D Visual Odometry}
\label{sec:results}

The RGB-D visual odometry problem is illustrated in Figure~\ref{fig:rgbd-vo}. In the context of visual odometry, we call our method Continuous Visual Odometry (CVO), which has appeared in our earlier work~\citep{MGhaffari-RSS-19,lin2019adaptive}. We compare CVO with the state-of-the-art direct (and dense) RGB-D visual odometry (DVO)~\citep{kerl2013dense,kerl2013robust}. Since the original DVO source code requires outdated ROS dependency~\citep{Kerl2013DVOrepo}, we reproduced DVO results using the version provided by Matthieu Pizenberg~\citep{MatthieuP2019DVO}, which only removes the dependency for ROS while maintaining the DVO core source code unchanged. 

\begin{figure}[t]
    \centering
    \includegraphics[width=0.55\columnwidth]{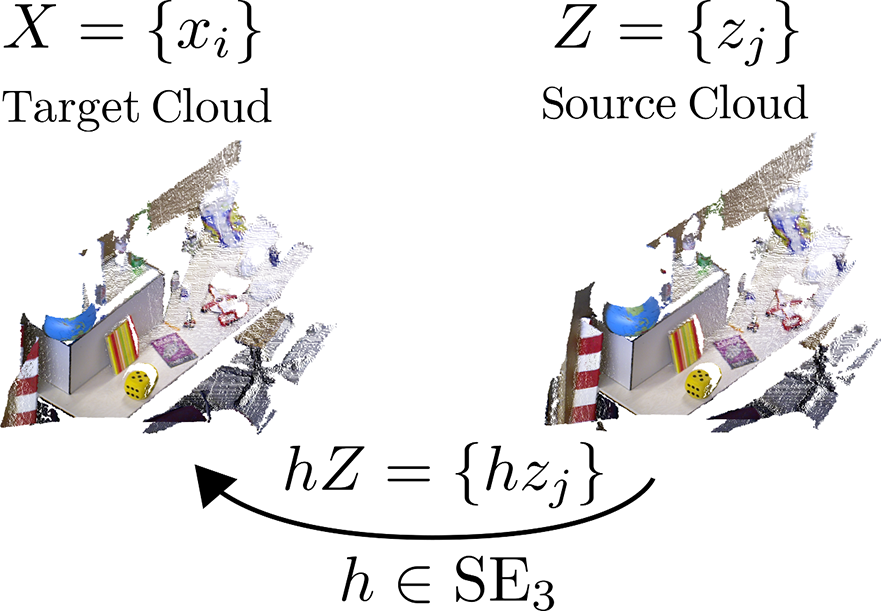}
    \caption{Visual odometry using depth cameras is the problem of finding a rigid body transformation between two colored point clouds.}
    \label{fig:rgbd-vo}
\end{figure}

\begin{table}[t]
\centering
\footnotesize
\begin{tabular}{lcr}
\toprule
    Parameters & Symbol & Value \\
        \midrule
        Transformation convergence threshold & $\epsilon_1$ & $1\mathrm{e}{-5}$ \\
        Gradient norm convergence threshold & $\epsilon_2$ & $5\mathrm{e}{-5}$ \\
        Minimum step length & & $0.2$ \\
        Kernel sparsification threshold & & $8.315\mathrm{e}{-3}$ \\
        Spatial kernel initial length-scale & $\ell_{init}$ & $0.1$ \\
        Spatial kernel signal variance & $\sigma$ & $0.1$  \\
        Color kernel length-scale & $\ell_c$ & $0.1$ \\
        Color kernel signal variance & $\sigma_c$ & $1$  \\
\bottomrule
\end{tabular}
\caption{ Parameters used for evaluation using TUM RGB-D Benchmark, similar values are chosen for all experiments. The kernel characteristic length-scale is chosen to be adaptive as the algorithm converges~\citep{MGhaffari-RSS-19}; intuitively, we prefer a large neighborhood of correlation for each point, but as the algorithm reaches the convergence reducing the local correlation neighborhood allows for faster convergence and better refinement.}
\label{tab:parameters}
\end{table}

\begin{table}[t]
    \centering
    \footnotesize 
    \begin{tabular}{lcccc|cccc}
        \toprule
         & \multicolumn{4}{c}{Training} & \multicolumn{4}{c}{Validation} \\
        \midrule
          & \multicolumn{2}{c}{CVO} & \multicolumn{2}{c}{DVO} & \multicolumn{2}{c}{CVO} &  \multicolumn{2}{c}{DVO} \\
         Sequence & Trans. & Rot. & Trans. & Rot. & Trans. & Rot. & Trans. & Rot. \\
        \midrule
        fr1/desk    & 0.0486 & 2.4860 & \bf 0.0387 & \bf 2.3589 & 0.0401 & \bf 2.0148 & \bf 0.0371 & 2.0645  \\
        fr1/desk2   & \bf 0.0535 & \bf 3.0383 & 0.0583 & 3.6529 & 0.0225 & 1.7691 & \bf 0.0208 & \bf 1.7416  \\ 
        fr1/room    & 0.0560 & \bf 2.4566 & \bf 0.0518 & 2.8686 & \bf 0.0446 & \bf 3.9183 & 0.2699 & 7.4144  \\
        fr1/360     & \bf 0.0991 & \bf 3.0025 & \bf 0.1602 & 4.4407 & \bf 0.1420 & \bf 3.0746 & 0.2811 & 7.0876  \\
        fr1/teddy   & \bf 0.0671 & 4.8089 & 0.0948 & \bf 2.5495 & n/a    & n/a    & n/a    & n/a  \\
        fr1/floor   & 0.0825 & 2.3745 & \bf 0.0635 & \bf 2.2805 & n/a    & n/a    & n/a    & n/a  \\
        fr1/xyz     & \bf 0.0240 & \bf 1.1703 & 0.0327 & 1.8751 & \bf 0.0154 & \bf 1.3872 & 0.0453 & 3.0061 \\
        fr1/rpy     & 0.0457 & 3.3073 & \bf 0.0336 & \bf 2.6701 & \bf 0.1138 & \bf 3.6423 & 0.3607 & 7.9991  \\
        fr1/plant   & 0.0316 & 1.9973 & \bf 0.0272 & \bf 1.5523 & \bf 0.0630 & 4.9185 & 0.0660 & \bf 2.5865  \\
        \midrule
        Average & \bf 0.0561 & 2.7380 & 0.0623 & \bf 2.6943 & \bf 0.0631 & \bf 2.9607 & 0.1544 & 4.5571 \\
        \bottomrule
    \end{tabular}
    \caption{The RMSE of Relative Pose Error (RPE) for \texttt{fr1} sequences. The trans. columns show the RMSE of the translational drift in $\mathrm{m}/\sec$ and the rot. columns show the RMSE of the rotational error in $\mathrm{deg}/\sec$. There's no corresponding validation data sets for \texttt{fr1/teddy} and \texttt{fr1/floor}.}
    \label{tab:fr1_results}
\end{table}

To improve the computational efficiency, we adopted a similar approach to Direct Sparse Visual Odometry (DSO) by~\citet{engel2014lsd} to create a semi-dense point cloud (around 3000 points) for each scan. To prevent insufficient points being selected in environments that lack rich visual information, we also used a Canny edge detector~\citep{canny1987computational} from OpenCV~\citep{opencv_library}. When the number of points selected by the DSO point selector is less than one-third of the desired number of points, more points will be selected by downsampling the pixels highlighted by the Canny detector. While generating the point cloud, RGB values are first transformed into the HSV colormap and normalized. The normalized HSV values are then combined with the normalized intensity gradients and used as the labels of the selected points in the color space. For all experiments, we used the same set of parameters, which are listed in Table~\ref{tab:parameters}.

All experiments are performed on a Dell XPS15 9750 laptop with Intel i7-8750H CPU (6 cores with 2.20 GHz each) and 32GB RAM. The source code is implemented in C++ and compiled with the Intel Compiler. The kernel computations are parallelized using the Intel Threading Building Blocks (TBB)~\citep{intel2019tbb}. Using compiler auto-vectorization and the parallelization, the average time for frame-to-frame registration is 0.5 $\sec$. Software is available for download at \href{https://github.com/MaaniGhaffari/cvo-rgbd}{\url{https://github.com/MaaniGhaffari/cvo-rgbd}}.

\subsection{TUM RGB-D Benchmark}

We performed experiments on two parts of \mbox{RGB-D} SLAM data set and benchmark by the Technical University of Munich~\citep{sturm2012benchmark}. This data set was collected indoors with a Microsoft Kinect using a motion capture system as a proxy for ground truth trajectory. For all tracking experiments, the entire images were used sequentially without any skipping, i.e., at full frame rate. We evaluated CVO and DVO on the training and validation sets for all the \texttt{fr1} sequences and the structure versus texture sequences. \mbox{RGB-D} benchmark tools~\citep{sturm2012benchmark} were then used to evaluate the Relative Pose Error (RPE) of both methods.

Table~\ref{tab:fr1_results} shows the Root-Mean-Squared Error (RMSE) of the RPE for \texttt{fr1} sequences. The Trans. columns show the RMSE of the translational drift in $\mathrm{m}/\sec$ and the Rot. columns show the RMSE of the rotational drift in $\mathrm{deg}/\sec$. There are no corresponding validation sequences for \texttt{fr1/teddy} and \texttt{fr1/floor}. On the validation set, CVO has improved performance compared with DVO which shows it can generalize across different scenarios better. From the results, we can see that CVO is intrinsically robust. The next experiment will further reveal that CVO has the advantage of performing well in extreme environments that lack rich structure or texture.

\subsection{Structure vs. Texture Sequences}

Table~\ref{tab:fr3_results} shows the RMSE of RPE for the structure vs. texture sequences. This data set contains image sequences in structure/nostructure and texture/notexture environments. As elaborated by~\citet{MGhaffari-RSS-19}, by treating point clouds as points in the function space (RKHS), CVO is inherently robust to the lack of features in the environment. CVO shows the best performance on cases where either structure or texture is not rich in the environment. This reinforces the claim that CVO is robust to such scenes.

\begin{table}
    \centering
    \footnotesize 
    \begin{tabular}{lllcccc|cccc}
        \toprule
         & & & \multicolumn{4}{c}{Training} & \multicolumn{4}{c}{Validation} \\
        \midrule
           \multicolumn{3}{c}{Sequence}  & \multicolumn{2}{c}{CVO} & \multicolumn{2}{c}{DVO} & \multicolumn{2}{c}{CVO} & \multicolumn{2}{c}{DVO} \\
         \multicolumn{3}{l}{structure-texture-dist.} & Trans. & Rot. & Trans. & Rot. & Trans. & Rot. & Trans. & Rot. \\
        \midrule
        $\times$   & \checkmark & near & \bf 0.0279 & \bf 1.3470 &  0.0563 & 1.7560 & \bf 0.0310 & 1.6367 & 0.0315 & \bf1.1498  \\
        $\times$   & \checkmark & far  & \bf 0.0609 & \bf 1.2342 &  0.1612 & 3.4135 & \bf 0.1374 & \bf 2.3929 & 0.5351 & 8.2529 \\
        \checkmark & $\times$   & near & \bf 0.0221 & \bf 1.3689 &  0.1906 & 10.6424 & \bf 0.0465 & \bf 2.0359 & 0.1449 & 4.9022 \\
        \checkmark & $\times$   & far  & \bf 0.0372 & \bf 1.3061 &  0.1171 & 2.4044 & \bf 0.0603 & \bf 1.8142 & 0.1375 & 2.2728 \\
        \checkmark & \checkmark & near & 0.0236 & 1.2972 & \bf 0.0175 & \bf 0.9315 & 0.0306 & 1.8694 & \bf0.0217 & \bf1.2653 \\
        \checkmark & \checkmark & far  & 0.0409 & 1.1640 & \bf 0.0171 & \bf 0.5717 & 0.0616 & 1.4760 & \bf0.0230 & \bf0.6312 \\
        $\times$   & $\times$   & near & \bf 0.2119 & \bf 9.7944 & 0.3506 & 13.3127 & \bf 0.1729 & \bf5.8674 & 0.1747 & 6.0443 \\
        $\times$   & $\times$   & far  & \bf 0.0799 & \bf 3.0978 & 0.1983 & 6.8419 & \bf 0.0899 & \bf 2.6199 & 0.2000 & 6.5192 \\
        \midrule
		\multicolumn{3}{c}{Average}    & \bf 0.0631 & \bf 2.5762 & 0.1386 & 4.9843 & \bf 0.0787 & \bf 2.4640 & 0.1586 & 3.8797 \\
        \bottomrule
    \end{tabular}
    \caption{The RMSE of Relative Pose Error (RPE) for the structure v.s texture sequence. The Trans. columns show the RMSE of the translational drift in $\mathrm{m}/\sec$ and the Rot. columns show the RMSE of the rotational error in $\mathrm{deg}/\sec$. The \checkmark means the sequence has structure/texture and $\times$ means the sequence does not have structure/texture. The results show while DVO performs better in structure and texture case, CVO has significantly better performance in the environments that lack structure and texture.}
    \label{tab:fr3_results}
\end{table}

We also note that DVO has the best performance on the case where the environment contains rich texture and structure information. This could be due to two factors: 1) CVO adopts a semi-dense point cloud construction from DSO~\citep{engel2018direct}, while DVO uses the entire dense image without subsampling. Although the semi-dense tracking approach of~\citet{engel2014lsd,engel2018direct} is computationally attractive and we advocate it, the semi-dense point cloud construction process used in this work is a heuristic process and might not necessarily capture the relevant information in each frame optimally; 2) DVO uses a motion prior as regularizer whereas CVO solely depends on the camera information with no regularizer. We conjecture this latter is the reason DVO, relative to the training set, does not perform well on validation sequences. The motion prior is a useful assumption when it is true. It can help to tune the method better on the training sets but if the assumption gets violated can lead to poor performance. The addition of an IMU sensor of course can improve the performance of both methods and is an interesting future research direction.

\section{Discussions and Further Considerations}
\label{sec:discussion}
In this section, we present discussions on a number of topics that are worth investigating.

\subsection{Locality of Solutions}
\label{subsec:locality}

\begin{figure}
    \centering
    \subfloat{\includegraphics[width=0.3\textwidth]{face_cropped.png}}\hfill
    \subfloat
    {\includegraphics[width=0.33\textwidth]{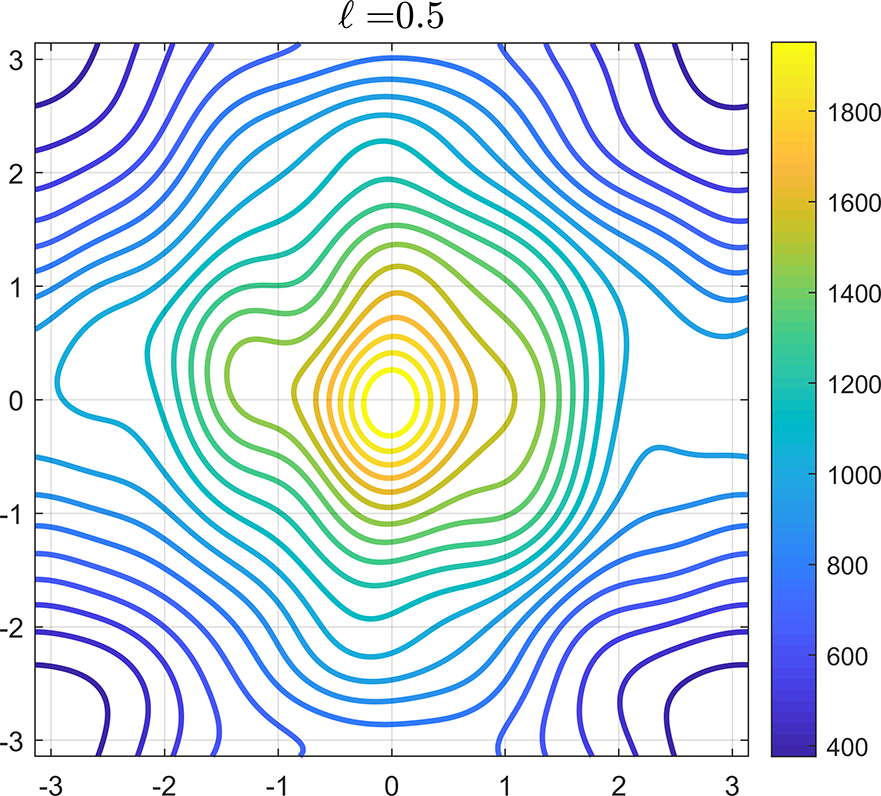}}\hfill
    \subfloat
    {\includegraphics[width=0.33\textwidth]{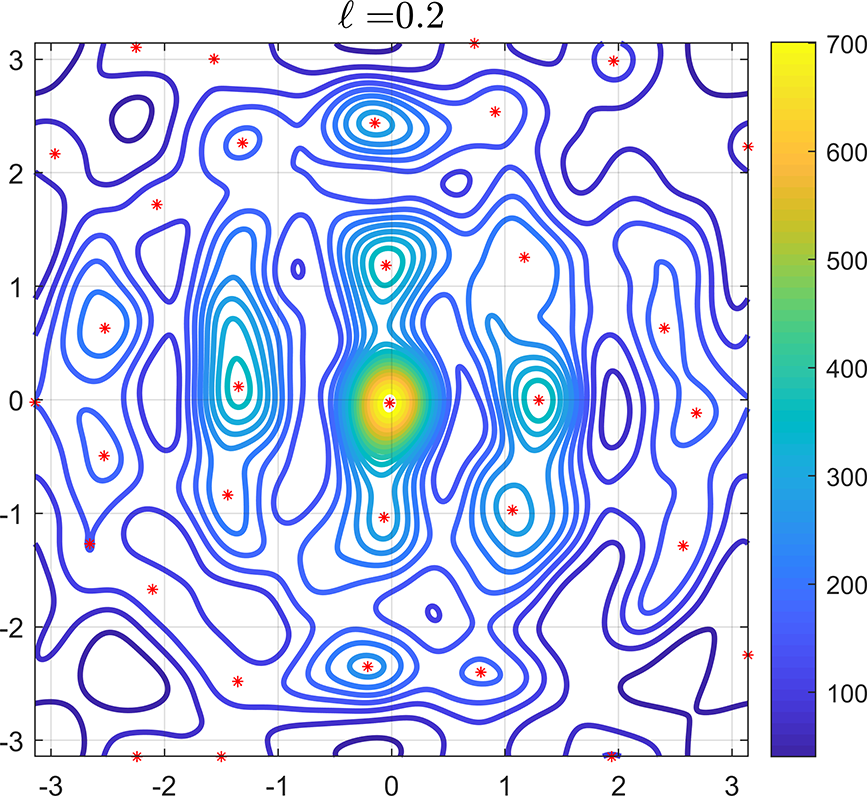}}
    \caption{Left: The images of the two point clouds. The blue stars represent $X$ while the red are $Z$. Center: A contour plot of $F:T^2\to\mathbb{R}$ where $\ell = 0.5$. Right: A contour plot of $F:T^2\to\mathbb{R}$ where $\ell = 0.2$. A total of 33 local maxima are found in this picture. A video of the effect of varying $\ell$ is available at \href{https://youtu.be/ETr6-c0VapQ}{\url{https://youtu.be/ETr6-c0VapQ}}.}
    \label{fig:contours}
\end{figure}

\begin{figure}[t]
    \centering
    \subfloat
    {\includegraphics[width=0.4\columnwidth]{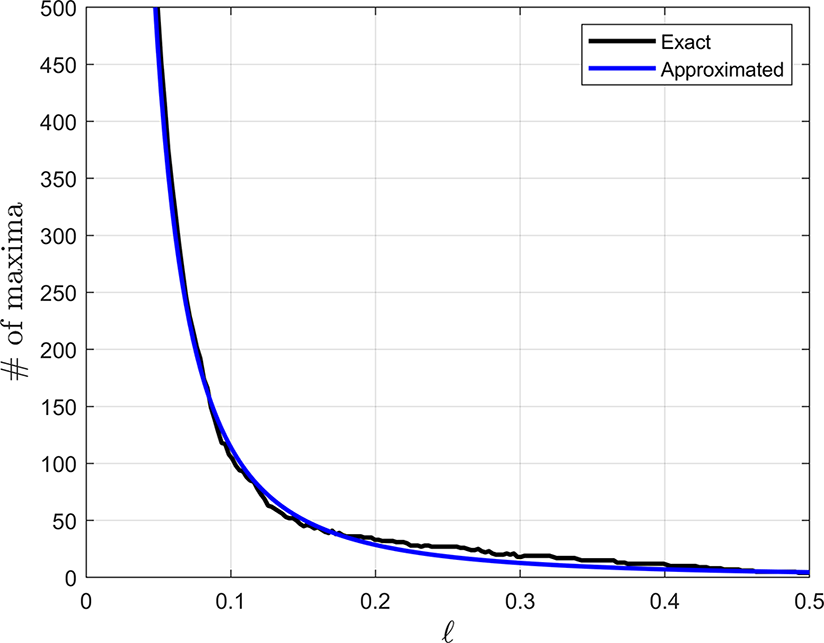}}~
    \subfloat
    {\includegraphics[width=0.4\columnwidth]{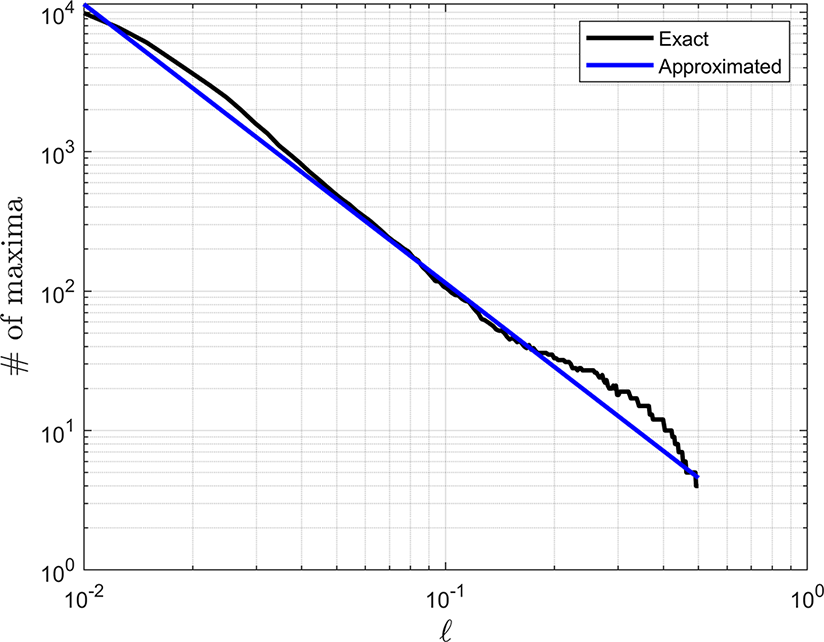}}
    \caption{The images above display the number of local maxima of the function $F:T^2\to\mathbb{R}$ as a function of the length-scale. The black curve shows the exact number of local maxima as reported by Matlab's ``\texttt{imregionalmax}'' while the blue curve is a plot of $N=1.16602\cdot\ell^{-2}$.}
    \label{fig:comparing_number_max}
\end{figure}

It is important to stress that the registration problem laid out here is only a local solver. However, how detrimental is this? For the purposes of gaining intuition, we will compute $F:T^2\to\mathbb{R}$ for the $T^2$ case, see Section \ref{sec:torus}. 
For the point clouds, we will use the same clouds as in Example \ref{ex:se2Hess} (although viewing them in $[0,2\pi)^2\sim T^2$ rather than $\mathbb{R}^2$). The torus is discretized into a 500 by 500 grid and length-scales are chosen in the range $[0.01,0.5]$. For each value of $\ell$, a contour plot of $F:T^2\to\mathbb{R}$ is constructed and the number of local maxima is recorded (with Matlab's ``\texttt{imregionalmax}'' function). Two contour plots corresponding to $\ell=0.5$ and $\ell=0.2$ are shown in Figure \ref{fig:contours}.

It is important to note that there exists a plethora of local maxima. This illustrates how important it is to have a good initialization, or else the gradient flow will end up getting trapped at a false solution. However, the number of local solutions depends heavily on $\ell$. This can be used in practice by having $\ell$ start off large and shrink between iterations, even though care must be taken if one deals with positive-definite kernel design and hyperparameter learning on curved spaces as the Gaussian kernel might not be positive-definite for all $\ell$~\citep{feragen2015geodesic,borovitskiy2020mat}. 

Another observation is that the number of local solutions seems to grow according to a power law of the length-scale, as shown in Figure \ref{fig:comparing_number_max}. Controlling the number of modes is an important problem because if we can control the number of modes, we can initially generate a convex problem and guide the algorithm to the global solution. This approach can be seen as the continuous analogue of the image pyramid method in computer vision~\citep{szeliski2010computer} where low-resolution images are processed first, and the next level has a higher resolution initialized by the previous step.

A heuristic explanation can be that we have more modes of the Gaussian kernel at lower bandwidths since a higher bandwidth makes the bumps around points coalesce. Assuming an inversely proportional relation to the bandwidth, then Figure \ref{fig:comparing_number_max} is a reasonable verification of the number of maxima. Nevertheless, a systematic approach for learning the bandwidth and controlling the complexity of the problem remains an open problem and is an interesting future study.

\subsection{Reduction to the Weighted Least Squares}
\label{subsec:wls}
The standard solution to sensor registration is typically formulated as a maximum likelihood or least-squares problem and solved via numerical optimization techniques. In the maximum likelihood setting, the measurement noise distribution significantly affects the robustness of the solver to outliers. Hence, in practice, the problem is necessarily solved using a robust estimator such as the iteratively reweighted least-squares algorithm or truncated least squares. The central shortcoming of this approach is the reliance on point-wise correspondences between two measurement sets. As a result, the implicit assumption is that the perceived data points by digital sensors are from the same points in reality. Of course, this is seldom true as the nature of digital sensing means a discretization of the reality is captured at best. Sparse and uniform (feature-scarce) measurement sets are where this shortcoming can become problematic. A prime example is the lack of robustness of the visual odometry algorithms to texture and structure-less scenes.

Without loss of generality, we will use the Gaussian (Squared Exponential) kernel based on a distance function, $d(\cdot,\cdot)$:
\begin{equation*}
    k(x,y) = \sigma^2\exp\left(-d(x,y)^2\right).
\end{equation*}
Further, let us define $w_{ij} := \sigma^2 \langle \ell_X(x_i),\ell_Z(z_j)\rangle_{\mathcal{I}}$ and $d_{ij} := d(x_i,z_j)$. Then we have
\begin{equation}\label{eq:scalar2}
\langle f_X,f_Z\rangle_{\Hcal} = \sum_{\substack{x_i\in X\\ z_j\in Z}} \, w_{ij} \cdot \exp\left(-d_{ij}^2\right).
\end{equation}

Now we explicitly apply the three assumptions that are used in the least squares problem.  
\begin{assumption}
The assumptions used in the least squares problem are:
\begin{enumerate}
    \item The point clouds $X$ and $Z$ have the same number of points.
    \item The point cloud $Z$ is ordered such that the matching indices are corresponding measurements.
    \item The residuals are only computed between corresponding measurements.  
\end{enumerate}
\end{assumption}

By applying the above-mentioned assumptions to~\eqref{eq:scalar2}, we get
\begin{equation*}\label{eq:scalar3}
\langle f_X,f_Z\rangle_{\Hcal} \approx \sum_{i} \, w_{i} \cdot \exp\left(-d_{i}^2\right),
\end{equation*}
where $w_{i} = w_{ii}$ and $d_{i} = d_{ii}$. and it is reduced to a single sum over matching indices of the double sum in~\eqref{eq:scalar2}. Finally, the $\exp$ can be approximated using a Taylor expansion as
\begin{equation*}
    \exp\left(-d_{i}^2\right) \approx 1 - d_i^2 + \cdots.
\end{equation*}

\begin{problem}\label{prob:problem_wls}
	The problem of aligning the point clouds can now be rephrased as the following weighted least squares form
	\begin{equation*}\label{eq:min_wls}
		\arg\min_{h\in \Gcal} \, S(h),\quad S(h):= \sum_{i} \, w_{i} \cdot d(x_i,h z_i)^2.
	\end{equation*}
\end{problem}

\subsection{Feature Embedding via Tensor Product Representation}
\label{subsec:encode_feature_space}

For completeness, we include an extension of the proposed framework developed by~\citet{zhang2020new} that is of particular interest to practitioners. The framework contains a specialized solver implemented using GPU that exploits the double sum's sparse structure in the objective function. In addition, it incorporates deep learning-based features such as semantic segmentation labels by extending the feature space to a hierarchical distributed representation. For extensive experimental results and open source software that validates presented topics in this work, please see~\citet{zhang2020new}.

Let $(V_1, V_2, \dots)$ be different inner product spaces describing different types of non geometric features of a point, such as color, intensity, and semantics. Their tensor product, $V_1 \otimes V_2 \otimes \dots $ is also an inner product space. For any   $x\in X, z\in Z$ with features $\ell_X(x)=(u_1, u_2, \dots)$ and $\ell_Z(z)=(v_1, v_2, \dots)$, with $u_1, v_1 \in V_1$,  $u_2, v_2\in V_2$, $\dots$, we have
\begin{align}
    \label{eq:tensor_product}
    \langle \ell_X(x), \ell_Z(z) \rangle_{\mathcal{I}} = \langle u_1 \otimes u_2 \otimes \dots , v_1 \otimes v_2 \otimes \dots \rangle
    = \langle u_1 , v_1 \rangle \cdot \langle u_2, v_2  \rangle  \cdot \dots .
\end{align}
By substituting~\eqref{eq:tensor_product} into \eqref{eq:scalar}, we obtain \begin{align}
    \nonumber \langle f_X,f_{h^{-1}Z}\rangle_{\Hcal} = \sum_{\substack{x_i\in X, z_j\in Z}} \langle u_{1i} , v_{1j} \rangle \cdot \langle u_{2i}, v_{2j}  \rangle   \dots  k(x_i, h^{-1}z_j) 
\end{align}
After applying the kernel trick~\cite[]{bishop2006pattern,rasmussen2006gaussian,murphy2012machine} we arrive at
\begin{align}
    \langle f_X,f_{h^{-1}Z}\rangle_{\Hcal} = \sum_{\substack{x_i\in X, z_j\in Z}} \, k(x_i,h^{-1}z_j) \cdot \prod_k k_{V_k} (u_{ki}, v_{kj})
    := \sum_{\substack{x_i\in X, z_j\in Z}} \,k(x_i,h^{-1}z_j) \cdot c_{ij} . \label{eq:double_sum}
\end{align}

Equation~\eqref{eq:double_sum} describes the full geometric and non-geometric relationship between the two point clouds. Each $c_{ij}$ does not depend on the relative transformation; thus, it will be a constant when computing the gradient and the step size for solving registration problems. Interestingly, the double sum in~\eqref{eq:double_sum} is sparse because a point $x_i\in X$ is far away from the majority of the points $z_j \in Z$, either in the spatial (geometry) space or one of the feature (semantic) spaces.

This formulation can be further simplified to a purely geometric model, if we let the label functions $ \ell_X(x_i)= \ell_Z(z_j)=1$. Then  \eqref{eq:double_sum} becomes 
\begin{align}
\label{eq:geometric_only}
\langle f_X,f_{h^{-1}Z}\rangle_{\Hcal}=\sum_{\substack{x_i\in X, z_j\in Z}} \,k(x_i,h^{-1}z_j) .
\end{align}
Through~\eqref{eq:geometric_only}, the proposed method can register point clouds that do not have appearance measurements. It is worth noting that, when choosing the squared exponential kernel, \eqref{eq:geometric_only} has the same formulation as Kernel Correlation~\citep{tsin2004correlation}.

\subsection{Non-Isotropic Kernel}
For the 3D space, the kernel used throughout is
\begin{equation*}
k(x,y) = \sigma^2\exp\left(-\frac{\lVert x-y\rVert^2}{2\ell^2}\right).
\end{equation*}
where $\ell\in\mathbb{R}$. What happens if we change so $\ell\in\mathbb{R}^3$ so the length-scale changes depending on the direction? This makes the kernel into
\begin{equation}
k_\Lambda(x,y) = \sigma^2\exp\left( -\langle x-y,\Lambda(x-y)\rangle\right),
\end{equation}
where $\Lambda$ is positive-definite and symmetric. Going through the same calculations as before, we get that
\begin{equation}
\begin{split}
\omega &= \frac{1}{a^2}\Lambda\sum_{ij} \, c_{ij}k_\Lambda(x_i,h^{-1}z_j) \left(x_i\times(h^{-1}z_j)\right), \\
v &= \frac{1}{b^2}\Lambda \sum_{ij} \, c_{ij}k_\Lambda(x_i,h^{-1}z_j) \left( x_i-h^{-1}z_j\right).
\end{split}
\end{equation}
Note that in the original case,
\begin{equation*}
\Lambda = \begin{bmatrix}
\frac{1}{2\ell^2} & 0 & 0 \\
0 & \frac{1}{2\ell^2} & 0 \\
0 & 0 & \frac{1}{2\ell^2}
\end{bmatrix} .
\end{equation*}

In terms of applications, modeling directional similarities has been successful in the Generalized-ICP framework developed by~\citet{Segal-RSS-09}. In particular, it was shown that in structured environments, e.g., flat walls, ceiling, and floor, defining a norm weighted by local empirical covariances of the two point clouds effectively reduces the pose estimation error. If such a distance metric is used inside the kernel, the kernel can be considered to be anisotropic. In terms of modeling, this is similar to a covariance function with automatic relevance determination~\citep{neal1996bayesian}.

\subsection{Mat\'{e}rn Kernel}
Rather than using the Gaussian (squared exponential) kernel, we will use the Mat\'{e}rn kernel~\citep{rasmussen2006gaussian}. The following derivations are only valid, i.e., positive definite, for all hyperparameters in Euclidean spaces. For the analogue of squared exponential and Mat\'ern kernels on compact Riemannian manifold please see the work of~\citet{borovitskiy2020mat}.

This is an infinite family of kernels. The Mat\'{e}rn Kernel, for points a distance $d$ apart is
\begin{equation}
C_\nu(d) = \sigma^2\frac{2^{1-\nu}}{\Gamma(\nu)}\left(\sqrt{2\nu}\frac{d}{\rho}\right)^\nu K_\nu\left(\sqrt{2\nu}\frac{d}{\rho}\right),
\end{equation}
where $\Gamma$ is the gamma function, $K_\nu$ is the Bessel function of the second kind, and $\rho$ and $\nu$ are positive parameters. We have the following derivative property:
\begin{equation*}
\frac{d}{dt} K_{\nu}(t) = -K_{\nu-1}(t) - \frac{\nu}{t}K_\nu(t).
\end{equation*}
With half-integers, $\nu = p + 1/2$, $p\in\mathbb{N}$
\begin{equation*}
C_{p+1/2}(d) = \sigma^2\exp\left( -\frac{\sqrt{2p+1}d}{\rho}\right)\frac{p!}{(2p)!}
\sum_{i=0}^p \, \frac{(p+i)!}{i!(p-i)!}\left( \frac{2\sqrt{2p+1}d}{\rho}\right)^{p-i}.
\end{equation*}
Differentiating the kernel with respect to $d$, we see that
\begin{equation*}\begin{split}
\frac{\partial}{\partial d}C_\nu(d) &= \sigma^2\frac{2^{1-\nu}}{\Gamma(\nu)} \left[
\nu\sqrt{2\nu}\left(\sqrt{2\nu}\frac{d}{\rho}\right)^{\nu-1}K_\nu\left(\sqrt{2\nu}\frac{d}{\rho}\right) \right. \\
&\left.+ \left(\sqrt{2\nu}\frac{d}{\rho}\right)^\nu\frac{\sqrt{2\nu}}{\rho}
\left[ -K_{\nu-1}\left(\sqrt{2\nu}\frac{d}{\rho}\right) - \frac{\nu\rho}{\sqrt{2\nu}d}K_\nu\left(\sqrt{2\nu}\frac{d}{\rho}\right) \right]
\right]
\end{split}\end{equation*}
Working with this does not seem appealing, so let us work with half-integer values instead. Below, the kernels and their corresponding gradients are listed. In what follows, $d_{ij} = \lVert x_i-h^{-1}z_j\rVert$.
\subsubsection{$\nu=1/2$}
\begin{equation}
C_{1/2}(d) = \sigma^2\exp\left(-\frac{d}{\rho}\right)
\end{equation}
\begin{equation}
\begin{split}
\omega &= \frac{1}{a^2\rho} \sum_{ij} c_{ij} \exp\left(-\frac{d_{ij}}{\rho}\right) \left(x_i\times h^{-1}z_j\right), \\
v &= \frac{1}{b^2\rho} \sum_{ij} c_{ij} \exp\left(-\frac{d_{ij}}{\rho}\right) \left(x_i- h^{-1}z_j\right).
\end{split}
\end{equation}
\subsubsection{$\nu=3/2$}
\begin{equation}
C_{3/2}(d) = \sigma^2\left( 1 + \frac{\sqrt{3}d}{\rho}\right)\exp\left(-\frac{\sqrt{3}d}{\rho}\right)
\end{equation}
\begin{equation}
\begin{split}
\omega &= \frac{3}{a^2\rho^2} \sum_{ij} c_{ij} d_{ij}\exp\left(-\frac{\sqrt{3}d_{ij}}{\rho}\right)\left(x_i\times h^{-1}z_j\right), \\
v &= \frac{3}{b^2\rho^2} \sum_{ij} c_{ij} d_{ij}\exp\left(-\frac{\sqrt{3}d_{ij}}{\rho}\right)\left(x_i- h^{-1}z_j\right).
\end{split}
\end{equation}
\subsubsection{$\nu=5/2$}
\begin{equation}
C_{5/2}(d) = \sigma^2\left( 1 + \frac{\sqrt{5}d}{\rho} + \frac{5d^2}{3\rho^2}\right)\exp\left(-\frac{\sqrt{5}d}{\rho}\right)
\end{equation}
\begin{equation}
\begin{split}
\omega &= \frac{5}{3a^2\rho^3} \sum_{ij} c_{ij} d_{ij}\left(\rho+\sqrt{5}d_{ij}\right) \exp\left(-\frac{\sqrt{5}d_{ij}}{\rho}\right) \left(x_i\times h^{-1}z_j\right), \\
v &= \frac{5}{3b^2\rho^3} \sum_{ij} c_{ij} d_{ij}\left(\rho+\sqrt{5}d_{ij}\right) \exp\left(-\frac{\sqrt{5}d_{ij}}{\rho}\right) \left(x_i- h^{-1}z_j\right).
\end{split}
\end{equation}

\subsection{Relationship with Deep Learning}
\label{subsec:reldl}

Point clouds obtained by modern sensors such as RGB-D cameras, stereo cameras, and LIDARs contain up to $300,000$ points per scan at $10-60 \Hz$ and rich color and intensity (reflectivity of a material sensed by an active light beam) measurements besides the geometric information. In addition, deep learning~\citep{lecun2015deep} can provide semantic attributes of the scene as measurements~\citep{long2015fcn,chen2017deeplab,zhu2019improving}.

\emph{Representation learning} provides a way to perform \emph{semi-supervised learning}~\cite[]{chapelle2006semi}, \emph{self-supervised learning}~\cite[]{Dahlkamp-RSS-06,sofman2006improving,doersch2017multi,Pathak_2017_CVPR_Workshops}, and \emph{unsupervised learning}~\cite[]{Goodfellow-et-al-2016}. Although the motivation behind learning a ``good'' representation from data varies across different fields and research communities, its mathematical foundation is moving towards a more unifying direction~\cite[]{bengio2013representation,litjens2017survey} where exploiting well-established classical tools and modern data-driven approaches such as deep learning seem inevitable~\cite[]{Girshick_2014_CVPR}. Perhaps the main undesirable aspect of modern deep learning-based methods is the heavy dependency on large labeled data, i.e., \emph{supervised learning}.

An intrinsic metric available in our framework is the angle, $\theta$, between two functions. This indicator can be computed to track the optimization progress. The cosine of the angle is defined as
\begin{equation}
\label{eq:indicator}
 \cos(\theta) = \dfrac{\langle f_X,f_Z\rangle_{\Hcal}}{\lVert f_X \rVert \cdot \lVert f_Z \rVert} .
\end{equation}
However, calculating $\lVert f_X \rVert$ and $\lVert f_Z \rVert$ is time-consuming as it requires evaluating the double sum for each of the two point clouds. To approximate~\eqref{eq:indicator}, we use the following result. 
\begin{remark}
Suppose $k(x_i,x_j)=\delta_{ij}$ and $c_{ii}=1$, where $\delta_{ij}$ is the Kronecker delta, then $\lVert f_X\rVert = \sqrt{|X|}$.
\end{remark}
\begin{corollary}
Using the previous assumption, we define the following alignment indicator.
\begin{equation}
\label{eq:indicator_actual}
    i_{\theta} := \dfrac{1}{\sqrt{|X|\cdot |Z|}} \sum_{\substack{x_i\in X, z_j\in Z}}  c_{ij} \cdot k(x_i,z_j) .
\end{equation}
\end{corollary}
The behavior of the alignment indicator with respect to the rotation and translation errors is stuied by~\citet{zhang2020new}. 

As a connection with deep neural networks, OverlapNet~\citep{chen2020overlapnet} uses a neural network to predict a similar metric and detect loop closures. The cosine of the angle in~\eqref{eq:indicator} or the indicator in~\eqref{eq:indicator_actual} provide such a metric for \emph{self-supervised} learning while taking into account the semantic information. Given the promising results of the work of~\citet{chen2020overlapnet}, the combination of our metric with deep learning is an interesting future research direction. The expected outcomes will be an scalable diagnostic tool for monitoring the quality of point cloud alignments and representation learning for the model in~\eqref{eq:double_sum}.

A more direct approach will be the development of a differentiable loss functions for deep learning tasks. \citet{zhu2020monocular} developed a new continuous 3D loss function for monocular single-view depth prediction based on our function representation. The proposed loss function addresses the gap between dense image prediction and sparse LIDAR supervision. By simply adding this new loss function to existing network architectures, the accuracy and geometric consistency of depth predictions are improved significantly on all three state-of-the-art baseline networks tested by \citet{zhu2020monocular}.


\section{Conclusion}
\label{sec:conclusion}
We developed a new mathematical framework for sensor registration that enables nonparametric joint semantic/appearance and geometric representation of continuous functions using data. The continuous functions allow the registration to be independent of a specific signal resolution, and the framework is fully analytical with a closed-form derivation of the Riemannian gradient and Hessian. We studied a restriction of the framework where the Lie group acts on functions isometrically. This restriction has a variety of applications, and its high-dimensional derivation for the special Euclidean group acting on the Euclidean space showcases the point cloud registration and bird’s-eye view map registration abilities. We derived low-dimensional cases with numerical examples to show the generality of the proposed framework. A specific implementation of this framework for RGB-D cameras performs well in texture and structure-scarce environments and is inherently robust to the input noise.

The general problem where the Lie group does not act isometrically on the functions has interesting applications such as image registration and nonrigid structure from motion. Although, the Lie algebra of the affine matrix group does not possess any particular structure, i.e., it includes all square matrices, due to its mentioned applications it is an attractive future research direction. The development of a globally optimal solver is, of course, another challenging and attractive future research direction.

The developed framework is a data-driven analytical model that possesses interesting intrinsic mathematical structures simultaneously, i.e., equivariance, and computational structures such as sparsity and parallel processing. Autonomy via computational intelligence is a multifaceted research domain which nicely integrates mathematics, computer science, and engineering and can have enormous impacts on our future and improve our quality of life. We developed a novel framework that can spark new research directions in robotics via sensor registration and mapping applications, computer vision via 3D scene reconstruction applications, medical imaging via nonrigid 3D reconstructions and registration, machine learning via Lie group machine learning, and in general artificial intelligence via an equivariant cognitive model.

\acks{A. Bloch and W. Clark were supported in part by NSF grant DMS-1613819 and AFOSR grant FA 0550-18-0028. A. Bloch was also supported in part by NSF grant DMS-2103026. Toyota Research Institute provided funds to support this work. Funding for M. Ghaffari was in part
provided by NSF Award No. 2118818. The authors thank the anonymous reviewers for many constructive feedback and in-depth review of this work.}

\appendix
\section{Required Mathematical Machinery and Notation}
\label{sec:prelim}

We review some mathematical preliminaries to establish the notation.

\subsection{Matrix Lie Group of Motion in $\mathbb{R}^n$}
 \label{subsec:liesen}
 
The general linear group of degree $n$, denoted \mbox{$\GL_n(\mathbb{R})$}, is the set of all $n\times n$ real invertible matrices, where the group binary operation is the ordinary matrix multiplication. The $n$-dimensional special orthogonal group, denoted  
\begin{equation}
  \nonumber \SO(n) = \{{R}\in \GL_n(\mathbb{R}) |~ {R} {R}^\transpose = {I}_n, \operatorname{det} {R} = +1\},
\end{equation}
is the rotation group on $\mathbb{R}^n$. The $n$-dimensional special Euclidean group, denoted 
\begin{equation}
  \nonumber \SE(n) = \left\{ {h} = \left[\begin{array}{cc} {R} & {T} \\ 0 & 1 \end{array} \right] \in \GL_{n+1}(\mathbb{R}) |~ {R} \in \SO(n), {T} \in \mathbb{R}^n \right\},
\end{equation}
is the group of rigid transformations, i.e., \emph{direct isometries}, on $\mathbb{R}^n$. A transformation such as \mbox{${h} \in \SE(3)$} is the parameter space  in many sensor registration problems which consists of a rotation and a translation components. Let \mbox{$\bar{{h}} \in \SE(3)$} be an estimate of ${h}$. We compute the rotational and translational distances using $\lVert{\log(\bar{{R}} {R}^{\transpose})}\rVert_{\mathrm{F}}$ and $\lVert{\bar{{T}} - \bar{{R}} {R}^{\transpose} {T}} \rVert$, respectively, where $\log(\cdot)$ is the Lie logarithm map which is, here, the matrix logarithm. These definitions are consistent with the transformation distance that can be directly computed using $\lVert{\log (\bar{{h}} {h}^{-1})}\rVert_{\mathrm{F}}$. Here, $\lVert A \rVert^2_{\mathrm{F}} = \tr(A^{\transpose}A)$ is the Frobenius norm.

This paper studies not only matrix Lie groups, but also their actions on manifolds. We have the following definition:
\begin{definition}
Let $\Gcal$ be a group and $X$ a set. A (left) group action of $\Gcal$ on $X$, denoted as $\Gcal\curvearrowright X$, is a group homomorphism $\Gcal\to \emph{Aut}(X)$ (automorphism of $X$). If $X$ is a smooth manifold, the action is smooth if $\Gcal\to \diff(X)$ (diffeomorphism of $X$).
\end{definition}
\begin{remark}
A group action can similarly be viewed as a function $\varphi:\Gcal\times X\to X$ satisfying two conditions (where $\varphi(g,x)$ will be denoted by $g.x$)
\begin{enumerate}
    \item Identity: If $e\in \Gcal$ is the identity element, then $e.x=x$ for all $x\in X$.
    \item Compatibility: $(gh).x = g.(h.x)$ for all $g,h\in \Gcal$ and $x\in X$.
\end{enumerate}
\end{remark}
\begin{remark}
The standard action of $\SE(n)$ on $\mathbb{R}^n$ is given by $(R,T).x = Rx+T$ for $R\in \SO(n)$ and $T\in \mathbb{R}^n$.
\end{remark}
\subsection{Hilbert Space}
Let $V$ be a finite-dimensional vector space over the field of real numbers $\mathbb{R}$. An inner product on $V$ is a function \mbox{$\langle\cdot,\cdot\rangle: V\times V \to \mathbb{R}$} that is bilinear, symmetric, and positive definite. The pair $(V, \langle\cdot,\cdot\rangle)$ is called an inner product space. The inner product induces a norm that measures the magnitude or length of a vector:
	$\lVert {v} \rVert = \sqrt{\langle{v}, {v}\rangle}.$
The norm in turn induces a metric that allows for calculating the distance between two vectors:
	$d({v},{w}) = \lVert {v} - {w} \rVert = \sqrt{\langle{v} - {w}, {v} - {w}\rangle}.$
Such a metric is homogeneous; for $a\in\mathbb{R}$, $d(a {v}, a {w}) = |a| d({v}, {w})$, and translation invariant; $d({v} + {x}, {w} + {x}) = d({v}, {w})$. These properties are inherited from the induced norm. The distance metric is positive definite, symmetric, and satisfies the triangle inequality. In addition to measuring distances, it is important to be able to understand limits. This leads to the definition of a Cauchy sequence and completeness.

\begin{definition}[Cauchy Sequence]
A Cauchy sequence is a sequence $\{x_i\}_{i=1}^{\infty}$ such that for any real number $\epsilon > 0$ there exists a natural number $\bar{n} \in \mathbb{N}$ such that for some distance metric $d(x_n, x_m) < \epsilon$ for all $n,m > \bar{n}$.
\end{definition}

\begin{definition}[Completeness]
A metric space $(M,d)$ is complete if every Cauchy sequence in $M$ converges in $M$, i.e., to a limit that is in $M$.
\end{definition}

Such a metric space contains all its limit points. Note that completeness is with respect to the metric $d$ and not the topology of the space. Now, we can give a definition for a Hilbert space.

\begin{definition}[Hilbert Space]
A Hilbert space, $\Hcal$, is a complete inner product space; that is, any Cauchy sequence in $\Hcal$, using the metric induced by the inner product, converges to an element in $\Hcal$.
\end{definition}

The typical example for a Hilbert space is the space of square-integrable functions on $\mathbb{R}$, i.e., $\mathcal{H} = L^2(\mathbb{R},\mu)$ where $\mu$ is the Lebesgue measure on $\mathbb{R}$. The inner product is defined as:
\begin{equation}
\label{eq:ipl2}
	\langle f, g\rangle_{\Hcal} := \int f({x}) g({x}) \mathrm{d}\mu({x}).
\end{equation}
 Similarly, the induced norm by the inner product is \mbox{$\lVert f \rVert_{\Hcal} = \sqrt{\langle f, f\rangle_{\Hcal}}$}. The Hilbert space of functions can be thought of as an infinite-dimensional counterpart of the finite-dimensional vector spaces discussed earlier. However, the Hilbert spaces of interest in this work will be reproducing kernel Hilbert spaces discussed below.

\subsection{Representation and Reproducing Kernel Hilbert Space}
\label{sec:representation_rkhs}

We now move to a more special type of Hilbert Space called Reproducing Kernel Hilbert Space (RKHS)~\citep{berlinet2004reproducing} which we will use in this work.

\begin{definition}[Kernel]
Let ${x}, {x}' \in \mathcal{X}$ be a pair of inputs for a function $k:\mathcal{X} \times \mathcal{X} \to \mathbb{R}$ known as the kernel. A kernel is symmetric if $k({x}, {x}')=k({x}',{x})$, and is positive definite if for any nonzero $f \in \Hcal~ (\text{or}~L^2(\Xcal,\mu))$: 
\begin{equation}
	\nonumber \int k({x}, {x}') f({x}) f({x}') \mathrm{d}\mu({x}) \mathrm{d}\mu({x}') > 0.
\end{equation}
\end{definition}

\begin{definition}[Reproducing Kernel Hilbert Space]
Let $\Hcal$ be a real-valued Hilbert space on a non-empty set $\Xcal$. A function $k:\Xcal \times \Xcal \to \mathbb{R}$ is a reproducing kernel of the Hilbert space $\Hcal$ iff:
\begin{enumerate}
	\item $\forall {x} \in \Xcal, \quad k(\cdot, {x}) \in \Hcal$,
	\item $\forall {x} \in \Xcal, \quad \forall f \in \Hcal \quad \langle f, k(\cdot, {x}) \rangle = f({x})$.
\end{enumerate}
The Hilbert space $\Hcal$ (RKHS) which possesses a reproducing kernel $k$ is called  a Reproducing Kernel Hilbert Space or a proper Hilbert space.
\end{definition}

The second property is called \emph{the reproducing property}; that is using the inner product of $f$ with $k(\cdot,{x})$, the value of function $f$ is reproduced at point ${x}$. Also, using both conditions we have: $\forall {x}, {z} \in \Xcal$, $k({x},{z}) = \langle k(\cdot, {x}), k(\cdot, {z}) \rangle$.

\begin{lemma}
Any reproducing kernel is a positive definite function~\citep{berlinet2004reproducing}.
\end{lemma}

Finding a reproducing kernel of an RKHS might seem difficult, but fortunately, there is a one-to-one relation between a reproducing kernel and its associated RKHS, and such a reproducing kernel is unique. Therefore, our problem reduces to finding an appropriate kernel.

\begin{theorem}[Moore-Aronszajn Theorem,~\citealp{berlinet2004reproducing}]
\label{th:ma}
Let $k$ be a positive definite function on $\Xcal \times \Xcal$. There exists only one Hilbert space $\Hcal$ of functions on $\Xcal$ with $k$ as reproducing kernel. The subspace $\Hcal_0$ of $\Hcal$ spanned by the function $k(\cdot, {x}), {x} \in \Xcal$ is dense~\footnote{A dense subset of $M$ implies the closure of the subset $X$ equals $M$.} in $\Hcal$ and $\Hcal$ is the set of functions on $\Xcal$ which are point-wise limits of Cauchy sequence in $\Hcal_0$ with the inner product 
\begin{equation}
\label{eq:iprkhs}
	\langle f, g \rangle_{\Hcal_0} = \sum_{i=1}^n \sum_{j=1}^m \alpha_i \beta_j k({z}_j, {x}_i),
\end{equation}
where $f =  \sum_{i=1}^n \alpha_i k(\cdot, {x}_i)$ and $g =  \sum_{j=1}^m \beta_j k(\cdot, {z}_j)$.
\end{theorem}

The important property while working in an RKHS is that the convergence in norm implies point-wise convergence; the converse need not be true. In other words, if two functions in an RKHS are close in the norm sense, they are also close point-wise. We will rely on this property to solve the problem discussed in this paper. In Theorem~\ref{th:ma}, $f$ and $g$ are defined only in $\Hcal_0$. The following theorem known as the representer theorem ensures that the solution of minimizing the regularized risk functional admits such a representation.

\begin{theorem}[Nonparametric Representer Theorem,~\citealp{scholkopf2001generalized}]\label{th:representer}
Let $\Xcal$ be a nonempty set and $\Hcal_k$ be an RKHS with reproducing kernel $k$ on $\Xcal \times \Xcal$. Suppose we are given a training sample $({x}_1,y_1),\dots,({x}_m,y_m) \in \Xcal \times \mathbb{R}$, a strictly monotonically increasing real-valued function $h$ on $[0, \infty)$, an arbitrary cost function $c:(\Xcal \times \mathbb{R}^2)^m \to \mathbb{R} \cup \{\infty\}$, and a class of functions~\footnote{$\mathbb{R}^{\Xcal}$ is the space of functions mapping $\Xcal$ to $\mathbb{R}$.}
\begin{align}
\nonumber \Fcal = \{f \in \mathbb{R}^{\Xcal} | f(\cdot) = \sum_{i=1}^\infty \beta_i k(\cdot, {z}_i), \beta_i \in \mathbb{R}, {z}_i \in \Xcal, \lVert f \rVert_{\Hcal_{k}} < \infty \},
\end{align}
where $\lVert \cdot \rVert_{\Hcal_{k}}$ is the induced norm of the RKHS $\Hcal_k$ associated with $k$.
Then any $f \in \Fcal$ minimizing the regularized risk functional 
\begin{equation}
\nonumber c(({x}_1,y_1,f(x_1)),\dots,({x}_m,y_m,f({x}_m))) + h(\lVert f \rVert_{\Hcal_{k}})
\end{equation}
admits a representation of the form
\begin{equation}
	f(\cdot) =  \sum_{i=1}^m \alpha_i k(\cdot, {x}_i).
\end{equation}
\end{theorem}
\subsection{Riemannian Geometry}

Let $M$ be a smooth manifold. This paper will be concerned  with the problem of finding the maximum of a function $f:M\to\mathbb{R}$. This will be accomplished (locally) via gradient ascent. Due to the fact that $M$ will not generally be Euclidean, we will be using tools from Riemannian geometry.

Recall that a vector field $X:M\to TM$ is a section of the tangent bundle, i.e. for $\pi:TM\to M$ we have $\pi\circ X = \mathrm{Id}$. The set of all (smooth) vector fields on $M$ will be denoted by $\mathfrak{X}(M)$. In contrast to vector fields, 1-forms are sections of the cotangent bundle, $M\to T^*M$ and the set of all 1-forms are denoted by $\Omega^1(M)$. 

For a function $f:M\to\mathbb{R}$, the differential is naturally a covector rather than a vector: $df_x:T_xM\to\mathbb{R}$
\begin{equation*}
    df_x(v) = \left.\frac{d}{dt}\right|_{t=0} \, f\left(\gamma(t)\right),\quad \gamma(0)=x,\quad \dot\gamma(0) = v.
\end{equation*}
In order to determine the gradient, $\nabla f\in \mathfrak{X}(M)$, we need an identification $\Omega^1(M)\to \mathfrak{X}(M)$ which is where the Riemannian metric will be used.
\begin{definition}
    A Riemannian metric on a manifold $M$ is the assignment to each point $p\in M$, of an inner product $\langle\cdot,\cdot\rangle_p$ on $T_pM$ which is smooth in the following sense: for $X,Y\in\mathfrak{X}(M)$, the map $p\mapsto \langle X(p),Y(p)\rangle_p$ is a smooth function on $M$.
\end{definition}
This gives a way to turn $df$ into $\nabla f$ as follows:
\begin{equation*}
    df = \langle \nabla f,\cdot \rangle.
\end{equation*}

On a general manifold, if we wish to differentiate a vector field $Y$ in the direction of $X$ (i.e. a directional derivative), we cannot compute it in the traditional sense since different tangent vectors lie in different tangent spaces so we cannot take their differences. A way around this issue is to use a connection (see, e.g., \textsection 6 of \citealt{diffGeometry}):
\begin{definition}
    An affine connection on a manifold, $M$, is a $\mathbb{R}$-bilinear map $\nabla:\mathfrak{X}(M)\times\mathfrak{X}(M)\to\mathfrak{X}(M)$ such that for all $f\in C^\infty(M)$ and $X,Y\in\mathfrak{X}(M)$:
    \begin{itemize}
        \item[i.] $\nabla_{fX}Y = f\nabla_XY$,
        \item[ii.] $\nabla_XfY = \mathcal{L}_Xf\cdot Y + f\nabla_XY$,
    \end{itemize}
    where $\mathcal{L}$ is the usual Lie derivative~\citep[Section 20]{tu2011introduction}.
\end{definition}
For a Riemannian metric there exists a unique affine connection, called the Levi-Civita connection that is both torsion-free and compatible with the metric. A connection is compatible with the metric if for all $X,Y,Z\in\mathfrak{X}(M)$,
\begin{equation*}
    Z\langle X,Y\rangle = \langle \nabla_ZX,Y\rangle + \langle X,\nabla_ZY\rangle,
\end{equation*}
and torsion-free if
\begin{equation*}
    \nabla_XY-\nabla_YX = [X,Y].
\end{equation*}
This connection can be constructed in the following way:
\begin{equation}\label{eq:LC_connection}
    \begin{split}
        2\langle \nabla_XY,Z\rangle &= 
        X\langle Y,Z\rangle + Y\langle Z,X\rangle - Z\langle X,Y\rangle \\
        &\quad - \langle X,[Y,Z]\rangle + \langle Y,[Z,X]\rangle + \langle Z,[X,Y]\rangle.
    \end{split}
\end{equation}
\begin{remark}
    Suppose that $\Gcal$ is a commutative Lie group and we choose an orthogonal basis, $\{e_i\}$, of its Lie algebra $\mathfrak{g}$. If we let $\{E_i\}$ be the corresponding left-invariant vector fields, then the connection vanishes in these coordinates: $\nabla_{E_i}E_j=0$ for all $i,j$. This will be useful in Sections \ref{sec:circle_example} and \ref{sec:torus} where we compute the registration problem for $\Gcal=S^1$ and $\Gcal = T^2$.
\end{remark}
The Levi-Civita connection is used to define a geodesic on a Riemannian manifold, i.e. a ``straight line.''
\begin{definition}
    A geodesic is a curve $\gamma:I\to M$ where $I$ is an interval in $\mathbb{R}$ and
    \begin{equation*}
        \nabla_{\dot\gamma}\dot\gamma = 0,
    \end{equation*}
    where $\nabla_{\dot\gamma}\dot\gamma$ is the covariant derivative of the velocity vector field $\dot\gamma(t) = \frac{d}{dt} \gamma(t)$.
\end{definition}
\begin{theorem}
Let $\gamma:I\to M$ be a curve. Then the following are equivalent:
\begin{enumerate}
    \item $\gamma$ is a geodesic,
    \item In coordinates $\gamma = (y^1,\ldots,y^n)$, the curve satisifies the following system of differential equations
    $$\ddot{y}^k + \sum_{ij} \Gamma^k_{ij}\dot{y}^i\dot{y}^j = 0,$$
    where $\Gamma_{ij}^k$ are the Christoffel symbols~\citep[Sec. 13.3, p. 99]{diffGeometry}, and
    \item $\gamma$ satisfies the Euler-Lagrange equations with Lagrangian 
    $L(y,\dot{y}) = \frac{1}{2}\langle \dot{y},\dot{y}\rangle_y.$
\end{enumerate}
\end{theorem}
\begin{example}
    Consider the case where $M=\mathbb{R}^2$ and the metric is given by $\langle u,v\rangle_p = u^\transpose m(p)v$ for a symmetric, positive-definite matrix $m$ which depends on the point $p=(x,y)\in\mathbb{R}^2$. The Euler-Lagrange equations are
    \begin{equation*}
        \frac{d}{dt}\frac{\partial L}{\partial \dot{x}} - \frac{\partial L}{\partial x} = 0, \quad 
        \frac{d}{dt}\frac{\partial L}{\partial \dot{y}} - \frac{\partial L}{\partial y} = 0.
    \end{equation*}
    Computing the first equation results with:
    \begin{equation*}
        \begin{split}
            \frac{d}{dt}\frac{\partial L}{\partial\dot{x}}
            &= \frac{d}{dt}\left(m_{11}\dot{x}+m_{12}\dot{y}\right) 
            = m_{11}\ddot{x}+m_{12}\ddot{y} + \frac{\partial m_{11}}{\partial x}\dot{x}^2 + \frac{\partial m_{12}}{\partial y}\dot{y}^2
            + \left(\frac{\partial m_{11}}{\partial y} + \frac{\partial m_{12}}{\partial x}\right)\dot{x}\dot{y},\\
            \frac{\partial L}{\partial x} &= 
            \frac{1}{2}\frac{\partial m_{11}}{\partial x} \dot{x}^2 + \frac{\partial m_{12}}{\partial x}\dot{x}\dot{y} + \frac{1}{2}\frac{\partial m_{22}}{\partial x} \dot{y}^2.
        \end{split}
    \end{equation*}
    Combining these (as well as performing the analogous computation for the $y$ equation) we obtain the Euler-Lagrange equations.
    \begin{equation*}
        \begin{split}
            & m_{11}\ddot{x} + m_{12}\ddot{y} + \frac{1}{2}\frac{\partial m_{11}}{\partial x}\dot{x}^2 + \left( \frac{\partial m_{12}}{\partial y} - \frac{1}{2}\frac{\partial m_{22}}{\partial x}\right)\dot{y}^2 
            + \frac{\partial m_{11}}{\partial y}\dot{x}\dot{y} = 0, \\
            & m_{12}\ddot{x} + m_{22}\ddot{y} + \left(\frac{\partial m_{12}}{\partial x} - \frac{1}{2}\frac{\partial m_{11}}{\partial y}\right) \dot{x}^2 + \frac{1}{2}\frac{\partial m_{22}}{\partial y} \dot{y}^2 + \frac{\partial m_{22}}{\partial x}\dot{x}\dot{y} = 0.
        \end{split}
    \end{equation*}
    Decoupling the acceleration terms produces
    \begin{equation*}
        \begin{split}
            \ddot{x} + \Gamma_{xx}^x\dot{x}^2 + \Gamma_{xy}^x\dot{x}\dot{y} + \Gamma_{yy}^x \dot{y}^2 &= 0, \\
            \ddot{y} + \Gamma_{xx}^y\dot{x}^2 + \Gamma_{xy}^y\dot{x}\dot{y} + \Gamma_{yy}^y \dot{y}^2 &= 0,
        \end{split}
    \end{equation*}
    where
    \begin{equation*}
        \begin{split}
            \Gamma_{xx}^x &= \frac{1}{\det(m)} \left(
            \frac{1}{2}m_{22}\frac{\partial m_{11}}{\partial x} - m_{12}\frac{\partial m_{12}}{\partial x} + \frac{1}{2}m_{12}\frac{\partial m_{11}}{\partial y}\right), \\
            \Gamma_{xy}^x &= \frac{1}{\det(m)} \left(
            m_{22}\frac{\partial m_{11}}{\partial y} - m_{12}\frac{\partial m_{22}}{\partial x} \right), \\
            \Gamma_{yy}^x &= \frac{1}{\det(m)} \left(
            m_{22}\frac{\partial m_{12}}{\partial y} - \frac{1}{2}m_{22}\frac{\partial m_{22}}{\partial x} - \frac{1}{2}m_{12}\frac{\partial m_{22}}{\partial y} \right), \\
            \Gamma_{xx}^y &= \frac{1}{\det(m)} \left(
            m_{11}\frac{\partial m_{12}}{\partial x} - \frac{1}{2}m_{12}\frac{\partial m_{11}}{\partial x} - \frac{1}{2}m_{11}\frac{\partial m_{11}}{\partial x}\right), \\
            \Gamma_{xy}^y &= \frac{1}{\det(m)} \left(
            m_{11}\frac{\partial m_{22}}{\partial x} - m_{12}\frac{\partial m_{11}}{\partial y} \right), \\
            \Gamma_{yy}^y &= \frac{1}{\det(m)} \left( \frac{1}{2}m_{12}\frac{\partial m_{22}}{\partial x} - m_{12}\frac{\partial m_{12}}{\partial y} + \frac{1}{2}m_{11}\frac{\partial m_{22}}{\partial y}\right).
        \end{split}
    \end{equation*}
\end{example}
In particular, due to the existence and uniqueness of solutions to ordinary differential equations, geodesics always exist locally. Moreover, when the manifold is a Lie group and the metric is invariant under left-translations, the geodesic equations can be described via the Euler-Poincar\'{e} equations~\citep{Bloch1996}.
\begin{definition}
    Let $\Gcal$ be a Lie group. The function $\ell_g:\Gcal\to\Gcal$, $h\mapsto gh$ for $g\in\Gcal$ is called left-translation by $g$. Its tangent lift, $\left(\ell_g\right)_*:T\Gcal\to T\Gcal$ is defined as follows: let $v\in T_h\Gcal$ and $\gamma:(-\varepsilon:\varepsilon)\to\Gcal$ such that $\gamma(0) = h$ and $\dot{\gamma}(0) = v$. Then, we have
    \begin{equation*}
        \left(\ell_g\right)_*v := \left.\frac{d}{dt}\right|_{t=0} \, \ell_g\circ\gamma(t). 
    \end{equation*}
\end{definition}
\begin{definition}
    For a Lie group, $\mathcal{G}$, a function $L:T\mathcal{G}\to\mathbb{R}$ is said to be a left-invariant Lagrangian if $\left(\ell_g\right)^*L = L$ for all $g \in \mathcal{G}$. That is,
    \begin{equation*}
        L(x,v) = L(gx,\left(\ell_g\right)_*v).
    \end{equation*}
\end{definition}
\begin{theorem}\label{th:EP}
    Let $\mathcal{G}$ be a Lie group and $L:T\mathcal{G}\to\mathbb{R}$ be a left-invariant Lagrangian. Let $\mathcal{L}:\mathfrak{g}\to\mathbb{R}$ be its restriction to the tangent space of $\mathcal{G}$ at the identity (the 
    corresponding Lie algebra). For a curve $g(t)\in \Gcal$, let 
    $$\xi(t) = g(t)^{-1}\cdot \dot{g}(t); \quad
    i.e.,\quad \xi(t) = \left(\ell_{g^{-1}}\right)_*\dot{g}(t).$$
    Then the following are equivalent:
    \begin{enumerate}
        \item $g(t)$ satisfies the Euler-Lagrange equations for $L$ on $\mathcal{G}$,
        \item The Euler-Poincar\'{e} equations hold:
        \begin{equation}\label{eq:eulerPoincare}
            \frac{d}{dt} \frac{\delta \mathcal{L}}{\delta\xi} = 
            \mathrm{ad}^*_\xi\frac{\delta\mathcal{L}}{\delta\xi}.
        \end{equation}
    \end{enumerate}
\end{theorem}
\begin{proof}
See Theorem 5.2 of \citet{Bloch1996}.
\end{proof}
\begin{remark}
    The reasoning behind using $g^{-1}\dot{g}$ rather than just $\dot{g}$ is because $\dot{g}\in T_g\Gcal$ and $g^{-1}:T_g\Gcal\to T_e\Gcal=\mathfrak{g}$ and therefore $g^{-1}\dot{g}\in\mathfrak{g}$.
\end{remark}
In coordinates, let $\xi = \sum_i \xi^ie_i\in\mathfrak{g}$ where $\{e_i\}$ forms a basis. Then equation \eqref{eq:eulerPoincare} takes the form
\begin{equation}
    \frac{d}{dt} \frac{\partial \mathcal{L}}{\partial\xi^j} = c_{ij}^k \frac{\partial\mathcal{L}}{\partial\xi^k}\xi^i, \quad [e_i,e_j] = \sum_{k} \, c_{ij}^k e_k.
\end{equation}
The idea of a geodesic is useful because it provides a map $T_xM\to M$ in the following way:
\begin{equation*}
    v \mapsto \gamma(1),\quad \gamma(0)=x,\quad \dot{\gamma}(0)=v
\end{equation*}
where $\gamma$ is a geodesic. This map is called the Riemannian exponential map and is denoted by $\mathrm{Exp}_x:T_xM\to M$.
\begin{remark}
    There are two important aspects of the Riemannian exponential that need to be  discussed. (1) In general, $\mathrm{Exp}_x$ is only defined on a neighborhood of $0\in T_xM$ and not the whole tangent space. (2) When $M=\mathcal{G}$ is a Lie group, the Lie exponential map $\exp:\mathfrak{g}\to\mathcal{G}$ will usually not agree with the Riemannian exponential although they have matching domain and codomain.
\end{remark}
The last tool from Riemannian geometry we will use is the Hessian. This is a generalization of the second derivative. For a function $f:M\to\mathbb{R}$, the Hessian is a second-order tensor given by
\begin{equation*}
    H(f) = \sum_{ij} \, \left(
    \frac{\partial^2 f}{\partial x^i\partial x^j} - \sum_{k} \, \Gamma^k_{ij} \frac{\partial f}{\partial x^k} \right) dx^i \otimes dx^j.
\end{equation*}
A different, equivalent, definition of the Hessian is:
\begin{equation*}
    H(f)(X,Y) = \langle \nabla_X \mathrm{grad}(f),Y\rangle = X(Yf)-df(\nabla_X Y).
\end{equation*}
This object can be used to calculate the covariance as well as be an ingredient in Newton's root finding method.

\bibliography{strings-full,ieee-full,references}

\end{document}